\documentclass[reqno]{amsart}

\usepackage[left=26mm,top=30mm,right = 26mm, bottom = 30mm]{geometry}

\usepackage[english]{babel} 
\usepackage[utf8]{inputenc} 
\usepackage[T1]{fontenc}
\usepackage{lmodern} 
\usepackage{thmtools} 
\usepackage{amssymb} 
\usepackage{tikz-cd}
\usepackage{enumitem}
\usepackage[backgroundcolor=blue!5!white]{todonotes}
\usepackage{mathtools} 
\usepackage{hyperref}
\usepackage[capitalize]{cleveref} 
\usepackage{comment} 
\usepackage{mathrsfs}
\usepackage{quiver}
\usepackage{mleftright}
\usepackage[new]{old-arrows} 

\newtheorem{theorem}{Theorem}[section]
\newtheorem{lemma}[theorem]{Lemma}
\newtheorem{proposition}[theorem]{Proposition}
\newtheorem{corollary}[theorem]{Corollary}

\newtheorem{claim}[theorem]{Claim}
\theoremstyle{definition}
\newtheorem{definition}[theorem]{Definition}
\newtheorem{notation}[theorem]{Notation}
\newtheorem{example}[theorem]{Example}

\newtheorem{remark}[theorem]{Remark}

\newenvironment{claimproof}[1][Proof of Claim]{\begin{proof}[#1]}{\end{proof}}

\numberwithin{equation}{section}

\newcommand{\BA}{\mathsf{BA}}
\newcommand{\Pos}{\mathsf{Pos}}
\newcommand{\Ctx}{\mathsf{Ctx}}
\newcommand{\LT}{\mathsf{LT}}
\newcommand{\F}{\mathbb{F}}

\newcommand{\cat}[1]{\mathsf{#1}}

\newcommand{\N}{\mathbb{N}}
\renewcommand{\L}{\mathcal{L}}
\newcommand{\T}{\mathcal{T}}
\renewcommand{\P}{\mathbf{P}}
\newcommand{\Hom}{\mathsf{Hom}}
\newcommand{\Set}{\mathsf{Set}}
\newcommand{\tmn}{\mathbf{1}}

\newcommand{\Free}{\mathrm{Free}}

\newcommand{\QA}{\cat{QA}}
\newcommand{\DoctBA}{\cat{Doct}_{\BA}}
\newcommand{\EA}{{\forall\mkern-2mu\exists}}

\newcommand{\fa}[2]{(\forall{#1})_{#2}\,}
\newcommand{\ex}[2]{(\exists{#1})_{#2}\,}

\newcommand{\ple}[1]{\langle#1\rangle}

\newcommand{\C}{\mathsf{C}}

\newcommand{\m}{\mathfrak{m}}
\newcommand{\n}{\mathfrak{n}}

\newcommand{\id}{\mathrm{id}}
\newcommand{\pr}{\mathrm{pr}}
\newcommand{\op}{^{\mathrm{op}}}

\newcommand{\R}{\mathbf{R}}
\newcommand{\D}{\cat{D}}

\hypersetup{hypertexnames=false} 

\title{Freely adding one layer of quantifiers to a Boolean doctrine}

\hypersetup{
    pdftitle={},
    pdfauthor={},
    pdfmenubar=false,
    pdffitwindow=true,
    pdfstartview=FitH,
    colorlinks=true,
    linkcolor=teal,
    citecolor=magenta,
    urlcolor=cyan
}



\author[M. Abbadini]{Marco Abbadini}
\address{University of Birmingham, School of Computer Science,
B15 2TT Birmingham (UK)}
\email{m.abbadini@bham.ac.uk}

\author[F. Guffanti]{Francesca Guffanti}
\address{University of Luxembourg,
Department of Computer Science,
6, avenue de la Fonte,
L-4364 Esch-Sur-Alzette (Luxembourg)}
\email{francesca.guffanti@uni.lu}

\begin{document}

\begin{abstract}
    \vspace{-0.3em}
    We describe the layer of quantifier alternation depth at most one of the quantifier completion of a Boolean doctrine over a small category.
    This amounts to a doctrinal version of Herbrand's theorem for formulas with quantifier alternation depth at most one modulo a universal theory.
    The resulting construction satisfies a universal property that makes it the free QA-one-step Boolean doctrine.   

    To achieve this version of Herbrand's theorem, we characterize, within the doctrinal setting, the classes $\mathcal{A}$ of quantifier-free formulas for which there is a model $M$ such that $\mathcal{A}$ is precisely the class of formulas whose universal closure is valid in $M$. 
\end{abstract}

\maketitle

\setcounter{tocdepth}{3}
\makeatletter
\def\l@subsection{\@tocline{2}{0pt}{2.5pc}{5pc}{}}
\def\l@subsubsection{\@tocline{2}{0pt}{5pc}{7.5pc}{}}
\makeatother

\vspace{-2em}

\tableofcontents

\section{Introduction}

A common measure of the complexity of a first-order formula is its quantifier alternation depth, which counts the number of blocks of alternating existential and universal quantifiers. 
This gives a stratification $\mathcal{F}_0 \subseteq \mathcal{F}_1 \subseteq \mathcal{F}_2 \subseteq \dots$ of the set $\mathcal{F}$ of all first-order formulas, where $\mathcal{F}_n$ consists of the formulas whose quantifier alternation depth is at most $n$.
We use a combination of this stratification with the algebraic approach to logic.
Among various algebraizations of first-order logic, we choose to work with
Lawvere's hyperdoctrines \cite{Lawvere69,Lawvere70}, of categorical flavor.
In contrast to other algebraic approaches to first-order logic (such as polyadic or cylindric algebras), formulas in hyperdoctrines are indexed by their set of free variables.
This allows to faithfully reflect the fact that quantifiers change the set of free variables.
To cover a range of different logical fragments and rules, many variations of Lawvere's original definition of hyperdoctrines have been introduced.
Among these, we place ourselves in the setting of \emph{first-order Boolean doctrines}, which capture the Boolean setting with quantifiers.\footnote{We do not include equality as part of the language of a first-order Boolean doctrine, and we treat the case with equality separately in \cref{s:equality}.}

In this paper, we address the following question: given $\mathcal{F}_0$, how can one construct the set $\mathcal{F}_1$ obtained by freely adding one layer of quantification?
We rephrase this question in doctrinal terms. Let $\P$ be a Boolean doctrine over a small category. Morally, $\P$ represents the collection of equivalence classes of quantifier-free formulas modulo a universal theory $\T$.
The Boolean doctrine $\P$ admits a quantifier completion $\P\hookrightarrow\P^\EA$, which freely adds quantifiers: in other words, $\P^\EA$ is a first-order Boolean doctrine with an appropriate universal property, and it represents the collection of equivalence classes of formulas modulo $\T$.
Now we can define a subfunctor $\P_1^\EA$ of $\P^\EA$ which represents the collection of equivalence classes of formulas with quantifier alternation depth $\leq 1$ modulo $\T$: simply define $\P_1^\EA$ as the collection of Boolean combinations of quantifications (in $\P^\EA$) of formulas in $\P$.
The question then is:
how can one construct $\P^\EA_1$ in terms of $\P$?

First, we explicitly describe the fibers of $\P^\EA_1$ in terms of $\P$, and this amounts to a doctrinal version of Herbrand's theorem for formulas of quantifier alternation depth $\leq 1$ modulo a universal theory.\footnote{After a presentation of this work at the Logic Colloquium 2025, Joshua Wrigley proved that the restriction to existential formulas of this result continues to hold if one replaces the quantifier completion (which adds both existential and universal quantifiers) of a Boolean doctrine by the existential completion (which only adds the existential quantifier) of a ``bounded distributive lattice''-doctrine \cite{Wrigley2026}.}
As a choice of style, our proof of this doctrinal version of Herbrand's theorem is completely within the formalism of doctrines, without relying on results from the classical approach to logic.%
\footnote{Thanks to the link between the classical first-order calculus and first-order Boolean doctrines, one could derive the doctrinal version of Herbrand's theorem from the classical one.
In this link, there is a small caveat: one has to consider a version of classical first-order logic whose semantics allows the empty structure. In its sequent calculus, this logic makes use of sequents-in-contexts (which the doctrinal approach faithfully reflects), and is folklore among categorical logicians; we refer to \cite[Appendix~A]{AbbadiniGuffanti} for a short overview.
At times, this caveat may be delicate: for example, when empty structures are allowed, Henkin's proof of G\"odel's completeness theorem requires an adjustment (see the discussion preceding \cref{t:extension-to-rich}).}
To provide the mentioned description of the fibers of $\P^\EA_1$, it is enough to characterize, within the doctrinal setting, when a finite conjunction of universal sentences entails a finite disjunction of universal sentences modulo a universal theory $\mathcal{T}$ (\cref{t:description-P1}).
We give an idea of the characterization with an example: given a universal theory $\mathcal{T}$, the condition
\[
    \forall y\,\alpha(y)
    \vdash_\mathcal{T} 
    \forall z\,\beta(z)
\]
(with $\alpha$ and $\beta$ quantifier-free) holds if and only if there are finitely many unary terms $t_1(z), \dots, t_n(z)$ such that
\[
    \bigwedge_{i = 1}^{n}\alpha(t_i(z))
    \vdash_\mathcal{T}
    \beta(z).
\]
The interest of this equivalence lies in the fact that, while the first condition concerns formulas with quantifier alternation depth $\leq 1$ and thus can be written in $\P_1^\EA$, the second condition only concerns quantifier-free formulas, and thus can be written in $\P$.
Thanks to some basic properties of Boolean algebras, we then deduce how to characterize when any Boolean combination of universal sentences entails any other Boolean combination of universal sentences modulo $\mathcal{T}$ (\cref{t:description-P1-forall-exists,r:super-description-congruence-P-forall}).

To achieve these results, which are in \cref{s:Herbrand}, we need a detour about models (\cref{sec:completenes-first-layer}).
This detour, of independent interest, contains the technical core of the paper. 
Its main result  (\cref{t:characterization-ultrafilters}), when instantiated to the syntactic setting, amounts to the following:
given a universal theory $\T$, we characterize the classes $\mathcal{A}$ of quantifier-free formulas for which $\T$ has a model $M$ such that $\mathcal{A}$ is precisely the class of quantifier-free formulas whose universal closure is valid in $M$. 
The proof is inspired by Henkin's proof of the completeness theorem for first-order logic.

Then, in \cref{s:fibers-free,s:qa-free}, we turn the doctrinal version of Herbrand's theorem into the construction via generators and relations of a Boolean doctrine $\Free^\P_1$ (later proved to be isomorphic to $\P_1^\EA$). Without relying on Herbrand's theorem, we prove that the embedding $\P \hookrightarrow \Free^\P_1$ satisfies a universal property that makes it the free QA-one-step Boolean doctrine\footnote{QA-one-step Boolean doctrines were introduced in \cite[Def.~3.11]{AbbadiniGuffanti}.} over $\P$ (\cref{s:univ-prop}).
We conclude the paper by proving that, by Herbrand's theorem, $\P\hookrightarrow\P_1^\EA$ is isomorphic to $\P\hookrightarrow\Free_1^\P$ (\cref{t:section-6}); thus, $\P\hookrightarrow\P_1^\EA$ is the free QA-one-step Boolean doctrine, and the definition of $\Free^\P_1$ provides a construction of $\P^\EA_1$ in terms of $\P$ via generators and relations.

\section{Preliminaries}

\subsection{Preliminaries on doctrines}
Hyperdoctrines were introduced by F.~W.~Lawvere \cite{Lawvere69,Lawvere70} to interpret both syntax and semantics of first-order theories in the same categorical setting. 
In this paper, we consider Boolean doctrines and first-order Boolean doctrines, which are variations of Lawvere's hyperdoctrines; points of departure are, among others, the fact that we impose all axioms of Boolean algebras and that we do not consider the equality predicate. (But we will treat the case of doctrines with equality in \cref{s:equality}.) Boolean doctrines contain enough structure to interpret all logical connectives, while \emph{first-order} Boolean doctrines require further structure allowing to interpret also the quantifiers.

Lawvere's doctrinal setting is amenable to a number of generalizations different from ours; for the interested reader we mention primary doctrines (where one can interpret $\top$ and $\land$) \cite{MaiettiRosolini13b,Pasquali15,EmPaRo20}, existential doctrines ($\top$, $\land$, $\exists$) \cite{MaiettiRosolini13b}, universal doctrines ($\top$, $\land$, $\forall$) \cite{Pasquali15}, elementary doctrines ($\top$, $\land$, $=$) \cite{MaiettiRosolini13,EmPaRo20} and first-order doctrines (all logical connectives with the axioms of Heyting algebras and quantifiers) \cite{EmPaRo20}.

\begin{notation}
    $\N$ denotes the set of natural numbers, including $0$.
\end{notation}

\begin{notation}
    We let $\BA$ denote the category of Boolean algebras and Boolean homomorphisms, and $\Pos$ the category of partially ordered sets and order-preserving functions.
\end{notation}

\begin{notation}
    We let $\tmn_\C$ (or simply $\tmn$) denote the terminal object (when it exists) of a category $\C$.
    For $X \in \C$, we denote by $!_X$ the unique morphism from $X$ to $\tmn$.
\end{notation}

\begin{notation}[Conventions on finite products]\label{r:product-conventions}
    Throughout the paper, we suppress the canonical associativity
    and unit isomorphisms for finite products. In particular, we write
    $S\times\tmn=S=\tmn\times S$.

    Likewise, for a finite-product-preserving functor
    $M\colon\C\to\D$, we write
    \[
        M(X\times Y)=M(X)\times M(Y),
        \qquad
        M(\tmn_\C)=\tmn_\D,
    \]
    using the canonical comparison isomorphisms.
\end{notation}

\begin{definition}[Boolean doctrine]
    Given a category $\C$ with finite products, a \emph{Boolean doctrine over $\C$} is a functor $\P \colon \C\op \to  \BA$.
    The category $\C$ is called the \emph{base category of $\P$}.
    For each $X\in\C$, $\P(X)$ is called a \emph{fiber}. For each morphism $f\colon X'\to X$, the function $\P(f)\colon\P(X)\to\P(X')$ is called the \emph{reindexing along $f$}.
\end{definition}

\begin{definition}[Boolean doctrine morphism]
    A \emph{Boolean doctrine morphism} from $\P\colon\C\op\to \BA$ to $\mathbf{R}\colon \cat{D}\op\to \BA$ is a pair $(M,\m)$ with $M\colon \C\to\cat{D}$ a functor preserving finite products and $\m \colon \P\to \mathbf{R}\circ M\op $ a natural transformation.
    \[
        \begin{tikzcd}
            \C\op && {\cat{D}\op} \\
            \\
            & \BA
            \arrow["M\op", from=1-1, to=1-3]
            \arrow[""{name=0, anchor=center, inner sep=0}, "\P"', from=1-1, to=3-2]
            \arrow[""{name=1, anchor=center, inner sep=0}, "\R", from=1-3, to=3-2]
            \arrow["\m", curve={height=-6pt}, shorten <=8pt, shorten >=8pt, from=0, to=1]
        \end{tikzcd}
    \]

    Given Boolean doctrine morphisms $(M, \m) \colon \P\to\mathbf{R}$ and $(N,\mathfrak n) \colon \R \to \mathbf{S}$,
    \[
        \begin{tikzcd}
    	\C\op && {\cat{D}\op} && {\cat{E}\op} \\
    	\\
    	&& \BA
    	\arrow["M\op", from=1-1, to=1-3]
    	\arrow[""{name=0, anchor=center, inner sep=0}, "\P"', from=1-1, to=3-3]
    	\arrow[""{name=1, anchor=center, inner sep=0}, "\R", from=1-3, to=3-3]
    	\arrow["N\op", from=1-3, to=1-5]
    	\arrow[""{name=2, anchor=center, inner sep=0}, "{\mathbf{S}}", from=1-5, to=3-3]
    	\arrow["\m", curve={height=-6pt}, shorten <=8pt, shorten >=8pt, from=0, to=1]
    	\arrow["\n", curve={height=-6pt}, shorten <=8pt, shorten >=8pt, from=1, to=2]
        \end{tikzcd}
    \]
    their composite $(N, \n) \circ (M, \m) \colon \P \to \mathbf{S}$ is the pair $(N \circ M, \mathfrak{n}\circ\m) \colon \P \to \mathbf{S}$, where $N \circ M$ is the composite of the functors between the base categories, and the component at $X\in \C$ of the natural transformation $\mathfrak{n}\circ\m$ is defined as $(\mathfrak{n}\circ\m)_X = \mathfrak{n}_{M(X)}\circ\m_X$, i.e.\ the composite of the following functions:
    \[
    \P(X)\xrightarrow{\m_X}\mathbf{R}(M(X))\xrightarrow{\mathfrak{n}_{M(X)}}\mathbf{S}(NM(X)).
    \]
\end{definition}

Although 2-categorical aspects of doctrines would be very natural, we omit them for simplicity.

\begin{definition}[First-order Boolean doctrine]\label{d:bool_ex_doc}
    Given a category $\C$ with finite products, a \emph{first-order Boolean doctrine over $\C$} is a functor $\P \colon \C\op \to  \BA$ with the following properties.
    \begin{enumerate}
        \item {(Universal)} \label{i:h3}
        For all $X, Y \in \C$, letting $\pr^{X\times Y}_X \colon X \times Y \to X$ denote the projection onto the first coordinate, the function
        \[
        \P(\pr^{X\times Y}_X) \colon \P(X) \to \P(X \times Y)
        \]
        as an order-preserving map between posets has a right adjoint $\fa{Y}{X}$ (which is not required to be a Boolean homomorphism).
        This means that for every $\beta \in \P(X \times Y)$ there is a (necessarily unique) element $\fa{Y}{X} \beta \in \P(X)$ such that, for every $\alpha \in \P(X)$, 
        \[
            \alpha \leq \fa{Y}{X} \beta \ \text{ in } \P(X) \quad\iff \quad \P(\pr^{X\times Y}_X)(\alpha) \leq \beta \ \text{ in }\P(X \times Y).
        \] 
        
        \item
        (Beck--Chevalley) For every $Y\in\C$ and every morphism $f\colon X'\to X$ in $\C$, the following diagram in $\Pos$ commutes. 
        \[
            \begin{tikzcd}
            {X} & {\P(X\times Y)} & {\P(X)} \\
            X' & {\P(X'\times Y)} & {\P(X')}
            \arrow["{\P(f\times\id_{Y})}", from=1-2, to=2-2, swap]
            \arrow["{\P(f)}", from=1-3, to=2-3]
            \arrow["{\fa{Y}{X'}}"', from=2-2, to=2-3]
            \arrow["{\fa{Y}{X}}", from=1-2, to=1-3]
            \arrow["f", from=2-1, to=1-1]
            \end{tikzcd}
        \]
    \end{enumerate}
\end{definition}

\begin{definition}[First-order Boolean doctrine morphism] \label{d:ABA-morphism}
    A \emph{first-order Boolean doctrine morphism} from $\P\colon\C\op\to \BA$ to $\mathbf{R} \colon \cat{D}\op\to \BA$ is a Boolean doctrine morphism $(M,\m)\colon \P\to\R$ such that, for all $X,Y\in\C$, the following diagram commutes.
     \begin{equation}\label{eq:preserves-forall}
         \begin{tikzcd}
             \P(X\times Y)\arrow[d,"\fa{Y}{X}"']\arrow[r,"\m_{X\times Y}"] & \R(M(X)\times M(Y))\arrow[d,"\fa{M(Y)}{M(X)}"]\\
             \P(X)\arrow[r,"\m_{X}"'] & \R(M(X))
         \end{tikzcd}
     \end{equation}
\end{definition}

\begin{remark}[Universality $\Leftrightarrow$ existentiality]
    In \cref{d:bool_ex_doc}, we required the existence of the right adjoint $\fa{Y}{X} \colon \P(X \times Y)\to\P(X)$ of $\P(\pr^{X\times Y}_X)$ (with the Beck--Chevalley condition) for all $X,Y\in\C$.
    In this Boolean case, this condition is equivalent to the existence of the left adjoint $\ex{Y}{X}$ (with the Beck--Chevalley condition).
    This is true because the two quantifiers are interdefinable: $\forall = \lnot \exists \lnot$ and $\exists = \lnot \forall \lnot$.
\end{remark}

Next, we describe the leading example: the first-order Boolean doctrine that describes a first-order theory.

\begin{example}[Syntactic doctrine]\label{fbf}
    Fix a first-order language $\L = (\mathbb{F},\mathbb{P})$ (without equality) and a theory $\T$ in the language $\L$.
    We define a first-order Boolean doctrine 
    \[
    \LT^{\L,\T} \colon \Ctx_\mathbb{F} \op\to \BA,
    \]
    called the \emph{syntactic doctrine of ($\L$ and) $\T$},
    as follows. ($\LT$ stands for ``Lindenbaum--Tarski algebra''.)
    An object of the base category $\Ctx_\mathbb{F}$ is a finite list of distinct variables (also called \emph{context}), and a morphism from $\vec x=(x_1,\dots, x_n)$ to $\vec y=(y_1,\dots, y_m)$ is an $m$-tuple
    \begin{equation*}
        (t_1(\vec x),\dots,t_m(\vec x))\colon (x_1,\dots, x_n) \to (y_1,\dots, y_m)
    \end{equation*}
    of terms in the context $\vec x$. The composition of morphisms in $\Ctx_\mathbb{F}$ is given by simultaneous substitutions.
    On objects, the functor $\LT^{\L,\T} \colon \Ctx_\mathbb{F}\op\to \BA$ maps a context $\vec x$ to the poset reflection of the preordered set of formulas with at most those free variables, ordered by provable consequence $\vdash_\T$ in $\T$, according to which $\alpha$ is below $\beta$ if and only if the sequent $\alpha \Rightarrow_{\vec x} \beta$ is provable from $\T$; here, the subscript $\vec x$ in the sequent symbol $\Rightarrow$ means that the sequent is considered in the context $\vec x$. (We refer for example to \cite[Appendix~A]{AbbadiniGuffanti} for the rules of the sequent calculus with contexts for classical first-order logic. A consequence of the slight difference between the calculus \emph{with} contexts and the usual calculus \emph{without} contexts is that the sequent $\Rightarrow_{()} \exists x \top$ is not provable in the former, in accordance with admitting the empty set as a possible model.)
    On morphisms, the functor $\LT^{\L,\T}$ maps a morphism $\vec{t}(\vec{x}) \colon \vec{x}\to\vec{y}$ to the substitution $[\vec{t}(\vec{x})/\vec{y}] \colon \LT^{\L,\T}(\vec y) \to \LT^{\L,\T}(\vec x)$. 
    
    The functor $\LT^{\L,\T}$ is a first-order Boolean doctrine.
    In particular, given finite lists of variables $\vec x$ and $\vec y$ with no common variables, and letting $\pr_1$ denote the projection morphism $\vec x\colon(\vec x,\vec y)\to\vec x$, the right adjoint to $\LT^{\L,\T}(\pr_1)\colon\LT^{\L,\T}(\vec x)\to\LT^{\L,\T}(\vec x,\vec y)$ is 
    \[
    \forall y_1\,\dots\,\forall y_m\colon\LT^{\L,\T}(\vec x,\vec y)\to\LT^{\L,\T}(\vec x).
    \]
    
    If no confusion arises, instead of $\LT^{\L,\T}$ we simply write $\LT^{\T}$, omitting the superscript ``$\L$''.
\end{example}

With this example in mind, we suggest thinking of the objects of the base category of a first-order Boolean doctrine as lists of variables, the morphisms as terms, the fibers as sets of formulas, the reindexings as substitutions, the Boolean operations as logical connectives, and the adjunctions between fibers as quantifiers.

\begin{remark}
    The setting of first-order Boolean doctrines encompasses also \emph{many-sorted} first-order theories.
    Indeed, any many-sorted first-order theory gives rise to a syntactic doctrine in an analogous way.
\end{remark}

The following example is useful in defining models of a first-order Boolean doctrine.

\begin{example}[Subset doctrine]
    The \emph{subset doctrine} is the contravariant power set functor $\mathscr{P} \colon \Set\op \to \BA$, which maps a set $X$ to its power set Boolean algebra $\mathscr{P}(X)$, and $f\colon X'\to X$ to the preimage function
    \[
    \mathscr{P}(f) \coloneqq f^{-1}[-] \colon \mathscr{P}(X) \to \mathscr{P}(X').
    \]
    It is a first-order Boolean doctrine. For $X, Y \in \Set$, the right adjoint to $\pr_X^{-1}[-]\colon\mathscr{P}(X)\to\mathscr{P}(X\times Y)$ is
    \begin{align*}
        \fa{Y}{X} \colon \mathscr{P}(X\times Y) & \longrightarrow \mathscr{P}(X)\\
        S & \longmapsto \{x \in X \mid \text{for all }y \in Y, \, (x, y) \in S\},
    \end{align*}
    and the left adjoint $\ex{Y}{X}$ is the direct image function, which maps $S \in \mathscr{P}(X\times Y)$ to $\pr_X[S] \in \mathscr{P}(X)$.
\end{example}

\begin{definition}[Propositional model]\label{d:bool-model}
    A \emph{propositional model of a Boolean doctrine $\P \colon \C\op \to  \BA$} is a Boolean doctrine morphism $(M,\m)\colon \P\to\mathscr{P}$, where $\mathscr{P}$ is the subset doctrine.
    \[
        \begin{tikzcd}
            \C\op && \Set\op \\
            \\
            & \BA
            \arrow["M\op", from=1-1, to=1-3]
            \arrow[""{name=0, anchor=center, inner sep=0}, "\P"', from=1-1, to=3-2]
            \arrow[""{name=1, anchor=center, inner sep=0}, "{\mathscr{P}}", from=1-3, to=3-2]
            \arrow["\m", curve={height=-6pt}, shorten <=7pt, shorten >=7pt, from=0, to=1]
        \end{tikzcd}
    \]
\end{definition}

We recall that a \emph{universal formula} is a formula of the form $\forall x_1\dots\forall x_n\,\alpha(x_1,\dots,x_n,y_1,\dots,y_m)$ with $\alpha$ quantifier-free, that a \emph{universal sentence} is the universal closure of a quantifier-free formula, i.e.\ a universal formula with no free variables, and that a \emph{universal theory} is a theory consisting of universal sentences.
The notion of a propositional model $(M, \m)$ of a Boolean doctrine $\P$ is best understood if one thinks of $\P$ as the collection of quantifier-free formulas modulo a universal theory $\T$, and $(M, \m)$ as a model of $\T$. (Some more details are given in \cref{r:propositional-model-concretely} below.)

We use the terminology ``propositional model'' rather than ``model'', as we reserve the latter term for the appropriate concept in the context of \emph{first-order} Boolean doctrines.

\begin{definition}[Model]\label{d:univ-bool-model}
    A \emph{model of a first-order Boolean doctrine $\P \colon \C\op \to  \BA$} is a first-order Boolean doctrine morphism from $\P$ to the subset doctrine $\mathscr{P}$.
\end{definition}

\begin{remark}\label{r:model-classical}
    In the syntactic context, a model $(M, \m)$ of $\LT^\mathcal{T}$ corresponds precisely to a model of the theory $\mathcal{T}$ in the classical sense.
    The assignment of the functor $M$ on objects encodes the underlying set of the model, the assignment of $M$ on morphisms encodes the interpretation of the function symbols, and the natural transformation $\m$ encodes the interpretation of the relation symbols.
    In detail, the underlying set $\mathbb{M}$ of the model is the value of the functor $M$ at the object $(x)$ (the context with only one variable), the interpretation $\mathbb{I}(f)\colon \mathbb{M}^n\to\mathbb{M}$ of a function symbol $f$ of arity $n$ is the value of the functor $M$ at the morphism $f\colon (x_1,\dots,x_n)\to(y)$, and the interpretation $\mathbb{I}(Q)\subseteq \mathbb{M}^n$ of a relation symbol $Q$ of arity $n$ is $\m_{(x_1,\dots,x_n)}(Q(x_1,\dots,x_n))\in\mathscr{P}(M(x_1,\dots,x_n))$.
    
    Note that, in \cref{d:bool-model,d:univ-bool-model}, the functor $M$ might assign the empty set to some objects of $\C$. In the syntactic context, this means that we allow the empty model.
\end{remark}

\subsection{Quantifier completion of a Boolean doctrine}\label{sec:quantifier compl}

\begin{definition}
     A \emph{quantifier completion of a Boolean doctrine $\P\colon \C\op\to\BA$} is a Boolean doctrine morphism $(I,\mathfrak{i})\colon \P \to \P^\EA$, where $\P^\EA\colon {\C'}\op\to\BA$ is a first-order Boolean doctrine, with the following universal property: for every first-order Boolean doctrine $\R \colon \D \op \to \BA$ and every Boolean doctrine morphism $(M,\m) \colon \P \to \R$, there is a unique first-order Boolean doctrine morphism $(N,\n) \colon \P^\EA \to \R$ such that $(M,\m) = (N,\n) \circ (I, \mathfrak{i})$.
        \begin{equation}\label{d:univ-prof-forall}
        \begin{tikzcd}
             \P\arrow[r,"{(I,\mathfrak{i})}"',swap]\arrow[dr,"{(M,\m)}",swap] & \P^\EA\arrow[d,"{(N,\n)}"',dashed,swap]\\& \R
        \end{tikzcd}
        \end{equation}
\end{definition}

With a slight abuse of notation,
we might use the name quantifier completion to refer simply to the first-order Boolean doctrine $\P^\EA$, instead of the Boolean doctrine morphism $\P\to\P^\EA$.

\begin{theorem}[{\cite[Thm.~5.4]{AbbadiniGuffanti}}]
    Every Boolean doctrine over a small category has a quantifier completion.
\end{theorem}

In \cite[Thm.~5.2 and Thm.~5.4]{AbbadiniGuffanti} we show more: we can suppose that the quantifier completion shares the same base category, that the functor between the base categories is the identity, and that the natural transformation is componentwise injective. I.e., every Boolean doctrine $\P$ over a small category $\C$ has a quantifier completion of the form $(\id_\C,\mathfrak{i})\colon \P \to \P^\EA$ (with $\C$ as the base category of $\P^\EA$) with $\mathfrak{i}$ componentwise injective.

We mention that, in general, the quantifier completion of a Boolean doctrine $\P$ differs from the existential completion (in the sense of \cite{Trotta2020}) of $\P$ as a primary doctrine; for the comparison we refer to \cite[Rem.~5.9]{AbbadiniGuffanti}.

\begin{corollary}[{\cite[Cor.~5.5]{AbbadiniGuffanti}}]
    The forgetful functor from the category of first-order Boolean doctrines over a small category to the category of Boolean doctrines over a small category has a left adjoint.
\end{corollary}

\begin{example}[{\cite[Prop.~5.7 and Cor.~5.8]{AbbadiniGuffanti}}]\label{ex:sanitycheck}
    Let $\mathcal{T}$ be a universal theory in the language $(\F,\mathbb{P})$, let $\LT^\mathcal{T}\colon\Ctx_\F\op\to\BA$ be the syntactic doctrine of $\mathcal{T}$ as in \cref{fbf}, let $\LT^\mathcal{T}_0$ be the subfunctor of $\LT^\mathcal{T}$ consisting of quantifier-free formulas modulo $\T$, and let $i\colon\LT^\mathcal{T}_0\hookrightarrow\LT^\mathcal{T}$ be the componentwise inclusion. Then $(\id_{\Ctx_\F},i)\colon\LT^\mathcal{T}_0\hookrightarrow\LT^\mathcal{T}$ is a quantifier completion of the Boolean doctrine $\LT^\mathcal{T}_0$.

    In the more general case in which $\T$ is an arbitrary first-order theory (not necessarily consisting of universal formulas), the quantifier completion 
    of $\LT^\mathcal{T}_0$
    is isomorphic to the syntactic doctrine $\LT^{\mathcal{U}}$, where $\mathcal{U}$ is the theory in the same language as $\T$ whose axioms are the universal sentences derivable from $\mathcal{T}$.
\end{example}

\begin{remark} \label{r:propositional-model-concretely}
    Let $\T$ be a universal theory, and let $\LT^\T$ and $\LT_0^\T$ be as in \cref{ex:sanitycheck}. 
    By \cref{ex:sanitycheck}, $\LT^\T$ is the quantifier completion of $\LT^\T_0$. Therefore, the propositional models of $\LT^\T_0$ are in one-to-one correspondence with the models of the syntactic doctrine $\LT^\T$, which in turn are in one-to-one correspondence with the models of $\T$ in the classical sense (\cref{r:model-classical}).
\end{remark}

As mentioned above, we invite the reader to think of a Boolean doctrine $\P$ as the collection of quantifier-free formulas modulo a universal theory $\T$, and a propositional model of $\P$ as a classical model of $\T$.

\subsection{Adding constants to a doctrine}

It is common to treat some free variables as new constants added to the language.
In the doctrinal case, this is rendered by the following universal construction.

\begin{remark}[Adding constants to a doctrine]\label{r:const}
    Let $\P \colon \C\op \to \BA$ be a Boolean doctrine, and let $S \in \C$. We review from \cite{GuffConst} the construction that freely adds to $\P$ a constant of type $S$.
    
    Let $\C_S$ be the co-Kleisli category for the reader comonad $S \times - \colon \C \to \C$.
    The category $\C_S$ has the same objects as $\C$. For a pair of objects $X$ and $Y$ in $\C_S$ (equivalently, in $\C$), a morphism $f \colon X \rightsquigarrow Y$ in $\C_S$ is a morphism $f \colon S \times X \to Y$ in $\C$. The composite of $f \colon X \rightsquigarrow Y$ and $g \colon Y \rightsquigarrow Z$ is $g \circ \ple{\pr^{S\times X}_S,f} \colon X \rightsquigarrow Z$:
    \[S \times X \xrightarrow{\ple{\pr^{S\times X}_S,f}} S \times Y \xrightarrow{g} Z.\]
    The identity on the object $X$ in $\C_S$ is the morphism $X \rightsquigarrow X$ corresponding to the projection over $X$ in $\C$:
    \[S \times X \xrightarrow{\pr^{S\times X}_X}X.\]
    We remark that, in $\C_S$, there is a morphism $\tmn \rightsquigarrow S$, which corresponds to $\id_S\colon S\to S$ upon choosing $S$ as a product of $S$ and $\tmn$.
    
    The new doctrine $\P_S \colon {\C_S}\op\to\BA$ is defined as follows:
    \begin{center}
        the reindexing of \
        \begin{tikzcd}
        Y\\
        X\arrow[u,"f"', squiggly]
        \end{tikzcd}
        \ is \ 
        \begin{tikzcd}
        \P(S\times Y)\arrow[d,"{\P(\ple{ \pr^{S\times X}_S,f })}"]\\
        \P(S\times X).
        \end{tikzcd}
    \end{center}
    The Boolean doctrine $\P_S$ comes with a canonical Boolean doctrine morphism $(L_S,\mathfrak{l}_S)\colon \P \to \P_S$. The functor $L_S\colon \C \to \C_S$ maps a morphism $f \colon X \to Y$ to the morphism $ f\circ \pr^{S\times X}_X\colon X \rightsquigarrow Y$.
    The component at an object $X$ of the natural transformation $\mathfrak{l}_S$ is
    \begin{align*}
        (\mathfrak{l}_S)_X \colon \P(X) & \longrightarrow \P_S(X)=\P(S\times X)\\
                                \alpha &\longmapsto      \P(\pr^{S\times X}_X)(\alpha).
    \end{align*}
    The fibers of $\P_S$ inherit the Boolean structure of the fibers of $\P$.
    
    If the starting Boolean doctrine $\P$ has quantifiers, then $\P_S$ also has quantifiers and $(L_S,\mathfrak{l}_S)$ preserves them:
    for $X,Y\in \C_S$, the universal quantifier $((\forall_S)Y)_X \colon \P_S(X\times Y)\to \P_S(X)$ is
    \[
    \fa{Y}{S\times X}\colon \P(S\times X\times Y) \to \P(S\times X).
    \]
    All these facts are proved in \cite[Sec.~5]{GuffConst}.
    
    The Boolean doctrine morphism $(L_S,\mathfrak{l}_S)\colon \P \to \P_S$ and the morphism $\id_S\colon \tmn \rightsquigarrow S$ in $\C_S$ have the following universal property {\cite[Cor.~6.7 and Thm.~6.3(iii, vii)]{GuffConst}}:
    \begin{quote}
        For every Boolean doctrine $\mathbf{R}\colon\cat{D}\op\to\BA$, Boolean doctrine morphism $(M,\m)\colon \P\to \mathbf{R}$ and morphism $c \colon \tmn_\cat{D} \to M(S)$ in $\cat{D}$, there is a unique Boolean doctrine morphism $(N,\n) \colon \P_S \to \mathbf{R}$ such that $ (M,\m)=(N,\n)\circ(L_S,\mathfrak{l}_S)$ and $N(\id_S\colon \tmn \rightsquigarrow S)=c \colon \tmn_\cat{D} \to M(S)$.
    \end{quote}
\end{remark}

\begin{definition}[Propositional model at an object]\label{d:bool-mod-at-S}
    A \emph{propositional model of a Boolean doctrine $\P \colon \C\op \to  \BA$ at an object $S \in \C$} is a triple $(M,\m,s)$ where $(M,\m)\colon \P \to \mathscr{P}$ is a propositional model of $\P$ (i.e.\ a Boolean doctrine morphism from $\P$ to the subset doctrine $\mathscr{P}$) and $s \in M(S)$.
\end{definition}

Roughly speaking, a propositional model of $\P$ at $S$ is a propositional model of $\P$ together with a value assignment of $S$ in the model.
Up to a one-to-one correspondence, this is the same thing as a propositional model of the Boolean doctrine $\P_S$ obtained from $\P$ by adding a constant of type $S$. The following lemma describes the relationship between the interpretation of a formula in $\P_S$ in a propositional model of $\P_S$ and the interpretation of the associated formula in $\P$ in the associated propositional model of $\P$ at $S$.

\begin{lemma}\label{l:model-model-at-S}
    Let $\P \colon \C\op \to \BA$ be a Boolean doctrine, let $S\in\C$, let $\P_S$ be the Boolean doctrine obtained from $\P$ by adding a constant of type $S$, and let $(L_S,\mathfrak{l}_S)\colon \P \to \P_S$ be the canonical Boolean doctrine morphism. Let $(M,\m,s)$ be a propositional model of $\P$ at $S$ and $(N,\n)\colon \P_S\to\mathscr{P}$ the unique propositional model of $\P_S$ such that $ (M,\m)=(N,\n)\circ(L_S,\mathfrak{l}_S) $ and $N(\id_S\colon \tmn \rightsquigarrow S)\colon \{*\} \to M(S)$ is the function that sends $*$ to $s\in M(S)$. Let $Y\in\C$ and $\alpha\in \P(S\times Y)=\P_S(Y)$. Then, for all $y\in M(Y)$,
    \[
    y\in\n_Y(\alpha)\iff (s,y)\in\m_{S\times Y}(\alpha).
    \]
\end{lemma}

\begin{proof}
    Consider the naturality diagram of $\n$ with respect to the morphism $\id_{S\times Y}\colon Y\rightsquigarrow S\times Y$ in $\C_S$:
    \begin{equation}\label{diag:model-model-at-S}
        \begin{tikzcd}
            {\P_S(S\times Y)} & {\mathscr{P}(N(S)\times N(Y))} \\
            {\P_S(Y)} & {\mathscr{P}(N(Y)).}
            \arrow["{\n_{S\times Y}}", from=1-1, to=1-2]
            \arrow["{\n_Y}"', from=2-1, to=2-2]
            \arrow["{\P_S(\id_{S\times Y})}"', from=1-1, to=2-1]
            \arrow["{N(\id_{S\times Y})^{-1}[-]}", from=1-2, to=2-2]
        \end{tikzcd}
    \end{equation}
    Since $L_S$ is the identity on objects, $M$ and $N$ have the same value assignments on objects. The diagram on the left-hand side below commutes in $\C_S$ because the diagram on the right-hand side commutes in $\C$.
    \[
    \begin{tikzcd}[column sep=4.5em,row sep=2.7em]
        & Y &&& {S\times Y} \\
        \tmn & {S\times Y} && S & {S\times S\times Y} \\
        S && Y & S && Y
        \arrow["{\pr^{S \times S \times Y}_2}", squiggly, from=2-2, to=3-1]
        \arrow["{!_{S\times Y}}"',curve={height=15pt}, squiggly, from=1-2, to=2-1]
        \arrow["{\id_{S}}"', squiggly, from=2-1, to=3-1]
        \arrow["{\id_{S\times Y}}"{description}, squiggly, from=1-2, to=2-2]
        \arrow["{\pr^{S \times S \times Y}_{Y}}"', squiggly, from=2-2, to=3-3]
        \arrow["{\pr^{S\times Y}_Y}", curve={height=-25pt}, squiggly, from=1-2, to=3-3]
        \arrow["{\ple{\pr^{S \times Y}_S,\pr^{S \times Y}_S,\pr^{S \times Y}_Y}}"{description}, from=1-5, to=2-5]
        \arrow["{\pr^{S \times S \times Y}_2}", from=2-5, to=3-4]
        \arrow["{\pr^{S \times S \times Y}_Y}"', from=2-5, to=3-6]
        \arrow["{\pr^{S \times Y}_{S}}"', curve={height=20pt},from=1-5, to=2-4]
        \arrow["{\id_S}"', from=2-4, to=3-4]
        \arrow["{\pr^{S \times Y}_{Y}}", curve={height=-35pt}, from=1-5, to=3-6]
    \end{tikzcd}
    \]
    Since $N$ preserves finite products, the function $N(\id_{S\times Y})\colon N(Y)\to N(S)\times N(Y)$ maps an element $y\in N(Y)$ to $(s,y)\in N(S)\times N(Y)$.
    Moreover, we observe that
    \begin{equation} \label{eq:model-model-at-S}
    \P_S(\id_{S\times Y})((\mathfrak{l}_S)_{S\times Y}(\alpha))=\P(\ple{\pr^{S \times Y}_S,\pr^{S \times Y}_S,\pr^{S \times Y}_Y})(\P(\ple{\pr^{S \times S \times Y}_2,\pr^{S \times S \times Y}_{Y}})(\alpha))=\alpha.
    \end{equation}
    By \eqref{diag:model-model-at-S} and \eqref{eq:model-model-at-S} and since $\m=\n\circ\mathfrak{l}_S$, we have the following commuting diagram.
    \[\begin{tikzcd}
        {\P(S\times Y)} \\
        & {\P_S(S\times Y)} & {\mathscr{P}(M(S)\times M(Y))} \\
        & {\P_S(Y)} & {\mathscr{P}(M(Y))}
        \arrow["{\n_{S\times Y}}", from=2-2, to=2-3]
        \arrow["{\n_Y}"', from=3-2, to=3-3]
        \arrow["{\P_S(\id_{S\times Y})}", from=2-2, to=3-2]
        \arrow["{N(\id_{S\times Y})^{-1}[-]}", from=2-3, to=3-3]
        \arrow["{(\mathfrak{l}_S)_{S\times Y}}", from=1-1, to=2-2]
        \arrow["{\id_{\P({S\times Y})}}"', curve={height=12pt}, from=1-1, to=3-2]
        \arrow["{\m_{S\times Y}}", curve={height=-18pt}, from=1-1, to=2-3]
    \end{tikzcd}\]
    Thus,
    \begin{align*}
        y\in\n_Y(\alpha)&\iff y\in\n_Y(\P_S(\id_{S\times Y})((\mathfrak{l}_S)_{S\times Y}(\alpha)))\\
        &\iff y \in N(\id_{S\times Y})^{-1}[\m_{S\times Y}(\alpha)]\\
        &\iff (s,y)\in\m_{S\times Y}(\alpha).\qedhere
    \end{align*}
\end{proof}

\section{Characterization of classes of universally valid formulas} \label{sec:completenes-first-layer}

We recall one of the main goals of the paper.
Let $\P\colon \C\op \to \BA$ be a Boolean doctrine, with $\C$ small, and $(\id_\C,\mathfrak{i})\colon\P\hookrightarrow\P^\EA$ its quantifier completion.
We think of $\P$ as the set of quantifier-free formulas of $\P^\EA$. Let $\P^\EA_1$ be the subfunctor of $\P^\EA$ consisting of the ``formulas of quantifier alternation depth $\leq 1$''. 
Intuitively, $\P^\EA_1$ freely adds one layer of quantifier alternation depth to $\P$.
Our goal is to describe $\P^\EA_1$ explicitly, which we will do in \cref{s:Herbrand} with a doctrinal version of Herbrand's theorem.

For example, given $\alpha(x)$, $\beta(y)$, $\gamma(z)$ and $\delta(w)$ quantifier-free formulas (i.e.\ in the fibers of $\P$), when should $(\forall x\, \alpha(x)) \land (\forall y\, \beta(y))$ be below $(\forall z\, \gamma(z)) \lor (\forall w\, \delta(w))$ in $\P^\EA_1$?
Our approach to the question is via models.
In this light, the answer is: when every model of $\P$ satisfying $\forall x\, \alpha(x)$ and $\forall y\, \beta(y)$ also satisfies $\forall z\, \gamma(z)$ or $\forall w\, \delta(w)$; equivalently, when no model of $\P$ satisfies $\forall x\, \alpha(x)$, $\forall y\, \beta(y)$, $\lnot \forall z\, \gamma(z)$ and $\lnot \forall w\, \delta(w)$.
In this section, we take a detour and introduce the notion of a \emph{universal ultrafilter}, which captures axiomatically the set of universally valid formulas in a certain model.
The answer above can then be formulated as: when no universal ultrafilter contains $\alpha$ and $\beta$ but not $\gamma$ and $\delta$.
In turn, this amounts to an explicit condition in terms of $\P$, $\alpha$, $\beta$, $\gamma$ and $\delta$.

To illustrate the main result of the section, we introduce the following notation.
\begin{notation}\label{n:FM-IM-one-model}
    Let $\P \colon \C\op \to  \BA$ be a Boolean doctrine. For every propositional model $(M,\m)$ of $\P$ we define the family $(F_X^{(M, \m)})_{X \in\C}$ by setting
    \[
        F_X^{(M,\m)} \coloneqq \{\alpha \in \P(X) \mid \text{for all }x \in M(X),\, x \in \m_X(\alpha)\}.
    \]
\end{notation}

Roughly speaking, $(F^{(M,\m)}_X)_{X \in \C}$ consists of all the formulas whose ``universal closure'' is valid in the propositional model $(M,\m)$.

\begin{remark}\label{r:univ-vald-fmlas-one-model}
    We translate \cref{n:FM-IM-one-model} to the classical syntactic setting in light of \cref{r:propositional-model-concretely}.
    Let $\T$ be a universal theory.
    To simplify, we take $\{x_1, x_2, \dots\}$ as the set of variables, and instead of taking contexts as arbitrary tuples of distinct variables, we consider only contexts of the type $(x_1, \dots, x_n)$ for some $n \in \N$.
    For every model $M$ of $\T$, we define the family $(F_n^M)_{n \in \N}$ by setting 
    \[
        F_n^M \coloneqq \{\alpha(x_1, \dots, x_{n}) \text{ quantifier-free} \mid M \vDash \forall x_1 \dots \forall x_n\, \alpha(x_1, \dots, x_n)\}.
    \]
    (For this to be an accurate translation of \cref{n:FM-IM-one-model}, we should quotient the above set by equivalence modulo $\T$; however, our discussion works equally well without quotienting.)
\end{remark}

The end result of this section (\cref{t:characterization-ultrafilters}) is, given a Boolean doctrine $\P \colon \C\op \to \BA$ with $\C$ small, a characterization of the families of the form $(F_X^{(M,\m)})_{X \in \C}$ for  $(M,\m)$ varying among the propositional models of $\P$: as we will prove, these families are captured axiomatically by the notion of a \emph{universal ultrafilter}.

\begin{definition} [Universal ultrafilter]\label{d:uf}
    A \emph{universal ultrafilter for a Boolean doctrine $\P \colon \C\op \to \BA$} is a family $(F_X)_{X \in \C}$, with $F_X \subseteq \P(X)$ for all $X \in \C$, with the following properties.
    \begin{enumerate}
        \item\label{forall-reindex-ultrafilter} 
        For all $f \colon X \to Y$ and $\alpha \in F_Y$, $\P(f)(\alpha) \in F_X$.
        
        \item\label{i:uf-filter} For all $X \in \C$, $F_X$ is a filter of $\P(X)$.
        
        \item\label{i:uf-join}
        For all $\alpha_1 \in \P(X_1)$ and $\alpha_2\in \P(X_2)$, if $\P(\pr^{X_1\times X_2}_{X_1})(\alpha_1)\lor \P(\pr^{X_1\times X_2}_{X_2})(\alpha_2)\in F_{X_1\times X_2}$ then $\alpha_1 \in F_{X_1}$ or $\alpha_2\in F_{X_2}$.
        
        \item \label{i:uf-bot}
        $\bot_{\P(\tmn)}\notin F_\tmn$. 
    \end{enumerate}
\end{definition}

\begin{remark} \label{r:uf-translation-to-classic}
    We translate \cref{d:uf} to the classical syntactic setting.
    A \emph{universal ultrafilter for a universal theory $\mathcal{T}$} is a family $(F_n)_{n \in \N}$, with $F_n$ a set of quantifier-free formulas with $x_1, \dots, x_n$ as (possibly dummy) free variables, with the following properties.
    \begin{enumerate}
        \item For any $n,m \in \N$, $\alpha(x_1, \dots, x_m) \in F_m$ and $m$-tuple $(f_i(x_1, \dots,x_n))_{i = 1, \dots, m}$ of $n$-ary terms, 
        \[
        \alpha(f_1(x_1, \dots, x_n), \dots, f_m(x_1, \dots, x_n)) \in F_n.
        \]
        
        \item For all $n \in \N$, 
        \begin{enumerate}
            \item 
            for all quantifier-free formulas $\alpha(x_1, \dots, x_n), \beta(x_1, \dots, x_n)$, if $\alpha(x_1, \dots, x_n) \vdash_\T \beta(x_1, \dots, x_n)$ and $\alpha(x_1, \dots, x_n) \in F_n$, then $\beta(x_1, \dots, x_n) \in F_n$;
            
            \item 
            for all $\alpha_1(x_1, \dots, x_n), \alpha_2(x_1, \dots, x_n) \in F_n$ we have $\alpha_1(x_1, \dots, x_n) \land \alpha_2(x_1, \dots, x_n) \in F_n$;
            
            \item 
            $\top(x_1, \dots, x_n) \in F_n$ (where $\top(x_1, \dots, x_n)$ is the constant ``true'' with $n$ dummy variables).
        \end{enumerate}
            
        \item 
        For all $n_1,n_2 \in \N$ and all quantifier-free formulas $\alpha_1(x_1, \dots, x_{n_1})$ and $\alpha_2(x_1, \dots, x_{n_2})$, if 
        \[\alpha_1(x_1, \dots, x_{n_1}) \lor \alpha_2(x_{n_1 +1}, \dots, x_{n_1 + n_2}) \in F_{n_1 + n_2},\]
        then $\alpha_1(x_1, \dots, x_{n_1}) \in F_{n_1}$ or $\alpha_2(x_1, \dots, x_{n_2}) \in F_{n_2}$.
        
        \item $\bot \notin F_0$.
    \end{enumerate}
    
    For every model $M$ of $\T$, it is easy to check that the family $(F^M_n)_{n\in\N}$ defined by
    \[
    F^M_n \coloneqq \{ \alpha(x_1, \dots, x_{n}) \text{ quantifier-free}\mid M \vDash \forall x_1 \dots \forall x_n\, \alpha(x_1, \dots, x_n)\}
    \]
    is a universal ultrafilter in the sense above.
\end{remark}

The following is a reason for the name ``universal \emph{ultrafilter}''. A further one will be \cref{r:ultrafilter-as-pair}.

\begin{remark} \label{r:F_tmn-is-uf}
    Let $(F_X)_{X \in \C}$ be a universal ultrafilter for a Boolean doctrine $\P$.
    By \eqref{i:uf-filter}--\eqref{i:uf-bot} in \cref{d:uf}, $F_\tmn$ is a filter of $\P(\tmn)$ whose complement is an ideal, and so it is an ultrafilter of $\P(\tmn)$ (in the classical sense).
\end{remark}

Quite a few pages will be required to prove that universal ultrafilters characterize the families of the form $(F_X^{(M,\m)})_{X \in \C}$ for some propositional model $(M,\m)$ (\cref{t:characterization-ultrafilters}); here we start.

\subsection{Universal filters and ideals}

To study universal ultrafilters, the auxiliary notions of \emph{universal filters} and \emph{universal ideals} will come in handy.
To get an intuition about these, we extend \cref{n:FM-IM-one-model} from a single propositional model to a class of propositional models.

\begin{notation}\label{n:FM-IM}
    Let $\P \colon \C\op \to  \BA$ be a Boolean doctrine.
    For every class $\mathcal{M}$ of propositional models of $\P$ we define two families $(F_X^\mathcal{M})_{X \in \C}$ and $(I_X^\mathcal{M})_{X \in \C}$ by setting, for every $X \in \C$,
    \begin{align*}
        F_X^{\mathcal{M}} & \coloneqq \{\alpha \in \P(X) \mid \text{for all }(M,\m) \in \mathcal{M},\text{ for all }x \in M(X),\, x \in \m_X(\alpha)\},\\
        I_X ^{\mathcal{M}}& \coloneqq \{\alpha \in \P(X) \mid \text{for all }(M,\m) \in \mathcal{M},  \text{ not all }x \in M(X) \text{ satisfy } x \in \m_X(\alpha)\}.
    \end{align*}
\end{notation}
Roughly speaking, 
\begin{itemize}
    \item $F^\mathcal{M}=(F_X^{\mathcal{M}})_{X\in\C}$ consists of all the formulas whose ``universal closure'' is valid in all models in $\mathcal{M}$,
    \item $I^{\mathcal{M}}=(I_X^{\mathcal{M}})_{X\in\C}$ consists of all the formulas whose ``universal closure'' is invalid in all models in $\mathcal{M}$.
\end{itemize}

\begin{remark}\label{r:univ-vald-fmlas}
    We translate \cref{n:FM-IM} to the classical syntactic setting. Let $\T$ be a universal theory.
    For every class $\mathcal{M}$ of models of $\T$ we define the families $(F^\mathcal{M}_n)_{n \in \N}$ and $(I^\mathcal{M}_n)_{n \in \N}$ by setting
    \begin{align*}
        F_n^\mathcal{M} & \coloneqq \{ \alpha(x_1, \dots, x_{n}) \text{ quantifier-free}\mid \text{for all } M \in \mathcal{M},\ M \vDash \forall x_1 \dots \forall x_n\, \alpha(x_1, \dots, x_n)\},\\
        I_n^\mathcal{M} & \coloneqq \{ \alpha(x_1, \dots, x_{n}) \text{ quantifier-free}\mid \text{for all } M \in \mathcal{M},\ M \nvDash  \forall x_1 \dots \forall x_n\, \alpha(x_1, \dots, x_n)\}.
    \end{align*}
\end{remark}

\emph{Universal filters} (\cref{d:universal-filter} below) are meant to characterize the families of the form $F^\mathcal{M}$ for some class $\mathcal{M}$ of propositional models, i.e., the classes of formulas that are universally valid in all members of a certain class of propositional models.
Likewise, \emph{universal ideals} (\cref{d:univ-ideal} below) are meant to characterize the families of the form $I^\mathcal{M}$ for some class $\mathcal{M}$ of propositional models, i.e., the classes of formulas whose universal closure fails in all members of a certain class of propositional models. 
While we use universal filters and ideals to prove the main result of this section, we do not need these two characterizations, which we thus postpone to the appendix (\cref{t:characterization-filters,t:characterization-ideals}).

We start by introducing \emph{universal filters}.
As we have just anticipated, they are meant to characterize the families of the form $F^\mathcal{M}$ for some class $\mathcal{M}$ of propositional models (see \cref{n:FM-IM}).

\begin{definition}[Universal filter] \label{d:universal-filter}
    A \emph{universal filter for a Boolean doctrine $\P \colon \C\op \to \BA$} is a family $(F_X)_{X \in \C}$, with $F_X \subseteq \P(X)$ for each $X \in \C$, with the following properties.
    \begin{enumerate}
        \item\label{forall-reindex} For all $f \colon X \to Y$ and $\alpha \in F_Y$, $\P(f)(\alpha) \in F_X$.
        \item For all $X \in \C$, $F_X$ is a filter of $\P(X)$.
    \end{enumerate}
\end{definition}

\begin{remark} 
    We translate \cref{d:universal-filter} to the classical syntactic setting.
    A \emph{universal filter for a universal theory $\mathcal{T}$} is a family $(F_n)_{n \in \N}$, with $F_n$ a set of quantifier-free formulas with $x_1, \dots, x_n$ as (possibly dummy) free variables, with the following properties.
    \begin{enumerate}
        \item For any $n,m \in \N$, $\alpha(x_1, \dots, x_m) \in F_m$ and $m$-tuple $(f_i(x_1, \dots,x_n))_{i = 1, \dots, m}$ of $n$-ary terms, 
        \[
        \alpha(f_1(x_1, \dots, x_n), \dots, f_m(x_1, \dots, x_n)) \in F_n.
        \]
        \item For all $n \in \N$, 
        \begin{enumerate}
            \item for all quantifier-free formulas $\alpha(x_1, \dots, x_n)$ and $\beta(x_1, \dots, x_n)$, if $\alpha(x_1, \dots, x_n) \in F_n$ and $\alpha(x_1, \dots, x_n) \vdash_\T \beta(x_1, \dots, x_n)$, then $\beta(x_1, \dots, x_n) \in F_n$;
            \item for all $\alpha_1(x_1, \dots, x_n), \alpha_2(x_1, \dots, x_n) \in F_n$ we have $\alpha_1(x_1, \dots, x_n) \land \alpha_2(x_1, \dots, x_n) \in F_n$;
            \item $\top(x_1, \dots, x_n) \in F_n$.
        \end{enumerate}
    \end{enumerate}

    If the family $(F^\mathcal{M}_n)_{n\in\N}$ is defined by a class $\mathcal{M}$ of models as in \cref{r:univ-vald-fmlas}, it is easy to check that $(F^\mathcal{M}_n)_{n\in\N}$ satisfies the conditions above.
\end{remark}

\begin{example}
    Let $\P \colon \C\op \to \BA$ be a Boolean doctrine. The family $(\{\top_{\P(X)}\})_{X \in \C}$ is a universal filter for $\P$.
\end{example}

Now, we introduce \emph{universal ideals}. We recall that they are meant to characterize the families of the form $I^\mathcal{M}$ for some class $\mathcal{M}$ of propositional models (see \cref{n:FM-IM}).

\begin{definition}[Universal ideal] \label{d:univ-ideal}
    A \emph{universal ideal for a Boolean doctrine $\P \colon \C\op \to \BA$} is a family $(I_X)_{X \in \C}$, with $I_X \subseteq \P(X)$ for each $X \in \C$, with the following properties.
    \begin{enumerate}
            \item \label{i:conjunction} 
            For all $m \in \N$, $(f_j \colon X \to Y)_{j = 1, \dots, m}$ and $\alpha \in \P(Y)$, if $\bigwedge_{j = 1}^m \P(f_j)(\alpha) \in I_X$ then $\alpha \in I_Y$.
            
            \item \label{i:downward-closed}
            For all $X\in\C$, $I_X$ is downward closed.
            
            \item \label{i:ideal-closed-binary-joins} 
            For all $\alpha_1 \in I_{X_1}$ and $\alpha_2 \in I_{X_2}$, $\P(\pr^{X_1\times X_2}_{X_1})(\alpha_1) \lor \P(\pr^{X_1\times X_2}_{X_2})(\alpha_2) \in I_{X_1 \times X_2}$.
            
            \item \label{i:ideal-bot-tmn}
            $\bot_{\P(\tmn)} \in I_\tmn$.
        \end{enumerate}
\end{definition}

\begin{remark}
    We translate \cref{d:univ-ideal} to the classical syntactic setting.
    A \emph{universal ideal for a universal theory $\mathcal{T}$} is a family $(I_n)_{n \in \N}$, with $I_n$ a set of quantifier-free formulas with $x_1, \dots, x_n$ as (possibly dummy) free variables, with the following properties.
    \begin{enumerate}
        \item 
        For all $p,q,m \in \N$, every $(m\cdot q)$-tuple $(f_{j,k}(x_1, \dots, x_p))_{j \in \{1, \dots, m\},\, k \in \{1, \dots, q\}}$ of $p$-ary terms and every quantifier-free formula $\alpha(x_1, \dots, x_q)$, if 
        \[
            \bigwedge_{j = 1}^m \alpha(f_{j,1}(x_1, \dots, x_p), \dots, f_{j,q}(x_1, \dots, x_p)) \in I_p,
        \]
        then 
        \[
            \alpha(x_1, \dots, x_q) \in I_q.
        \]
        
        \item 
        For all $n \in \N$, all quantifier-free formulas $\alpha(x_1, \dots, x_n)$ and $\beta(x_1, \dots, x_n)$, if $\beta(x_1, \dots, x_n) \in I_n$ and $\alpha(x_1, \dots, x_n) \vdash_\T \beta(x_1, \dots, x_n)$, then $\alpha(x_1, \dots, x_n) \in I_n$.
        
        \item 
        For all $n_1,n_2 \in \N$, $\alpha_1(x_1, \dots, x_{n_1}) \in I_{n_1}$ and $\alpha_2(x_1, \dots, x_{n_2}) \in I_{n_2}$,  we have
        \[
        \alpha_1(x_1, \dots, x_{n_1}) \lor \alpha_2(x_{n_1 +1}, \dots, x_{n_1 + n_2}) \in I_{n_1 + n_2}.
        \]
        
        \item $\bot \in I_0$.
    \end{enumerate}

    These conditions are satisfied by any family $(I^{\mathcal{M}}_n)_{n\in\N}$ defined by a class $\mathcal{M}$ of models as in \cref{r:univ-vald-fmlas}.
\end{remark}

\begin{remark}\label{r:ultrafilter-as-pair}
    The ultrafilters in a Boolean algebra are precisely those filters whose complement is an ideal.
    Similarly, the universal ultrafilters for a Boolean doctrine $\P$ are precisely the universal filters for $\P$ whose componentwise complement is a universal ideal for $\P$.
\end{remark}

\subsection{Filters and ideals generated}

In this subsection, we investigate universal filters and ideals generated by families of formulas. In particular, we will characterize when the universal filter $F$ generated by a family $A$ intersects the universal ideal $I$ generated by a family $B$. This is important because, roughly speaking, thinking of the formulas in $A$ as formulas whose universal closure is ``true'' and those in $B$ as formulas whose universal closure is ``false'', the condition $F \cap I \neq \varnothing$ amounts to inconsistency.

For a Boolean doctrine $\P \colon \C\op \to \BA$ and a family $A = (A_X)_{X \in \C}$ with $A_X \subseteq \P(X)$, we define the \emph{universal filter generated by $A$} as the smallest universal filter containing $A$; this is well defined since an arbitrary componentwise intersection of universal filters is a universal filter.
The \emph{universal ideal generated by $A$} is defined analogously.

\begin{lemma}[Description of the universal filter and ideal generated by a family]\label{l:desc-filter-ideal-generated}
    Let $\P \colon \C\op \to \BA$ be a Boolean doctrine.
    Let $A = (A_X)_{X \in \C}$ be a family with $A_X \subseteq \P(X)$ for each $X \in \C$.
    \begin{enumerate}
        \item \label{i:descr-filter-generated} The universal filter for $\P$ generated by $A$ is the family $(F_X)_{X \in \C}$ where $F_X$ is the set of $\varphi \in \P(X)$ such that there are $n\in\N$, $Y_1,\dots,Y_n\in\C$, $\alpha_1\in A_{Y_1},\dots,\alpha_n\in A_{Y_n}$ and $(f_i\colon X\to Y_i)_{i= 1, \dots, n}$ such that, in $\P(X)$, 
        \[
        \bigwedge_{i=1}^n\P(f_i)(\alpha_i)\leq\varphi.
        \]

        \item \label{i:desc-generaliz-ideal}
        The universal ideal for $\P$ generated by $A$ is the family $(I_X)_{X \in \C}$ where $I_X$ is the set of $\varphi \in \P(X)$ such that there are $n,m\in\N$, $Y_1,\dots,Y_n\in\C$, $\alpha_1\in A_{Y_1},\dots,\alpha_n\in A_{Y_n}$ and $(f_j\colon\prod_{i=1}^n Y_i\to X)_{j = 1, \dots, m}$ such that, in $\P(\prod_{i=1}^n Y_i)$,
    \[
        \bigwedge_{j = 1}^m\P(f_j)(\varphi)\leq\bigvee_{i=1}^n\P(\pr^{\Pi_l Y_l}_{Y_i})(\alpha_i).
    \]
    \end{enumerate}
\end{lemma}

\begin{proof}
    \eqref{i:descr-filter-generated}. It is easily seen that the family $F = (F_X)_{X \in \C}$ contains $A = (A_X)_{X \in \C}$ and is contained in any universal filter containing $A$. 
    We are left to show that $F$ is a universal filter.
    The family $F$ is closed under reindexings because, if $g \colon Y \to X$ is a morphism in $\C$ and $\bigwedge_{i=1}^n\P(f_i)(\alpha_i)\leq\varphi$, then
    \[
        \bigwedge_{i=1}^n\P(f_i \circ g)(\alpha_i) =\P(g)\mleft(\bigwedge_{i=1}^n\P(f_i)(\alpha_i)\mright) \leq \P(g)(\varphi).
    \]
    It is easy to see that, for each $X \in \C$, $F_X$ is upward closed, is closed under binary meets, and contains $\top_{\P(X)}$ (take $n = 0$).

    \eqref{i:desc-generaliz-ideal}. The family $(I_X)_{X \in \C}$ contains $(A_X)_{X \in \C}$: indeed, for $X \in \C$ and $\varphi \in A_X$, take $n = 1$, $m = 1$ and $f$ the identity on $X$.
    Moreover, it is easily seen that any universal ideal containing $(A_X)_{X \in \C}$ contains $(I_X)_{X \in \C}$.
    
    We are left to show that $(I_X)_{X \in \C}$ is a universal ideal.
    
    Let $(f_j\colon X\to Y)_{j=1,\dots,m}$ and $\alpha\in \P(Y)$ with $\bigwedge_{j = 1}^m \P(f_j)(\alpha) \in I_X$.
    There are $Z_1,\dots,Z_n \in \C$, $\beta_1\in A_{Z_1},\dots,\beta_n\in A_{Z_n}$ and $(g_k\colon \prod_{i=1}^nZ_i\to X)_{k=1,\dots, p}$ such that
    \[
        \bigwedge_{k=1}^p\P(g_k)\mleft(\bigwedge_{j = 1}^m \P(f_j)(\alpha)\mright)\leq\bigvee_{i=1}^n\P(\pr^{\Pi_l Z_l}_{Z_i})(\beta_i).
    \]
    Therefore,
    \[
        \bigwedge_{j \in \{1, \dots, m\},\, k \in \{1, \dots, p\}}\P(f_j \circ g_k)(\alpha)\leq\bigvee_{i=1}^n\P(\pr^{\Pi_l Z_l}_{Z_i})(\beta_i).
    \]
    Thus, by definition of $I$, $\alpha\in I_Y$.
    
    For each $X \in \C$, $I_X$ is clearly downward closed.
    
    Let $\alpha_1\in I_{X_1}$ and $\alpha_2\in I_{X_2}$.
    By definition of $(I_X)_{X \in \C}$, there are objects $(Y_i)_{i=1,\dots,n}$, $(Z_k)_{k=1,\dots,p}$, elements $(\beta_i\in A_{Y_i})_{i=1,\dots,n}$, $(\gamma_k\in A_{Z_k})_{k=1,\dots,p}$ and morphisms $(f_j\colon\prod_{i=1}^nY_i\to X_1)_{j=1,\dots,m}$, $(g_h\colon\prod_{k=1}^pZ_k\to X_2)_{h=1,\dots,q}$ such that
    \[
    \bigwedge_{j=1}^m\P(f_j)(\alpha_1)\leq\bigvee_{i=1}^n\P(\pr^{\Pi_l Y_l}_{Y_i})(\beta_i)\quad\text{and}\quad\bigwedge_{h=1}^q\P(g_h)(\alpha_2)\leq\bigvee_{k=1}^p\P(\pr^{\Pi_l Z_l}_{Z_k})(\gamma_k).
    \]
    Then $\P(\pr^{X_1\times X_2}_{X_1})(\alpha_1)\lor\P(\pr^{X_1\times X_2}_{X_2})(\alpha_2)\in I_{X_1\times X_2}$: indeed, consider the morphisms
    \[(f_j\times g_h\colon(\prod_{i=1}^nY_i)\times(\prod_{k=1}^pZ_k)\to X_1\times X_2)_{j\in\{1,\dots,m\},\,h\in\{1,\dots,q\}}\]
    and compute:
    \begin{align*}
        &\bigwedge_{j=1}^m \bigwedge_{h=1}^q \P(f_j\times g_h)\big(\P(\pr^{X_1\times X_2}_{X_1})(\alpha_1) \lor \P(\pr^{X_1\times X_2}_{X_2})(\alpha_2)\big)\\
        & = \bigwedge_{j=1}^m \bigwedge_{h=1}^q \big(\P(\pr^{\Pi_l Y_l\times \Pi_l Z_l}_{\Pi_l Y_l})(\P(f_j)(\alpha_1)) \lor \P(\pr^{\Pi_l Y_l\times \Pi_l Z_l}_{\Pi_l Z_l})(\P(g_h)(\alpha_2))\big)\\
        & = \P(\pr^{\Pi_l Y_l\times \Pi_l Z_l}_{\Pi_l Y_l})\mleft(\bigwedge_{j=1}^m\P(f_j)(\alpha_1)\mright)\lor\P(\pr^{\Pi_l Y_l\times \Pi_l Z_l}_{\Pi_l Z_l})\mleft(\bigwedge_{h=1}^q\P(g_h)(\alpha_2)\mright)\\
        & \leq \P(\pr^{\Pi_l Y_l\times \Pi_l Z_l}_{\Pi_l Y_l})\mleft(\bigvee_{i=1}^n\P(\pr^{\Pi_l Y_l}_{Y_i})(\beta_i)\mright)\lor\P(\pr^{\Pi_l Y_l\times \Pi_l Z_l}_{\Pi_l Z_l})\mleft(\bigvee_{k=1}^p\P(\pr^{\Pi_l Z_l}_{Z_k})(\gamma_k)\mright)\\
        & = \mleft(\bigvee_{i=1}^n\P(\pr^{\Pi_l Y_l\times \Pi_l Z_l}_{Y_i})(\beta_i)\mright)\lor \mleft(\bigvee_{k=1}^p\P(\pr^{\Pi_l Y_l\times \Pi_l Z_l}_{Z_k})(\gamma_k)\mright).
    \end{align*}
    
    Finally, $\bot_{\P(\tmn)}$ belongs to $I_\tmn$: take $n = 0$, $m=1$ and $f_1=\id_\tmn$.
\end{proof}

\begin{lemma}[When a filter intersects a generated ideal]\label{l:id-gen-intersects-filter}
    Let $F$ be a universal filter for a Boolean doctrine $\P \colon \C\op \to \BA$ and $B = (B_X)_{X \in \C}$ a family with $B_X \subseteq \P(X)$ for all $X \in \C$.
    The following are equivalent.
    \begin{enumerate}
        \item \label{i:intersects}
        $F$ intersects the universal ideal for $\P$ generated by $B$.
        
        \item \label{i:condition-for-intersects}
        There are $m \in \N$, $Z_1, \dots, Z_m \in\C$, $\beta_1 \in B_{Z_1}, \dots, \beta_m \in B_{Z_m}$ such that 
        \[
             \bigvee_{j=1}^m\P(\pr_{Z_j}^{\Pi_{i} Z_i})(\beta_j)\in F_{\Pi_{j} Z_j}.
        \]
    \end{enumerate}
\end{lemma}
\begin{proof}
    \eqref{i:intersects} $\Rightarrow$ \eqref{i:condition-for-intersects}.
    Suppose that the filter $F$ intersects the ideal $I$ generated by $B$, i.e.\ that there are $X \in \C$ and $\varphi \in F_X \cap I_X$.
    By \cref{l:desc-filter-ideal-generated}(\ref{i:desc-generaliz-ideal}), from $\varphi \in I_X$ we deduce the existence of $Z_1,\dots,Z_m\in\C$, $\beta_1\in B_{Z_1},\dots,\beta_m\in B_{Z_m}$ and $(g_k\colon\prod_{j=1}^m Z_j\to X)_{k = 1, \dots, p}$ such that,
    in $\P(\prod_{j=1}^m Z_j)$,
    \[
        \bigwedge_{k = 1}^p\P(g_k)(\varphi)\leq\bigvee_{j=1}^m\P(\pr_{Z_j}^{\Pi_i Z_i})(\beta_j).
    \]
    Since $F$ is closed under reindexings and finite meets and is upward closed, we get $\bigvee_{j=1}^m\P(\pr_{Z_j}^{\Pi_{i} Z_i})(\beta_j)\in F_{\Pi_j Z_j}$, as desired.

     \eqref{i:condition-for-intersects} $\Rightarrow$ \eqref{i:intersects}. This is straightforward from the closure properties of universal ideals.
\end{proof}

\begin{corollary}\label{c:id-gen-intersects-filter}
    Let $F$ be a universal filter for a Boolean doctrine $\P \colon \C\op \to \BA$ and $B = (B_X \subseteq \P(X))_{X \in \C}$ a family such that, for all $m \in \N$, $Z_1, \dots, Z_m \in\C$, $\beta_1 \in B_{Z_1}, \dots, \beta_m \in B_{Z_m}$, we have $\bigvee_{j=1}^m\P(\pr^{\Pi_{i} Z_i}_{Z_j})(\beta_j)\in B_{\Pi_{j} Z_j}$. Then $F$ intersects the universal ideal for $\P$ generated by $B$ if and only if $F$ intersects $B$.
\end{corollary}

\begin{lemma}[When a generated filter intersects a generated ideal]\label{l:intersects}
    Let $\P \colon \C\op \to \BA$ be a Boolean doctrine, and let $A = (A_X)_{X \in \C}$ and $B = (B_X)_{X \in \C}$ be families with $A_X \subseteq \P(X)$ and $B_X \subseteq \P(X)$ for each $X \in \C$.
    The following are equivalent.
    \begin{enumerate}
        \item
        The universal filter for $\P$ generated by $A$ intersects the universal ideal for $\P$ generated by $B$.
        
        \item
        There are $n,m \in \N$, $Y_1, \dots, Y_n, Z_1, \dots, Z_m \in\C$, $\alpha_1 \in A_{Y_1}, \dots, \alpha_n\in A_{Y_n}$, $\beta_1 \in B_{Z_1}, \dots, \beta_m \in B_{Z_m}$, and $(f_{i} \colon \prod_{j = 1}^m {Z_j} \to Y_i)_{i = 1, \dots, n}$ such that (in $\P(\prod_{j = 1}^m Z_j)$)
        \[
             \bigwedge_{i = 1}^n\P(f_{i})(\alpha_i) \leq \bigvee_{j=1}^m\P(\pr_{Z_j}^{\Pi_l Z_l})(\beta_j).
        \]
    \end{enumerate}
\end{lemma}
\begin{proof}
    The universal filter $F$ for $\P$ generated by $A$ intersects the universal ideal for $\P$ generated by $B$ if and only if, by \cref{l:id-gen-intersects-filter}, there are $Z_1, \dots, Z_m \in\C$, $\beta_1 \in B_{Z_1}, \dots, \beta_m \in B_{Z_m}$ such that 
    \begin{equation} \label{e:in-filter}
        \bigvee_{j=1}^m\P(\pr_{Z_j}^{\Pi_l Z_l})(\beta_j)\in F_{\Pi_j Z_j}.
    \end{equation}
    By the description of the universal filter generated by a family (\cref{l:desc-filter-ideal-generated}), \eqref{e:in-filter} holds if and only if there are $Y_1,\dots,Y_n\in\C$, $\alpha_1\in A_{Y_1},\dots,\alpha_n\in A_{Y_n}$ and $(f_i\colon \prod_j Z_j\to Y_i)_{i= 1, \dots, n}$ such that, in $\P(\prod_{j = 1}^m Z_j)$, 
    \[
    \bigwedge_{i=1}^n\P(f_i)(\alpha_i) \leq \bigvee_{j=1}^m\P(\pr_{Z_j}^{\Pi_l Z_l})(\beta_j). \qedhere
    \]
\end{proof}

\begin{lemma}[When a finitely generated filter intersects a finitely generated ideal] \label{l:intersect-multiple}
    Let $\P \colon \C\op \to \BA$ be a Boolean doctrine, ${\bar{i}}, {\bar{j}} \in \N$, $Y_1, \dots, Y_{\bar{i}}, Z_1,\dots, Z_{\bar{j}} \in \C$, $(\alpha_i \in \P(Y_i))_{i = 1, \dots, {\bar{i}}}$, and $(\beta_j \in \P(Z_j))_{j = 1, \dots, {\bar{j}}}$.
    The following are equivalent.
    \begin{enumerate}
        \item \label{i:intersect-multiple-1}
        The universal filter generated by $\alpha_1, \dots, \alpha_{\bar{i}}$ intersects the universal ideal generated by $\beta_1,\dots,\beta_{\bar{j}}$.
        
        \item \label{i:intersect-multiple-2}
        There are $n \in \N$, $l_1, \dots, l_n \in \{1, \dots, {\bar{i}}\}$, and $(g_{i} \colon \prod_{j=1}^{\bar{j}} Z_{j} \to Y_{l_i})_{i = 1, \dots, n}$ such that in $\P(\prod_{j=1}^{\bar{j}} Z_{j})$
        \[
        \bigwedge_{i = 1}^{n}\P(g_i)(\alpha_{l_i}) \leq \bigvee_{j=1}^{\bar{j}}\P(\pr^{\Pi_p Z_p}_{Z_j})(\beta_j).
        \]
    \end{enumerate}
\end{lemma}

\begin{proof}
    By \cref{l:intersects}, the universal filter generated by $\alpha_1, \dots, \alpha_{\bar{i}}$ intersects the universal ideal generated by $\beta_1,\dots,\beta_{\bar{j}}$ if and only if there are $n,m \in \N$, $l_1, \dots, l_n \in \{1, \dots, {\bar{i}}\}$, $k_1, \dots, k_m \in \{1, \dots, {\bar{j}}\}$, and $(f_i \colon \prod_{h = 1}^{m} Z_{k_h} \to Y_{l_i})_{i = 1, \dots, n}$, such that, in $\P(\prod_{h = 1}^{m} Z_{k_h})$,
    \begin{equation} \label{eq:hypothesis-bigwedge-bigvee}
        \bigwedge_{i = 1}^n\P(f_{i})(\alpha_{l_i}) \leq \bigvee_{h=1}^m\P(\pr^{\Pi_{q} Z_{k_q}}_{Z_{k_h}})(\beta_{k_h}).
    \end{equation}

    Therefore, the implication $\eqref{i:intersect-multiple-2}\Rightarrow \eqref{i:intersect-multiple-1}$ holds (take $m = {\bar{j}}$, $k_h=h$ and $g_i = f_i$).
    For the implication $\eqref{i:intersect-multiple-1}\Rightarrow \eqref{i:intersect-multiple-2}$, suppose \eqref{eq:hypothesis-bigwedge-bigvee} holds.
    For each $i \in \{1, \dots, n\}$, set $g_i \colon \prod_{j=1}^{\bar{j}} Z_{j}  \to Y_{l_i}$ as the composite 
    \[
        \prod_{j=1}^{\bar{j}} Z_{j} \xrightarrow{\ple{\pr^{\Pi_p Z_p}_{Z_{k_1}},\dots,\pr^{\Pi_p Z_p}_{Z_{k_m}}}} \prod_{h = 1}^{m} Z_{k_h} \xrightarrow{f_i}  Y_{l_i}.
    \]
    Therefore, we get, in $\P(\prod_{j=1}^{\bar{j}} Z_{j})$,
    \begin{align*}
        \bigwedge_{i = 1}^n\P(g_i)(\alpha_{l_i})&=\bigwedge_{i = 1}^n\P(f_i \circ \ple{\pr^{\Pi_{p} Z_{p}}_{Z_{k_1}},\dots,\pr^{\Pi_{p} Z_{p}}_{Z_{k_m}}})(\alpha_{l_i})\\
        &=\bigwedge_{i = 1}^n\P(\ple{\pr^{\Pi_{p} Z_{p}}_{Z_{k_1}},\dots,\pr^{\Pi_{p} Z_{p}}_{Z_{k_m}}})(\P(f_{i})(\alpha_{l_i}))\\
        &=\P(\ple{\pr^{\Pi_{p} Z_{p}}_{Z_{k_1}},\dots,\pr^{\Pi_{p} Z_{p}}_{Z_{k_m}}})\mleft(\bigwedge_{i = 1}^n\P(f_{i})(\alpha_{l_i})\mright)\\
        & \leq \P(\ple{\pr^{\Pi_{p} Z_{p}}_{Z_{k_1}},\dots,\pr^{\Pi_{p} Z_{p}}_{Z_{k_m}}})\mleft(\bigvee_{h=1}^m\P(\pr^{\Pi_{q} Z_{k_q}}_{Z_{k_h}})(\beta_{k_h})\mright) && \text{by \eqref{eq:hypothesis-bigwedge-bigvee}}
        \\
        & = \bigvee_{h=1}^m\P(\ple{\pr^{\Pi_{p} Z_{p}}_{Z_{k_1}},\dots,\pr^{\Pi_{p} Z_{p}}_{Z_{k_m}}})(\P(\pr^{\Pi_{q} Z_{k_q}}_{Z_{k_h}})(\beta_{k_h}))\\
        & = \bigvee_{h=1}^m\P(\pr^{\Pi_{q} Z_{k_q}}_{Z_{k_h}} \circ {\ple{\pr^{\Pi_{p} Z_{p}}_{Z_{k_1}},\dots,\pr^{\Pi_{p} Z_{p}}_{Z_{k_m}}}})(\beta_{k_h})\\
        & = \bigvee_{h=1}^m\P(\pr^{\Pi_{p} Z_{p}}_{Z_{k_h}})(\beta_{k_h})\\
        & \leq \bigvee_{j=1}^{\bar j} \P(\pr^{\Pi_{p} Z_{p}}_{Z_j})(\beta_j). &&\qedhere
    \end{align*}
\end{proof}

For the reader interested in the translation of the condition \eqref{i:intersect-multiple-2} in \cref{l:intersect-multiple} to the classical syntactic setting, we refer to \eqref{i:explicit-sequent} in \cref{r:sequent-classical-setting} below.

\begin{lemma} \label{l:char-gen-intersects}
    Let $\P \colon \C\op \to \BA$ be a Boolean doctrine, $F = (F_X)_{X \in \C}$ a universal filter for $\P$, $I = (I_X)_{X \in \C}$ a universal ideal for $\P$, $Y \in \C$ and $\alpha \in \P(Y)$.
    \begin{enumerate}
        \item \label{i:intersect-1}
        The universal filter generated by $F$ and $\alpha$ intersects $I$ (in some fiber) if and only if there are $X \in \C$, $n\in\N$, $(f_i \colon X \to Y)_{i= 1, \dots, n}$ and $\beta \in F_X$ such that $\beta \land \bigwedge_{i = 1}^n\P(f_i)(\alpha) \in I_X$.
        
        \item \label{i:intersect-2}
        The universal ideal generated by $I$ and $\alpha$ intersects $F$ (in some fiber) if and only if there is $X \in \C$ with $I_X \cap F_X \neq \varnothing$ or there are $Z\in\C$, $\gamma \in I_Z$ such that $\P(\pr^{Y\times Z}_Y)(\alpha) \lor \P(\pr^{Y\times Z}_Z)(\gamma) \in F_{Y \times Z}$. 
    \end{enumerate}
\end{lemma}
\begin{proof}
    \eqref{i:intersect-1}.
    This is straightforward from \cref{l:intersects} and the closure properties of universal filters and ideals.

    \eqref{i:intersect-2}.
    Let $I'$ be the universal ideal generated by $I$ and $\alpha$.
    By \cref{l:intersects}, $F$ intersects $I'$ if and only if there are $n, m \in \N$, $Y_1, \dots, Y_n, Z_1, \dots, Z_m \in \C$, $(\alpha_i \in F_{Y_i})_{i = 1, \dots, n}$, $(f_i \colon \prod_{j = 1}^m Z_j \to Y_i)_{i = 1, \dots, n}$ and, for each $j = 1, \dots, m$, an element $\beta_j$ which either belongs to $I_{Z_j}$ or is equal to $\alpha$ (in which case $Z_j = Y$), such that, in $\P(\prod_{j = 1}^m Z_j)$,
    \begin{equation} \label{eq:intersect-2-unfolded}
        \bigwedge_{i = 1}^{n}\P(f_i)(\alpha_i) \leq \bigvee_{j = 1}^{m}\P(\pr^{\Pi_l Z_l}_{Z_j})(\beta_j).
    \end{equation}
    Suppose \eqref{eq:intersect-2-unfolded} holds.
    By the closure properties of $F$, the right-hand side of \eqref{eq:intersect-2-unfolded} belongs to $F_{\Pi_l Z_l}$.
    Let $J$ be the set of indices $j$ such that $\beta_j=\alpha$, and let $K \coloneqq \{1, \dots, m\}\setminus J$.
        We set $Z \coloneqq \prod_{j \in K} Z_j$ and $\gamma \coloneqq \bigvee_{j \in K}\P(\pr^{Z}_{Z_j})(\beta_j) \in I_{Z}$, using \cref{d:univ-ideal}\eqref{i:ideal-closed-binary-joins} (and \eqref{i:ideal-bot-tmn} if $K = \varnothing$, in which case $Z = \tmn$ and $\gamma = \bot_{\P(\tmn)}$).
    If $J = \varnothing$, then $\gamma \in F_{Z} \cap I_{Z}$.
    Otherwise, reindexing along the morphism $Y \times Z \to \prod_{l = 1}^m Z_l$ whose component at $j \in J$ is $\pr^{Y \times Z}_Y$ and whose component at $j \in K$ is $\pr^{Z}_{Z_j}\circ\pr^{Y \times Z}_{Z}$, we obtain $\P(\pr^{Y\times Z}_Y)(\alpha) \lor \P(\pr^{Y\times Z}_Z)(\gamma) \in F_{Y \times Z}$.
    The converse implications are immediate.
\end{proof}

\subsection{Universal ultrafilter lemma}

One version of the classical ultrafilter lemma is: in a Boolean algebra, every filter disjoint from an ideal $I$ can be extended to a prime filter disjoint from $I$ (see \cite[Thm.~6]{Stone1938} in the larger context of bounded distributive lattices, and see also the earlier result by G.~Birkhoff \cite[Thm.~21.1]{Birkhoff1933}).
We give an analogous version in our context.

\begin{theorem}[Universal ultrafilter lemma] \label{t:gen-ult-lem}
    Let $\P \colon \C\op \to \BA$ be a Boolean doctrine, with $\C$ small, $(F_X)_{X \in \C}$ a universal filter for $\P$ and $(I_X)_{X \in \C}$ a universal ideal for $\P$. Suppose that, for all $X\in\C$, $F_X\cap I_X=\varnothing$. There is a universal ultrafilter for $\P$ that, componentwise, extends $F$ and is disjoint from $I$.
\end{theorem}

\begin{proof}
    Let $\mathcal{A}$ be the set of pairs $((G_X)_{X\in\C},(J_X)_{X\in \C})$ where $G$ is a universal filter for $\P$ that extends $F$ componentwise, $J$ is a universal ideal for $\P$ that extends $I$ componentwise, and $G$ and $J$ are componentwise disjoint.
    We order $\mathcal{A}$ by componentwise inclusion. Any nonempty chain of $\mathcal{A}$ has an upper bound: the componentwise union.
    So, by Zorn's lemma, $(F, I)$ is below some maximal element $(G,J)$.
    
    We prove that $G$ and $J$ are componentwise complementary. Since  $G$ and $J$ are componentwise disjoint, it is enough to show $\P(Y)=G_Y\cup J_Y$ for all $Y\in\C$.
    By way of contradiction, suppose this is not the case. So, there are $Y\in\C$ and $\alpha\in \P(Y)$ such that $\alpha\notin G_Y\cup J_Y$.
    Let $G'$ be the universal filter generated by $G$ and $\alpha$, and $J'$ the universal ideal generated by $J$ and $\alpha$.
    By maximality of $(G, J)$ in $\mathcal{A}$, $G'$ intersects $J$ and $G$ intersects $J'$.
    By \cref{l:char-gen-intersects}\eqref{i:intersect-1}, there are $X\in \C$, $n\in\N$, $(f_i\colon X\to Y)_{i=1,\dots, n}$, $\beta\in G_X$ such that
    \[
        \beta\land\bigwedge_{i = 1}^n\P(f_i)(\alpha)\in J_X.
    \]
    By \cref{l:char-gen-intersects}\eqref{i:intersect-2}, and since $G$ and $J$ are componentwise disjoint, there are $Z \in \C$ and $\gamma \in J_Z$ such that 
    \[
    \P(\pr^{Y\times Z}_Y)(\alpha)\lor\P(\pr^{Y\times Z}_Z)(\gamma)\in G_{Y\times Z}.
    \]
    The universal filter $G$ is closed under reindexings: thus, for every $i=1,\dots,n$,
    \[
        \P(\pr^{X\times Z}_X)\P(f_i)(\alpha)\lor\P(\pr^{X\times Z}_Z)(\gamma) = \P(f_i \times \id_Z)(\P(\pr^{Y\times Z}_Y)(\alpha) \lor \P(\pr^{Y\times Z}_Z)(\gamma)) \in G_{X\times Z},
    \]
    where we have reindexed $\P(\pr^{Y\times Z}_Y)(\alpha)\lor\P(\pr^{Y\times Z}_Z)(\gamma)$ along $f_i\times\id_Z\colon X\times Z\to Y\times Z$.
    Moreover, since $\beta\in G_X$,
    \[\P(\pr^{X\times Z}_X)(\beta)\in G_{X\times Z}.\]
    By distributivity of the lattice $\P(X\times Z)$, in $\P(X\times Z)$ we have
    \begin{align}
        & \P(\pr^{X\times Z}_X)(\beta)\land \bigwedge_{i=1}^n\big(\P(\pr^{X\times Z}_X)\P(f_i)(\alpha)\lor\P(\pr^{X\times Z}_Z)(\gamma)\big) \label{eq:G}\\
        & = \P(\pr^{X\times Z}_X)(\beta)\land \big(\P(\pr^{X\times Z}_X)\big(\bigwedge_{i=1}^n\P(f_i)(\alpha)\big)\lor\P(\pr^{X\times Z}_Z)(\gamma)\big) \notag\\
        & = \big(\P(\pr^{X\times Z}_X)(\beta)\land\P(\pr^{X\times Z}_X)\big(\bigwedge_{i=1}^n\P(f_i)(\alpha)\big)\big)\lor\big(\P(\pr^{X\times Z}_X)(\beta)\land\P(\pr^{X\times Z}_Z)(\gamma)\big)\notag\\
        & = \P(\pr^{X\times Z}_X)\big(\beta\land\bigwedge_{i=1}^n\P(f_i)(\alpha)\big)\lor \big(\P(\pr^{X\times Z}_X)(\beta)\land\P(\pr^{X\times Z}_Z)(\gamma)\big)\notag\\
        & \leq \P(\pr^{X\times Z}_X)\big(\beta\land\bigwedge_{i=1}^n\P(f_i)(\alpha)\big)\lor \P(\pr^{X\times Z}_Z)(\gamma). \label{eq:J}
    \end{align}
    Since $G_{X\times Z}$ is a filter, the conjunction in \eqref{eq:G} belongs to $G_{X\times Z}$. Moreover, the element in \eqref{eq:J} belongs to $J_{X\times Z}$. Since $J_{X\times Z}$ is downward closed, the element in \eqref{eq:G} belongs to $J_{X\times Z}$, as well, and so $G_{X\times Z}\cap J_{X\times Z}\neq\varnothing$, a contradiction.
    Thus, by \cref{r:ultrafilter-as-pair}, $G$ is a universal ultrafilter, and it has the desired property.
\end{proof}

\begin{remark}
    In our proof, we did not use the classical ultrafilter lemma for Boolean algebras.
    In turn, the latter follows from \cref{t:gen-ult-lem}: take $\C$ as the trivial category with one object and one morphism.
\end{remark}

\subsection{Richness of a Boolean doctrine with respect to a universal ultrafilter}\label{s:rich}

We want to build models out of universal ultrafilters. To do so, we take inspiration from Henkin's proof of the completeness theorem for first-order logic \cite{Henkin1949}.
Recall that a maximally consistent deductively closed first-order theory $\mathcal{T}$ is \emph{rich} if for every formula $\exists x \,\beta(x) \in \mathcal{T}$ there is a nullary term\footnote{I.e., a term-definable constant, also known as a ground term.} $c$ (a ``witness'') such that $\beta(c) \in \mathcal{T}$. 
We can easily find a model of a rich theory---namely, the set of all nullary terms, with the obvious interpretation of the function and predicate symbols.
We take inspiration from this definition to define \emph{richness} for a Boolean doctrine with respect to a universal ultrafilter. 
Recall that the formulas in the universal ultrafilter are meant to be those whose universal closure is valid (in a certain model). Richness says that, for every formula $\alpha(x)$ whose universal closure is not valid (i.e.\ not belonging to the universal ultrafilter), there is a constant $c$ witnessing this failure (i.e.\ with $\alpha(c)$ not belonging to the universal ultrafilter).

\begin{definition}[Richness]\label{d:rich}
    A Boolean doctrine $\P \colon \C\op \to \BA$ is \emph{rich} with respect to a universal ultrafilter $(F_X)_{X\in\C}$ for $\P$ if for all $X \in \C$ and $\alpha \in \P(X) \setminus F_X$ there is $c \colon \tmn \to X$ such that $\P(c)(\alpha)\notin F_\tmn$.
\end{definition}

\begin{remark} \label{r:rich-translation-to-classic}
    We translate \cref{d:rich} to the classical syntactic setting.
    A universal theory $\mathcal{T}$ is \emph{rich} with respect to a universal ultrafilter (in the sense of
    \cref{r:uf-translation-to-classic}) $(F_n)_{n \in \N}$ for $\T$ if, for every $n \in \N$ and every quantifier-free formula $\alpha(x_1, \dots, x_n) \notin F_n$, there are nullary terms $c_1, \dots, c_n$ such that $\alpha(c_1, \dots, c_n) \notin F_0$. 
\end{remark}

The following shows how to obtain a model out of a universal ultrafilter in the rich case.

\begin{proposition} \label{p:rich-has-model}
    Let $\P \colon \C\op \to \BA$ be a Boolean doctrine that is rich with respect to a universal ultrafilter $(F_X)_{X \in \C}$ for $\P$.
    There is a propositional model $(M,\m)$ of $\P$ such that, for all $X\in\C$,
    \[
        F_X = \{\alpha \in \P(X) \mid \text{for all }x \in M(X),\, x \in \m_X(\alpha)\}.
    \]
\end{proposition}

\begin{proof}
    Set $M \coloneqq \Hom(\tmn, -) \colon \C \to \Set$ (which is representable and hence preserves limits) and $\m \colon \P \to \mathscr{P}\circ M\op$ as the natural transformation whose component at $X \in \C$ is the function
    \begin{align*}
        \m_X \colon \P(X) & \longrightarrow \mathscr{P}(\Hom(\tmn,X))\\
        \alpha & \longmapsto \{c \colon \tmn \to X \mid \P(c)(\alpha)\in F_\tmn\}.
    \end{align*}
    The fact that $\m_X$ is a Boolean homomorphism is easily proved using that $F_\tmn$ is an ultrafilter (by \cref{r:F_tmn-is-uf}) and that $\P(c)$ is a Boolean homomorphism for each $c \colon \tmn \to X$.
    We prove naturality of $\m$: let $X, X' \in \C$, $\alpha\in \P(X)$, and $f\colon X'\to X$ and $c\colon\tmn\to X'$ morphisms in $\C$. We have
    \begin{align*}
        c\in \m_{X'}(\P(f)(\alpha))
        & \iff \P(c)(\P(f)(\alpha))\in F_\tmn\\
        & \iff\P(f \circ c)(\alpha)\in F_\tmn\\
        & \iff f \circ c\in\m_X(\alpha)\\
        & \iff c\in(f\circ - )^{-1}[\m_{X}(\alpha)].
    \end{align*}
    
    To conclude, let $X \in \C$, and let us prove $F_X = \{\alpha \in \P(X) \mid \text{for all }x \in M(X),\, x \in \m_X(\alpha)\}$.
    We first prove the left-to-right inclusion.
    Let $\alpha\in F_X$ and $x \in M(X) = \Hom(\tmn, X)$.
    We shall prove $x \in \m_X(\alpha)$.
    We have $\m_X(\alpha)=\{ c \colon \tmn \to X \mid \P(c)(\alpha)\in F_\tmn\}$.
    So, it is enough to prove $\P(x)(\alpha) \in F_\tmn$.
    This follows from $\alpha\in F_X$ since $F$ is closed under reindexings by \cref{d:uf}\eqref{forall-reindex-ultrafilter}.
    Conversely, suppose $\alpha\notin F_X$. By definition of richness, there is $c\colon\tmn\to X$ such that $\P(c)(\alpha) \notin F_\tmn$, so that $c\notin\m_X(\alpha)$.
\end{proof}

We want to get a model out of a universal ultrafilter also in the non-rich case.
To do so, we will produce a rich theory out of the non-rich one.
For this, the idea is to extend the language $\C$ by adding new constants meant to witness the failure of universal closures of the formulas not belonging to the universal ultrafilter $F$.
Once the constants are added, we interpret the formulas in $F$ in the extended language $\C'$, producing a new class of formulas $G$. However, $G$ may fail to be a universal ultrafilter in the extended language because of the new formulas involving the new constants.
To repair this, we will use the universal ultrafilter lemma (which relies on the axiom of choice) to extend $G$ to a universal ultrafilter $F'$ in the extended language (\cref{l:extension}).
But now we may lack witnesses for some new formulas (involving new constants) not belonging to $F'$.
This calls for an iterative process, where at each step we add new constants and use the universal ultrafilter lemma: the rich theory will be obtained as a colimit after $\omega$ many steps (\cref{t:extension-to-rich}).

The following lemma addresses a single iteration of this process.
Together with the universal ultrafilter lemma, it forms the technical core of the paper.

\begin{lemma} [Extension to richness: first step]\label{l:extension}
    Let $(F_X)_{X \in \C}$ be a universal ultrafilter for a Boolean doctrine $\P \colon \C\op \to \BA$, with $\C$ small.
    There are a small category $\C'$ with the same objects as $\C$, a Boolean doctrine $\P'\colon {\C'}\op \to \BA$, a Boolean doctrine morphism $(R,\mathfrak r) \colon \P\to \P'$ such that $R \colon \C \to \C'$ is the identity on objects, and a universal ultrafilter $(F'_X)_{X \in {\C}'}$ for $\P'$ with the following properties.
    \begin{enumerate}
        \item \label{i:extension-embedding}
        For all $X\in\C$, $F_X=\mathfrak{r}^{-1}_X[F'_{X}]$.
        
        \item \label{i:extension-richness}
        For all $X \in \C$ and $\alpha \in \P(X) \setminus F_X$, there is a morphism $c \colon \tmn_{\C'} \to X$ in $\C'$ with $\P'(c)(\mathfrak{r}_X(\alpha))\notin F'_{\tmn_{\C'}}$.
    \end{enumerate}
\end{lemma}

Before starting the proof, we give an informal outline of it.
For every context $X \in \C$ and formula $\sigma \in I_X\coloneqq \P(X) \setminus F_X$, we add to the language a constant $c_\sigma$ in the context $X$, and denote by $\P'$ the doctrine obtained from $\P$ with this expansion of the language.
We denote by $G$ the universal filter for $\P'$ generated by the formulas in $F$ seen in this new language together with the formulas $\lnot\sigma(c_\sigma/X)$ obtained by substituting $c_\sigma$ for $X$ in $\lnot\sigma$, for each context $X \in \C$ and formula $\sigma \in I_X$.
Moreover, we denote by $J$ the universal ideal for $\P'$ generated by the formulas belonging to $I_X$ (for $X$ ranging over contexts) seen in this new language. We prove that $G$ and $J$ are disjoint, and then apply the universal ultrafilter lemma to extend $G$ to a universal ultrafilter $F'$ disjoint from $J$.

\begin{proof}[Proof of \cref{l:extension}]
    Let $\mathcal{A}$ be the set of finite subsets of $\{(X,\alpha)\mid X \in \C,\, \alpha \in \P(X)\setminus F_X \}$, partially ordered by inclusion. The poset $\mathcal{A}$ is directed.

    For $\bar{X} \in \mathcal{A}$, we write $H_{\bar{X}} \coloneqq \prod_{(X, \alpha) \in \bar{X}} X$.
    
    We define an $\mathcal{A}$-shaped diagram in $\DoctBA$
    \[
    \begin{tikzcd}
        \mathcal{A} && \DoctBA \\
        {\bar{X}} && {{\P}_{H_{\bar{X}}}\colon \C_{H_{\bar{X}}}\op\to\BA} \\
        {\bar{Y}} && {{\P}_{H_{\bar{Y}}}\colon \C_{H_{\bar{Y}}}\op\to\BA}
        \arrow[from=1-1, to=1-3, "D"]
        \arrow["\subseteq"{marking}, draw=none, from=2-1, to=3-1]
        \arrow[maps to, from=2-1, to=2-3]
        \arrow[maps to, from=3-1, to=3-3]
        \arrow["{(L_{\bar{X}\bar{Y}},\mathfrak{l}_{\bar{X}\bar{Y}})}", from=2-3, to=3-3]
    \end{tikzcd}
    \]    
    where $D(\bar{X})$ is the Boolean doctrine ${\P}_{H_{\bar{X}}}\colon {\C_{H_{\bar{X}}}\op\to\BA}$ obtained from ${\P}$ by adding a constant of type $H_{\bar{X}}$. In particular $D(\varnothing)=\P$. For every $\bar{X}\in\mathcal{A}$, $D(\varnothing\subseteq\bar{X})$ is the canonical Boolean doctrine morphism $(L_{\bar{X}},\mathfrak{l}_{\bar{X}})\colon {\P}\to {\P}_{H_{\bar{X}}}$. Moreover, for every $\bar{X}\subseteq\bar Y$, $D(\bar{X}\subseteq\bar{Y})$ is the unique Boolean doctrine morphism $(L_{\bar{X}\bar{Y}},\mathfrak{l}_{\bar{X}\bar{Y}})\colon {\P}_{H_{\bar{X}}}\to{\P}_{H_{\bar{Y}}}$ such that
    \[    (L_{\bar{X}\bar{Y}},\mathfrak{l}_{\bar{X}\bar{Y}})\circ(L_{\bar{X}},\mathfrak{l}_{\bar{X}})=(L_{\bar{Y}},\mathfrak{l}_{\bar{Y}})\quad\text{ and }\quad L_{\bar{X}\bar{Y}}(\id_{H_{\bar{X}}}\colon \tmn {\overset{H_{\bar{X}}}{\rightsquigarrow} }H_{\bar{X}})= \pr_{H_{\bar X}}^{H_{\bar Y}}\colon \tmn\overset{H_{\bar{Y}}}{\rightsquigarrow}H_{\bar{X}},
    \]
    which is defined by the universal property of $\P_{H_{\bar{X}}}$.
    We refer to \cref{r:const} for more details about the construction that adds a constant of a certain type.

    Let ${\P'}\colon {\C'}\op\to\BA$ be the colimit of the directed diagram $D$ in $\DoctBA$, which is computed as follows (see \cite[Sections 2.2 and 3.1]{GuffRich}). The objects in the base category $\C'$ are those of $\C$, since for every $\bar{X}, \bar{Y} \in \mathcal{A}$ with $\bar{X} \subseteq \bar{Y}$, the functor $L_{\bar{X}\bar{Y}}$ is the identity on objects. A morphism $[(\bar{X},f)]_{A,B}$ in $\C'$ from $A$ to $B$---written as $[(\bar{X},f)]_{A,B} \colon A\dashrightarrow B$---is the equivalence class of a pair $(\bar{X},f)$ where $\bar{X}\in \mathcal{A}$ and $f \colon H_{\bar{X}}\times A \to B$ is a morphism in $\C$, with respect to the equivalence relation $\sim_{A,B}$ defined as follows: $( \bar{X},f \colon H_{\bar{X}} \times A \to B) \sim_{A,B} (\bar{Y},g \colon H_{\bar{Y}} \times A \to B)$ if and only if there is $\bar{Z}\in \mathcal{A}$ with $\bar{X}\subseteq\bar{Z}\supseteq\bar{Y}$ making the following diagram commute.
    \[
        \begin{tikzcd}
            && {H_{\bar X}\times A} \\
            {H_{\bar Z}\times A} && 
            && B \\
            && {H_{\bar Y}\times A}
            \arrow["{\pr_{H_{\bar{X}}}^{H_{\bar{Z}}}\times \id_{A}}", from=2-1, to=1-3]
            \arrow["f", from=1-3, to=2-5]
            \arrow["\pr^{H_{\bar Z}}_{H_{\bar Y}}\times \id_{A}"', from=2-1, to=3-3]
            \arrow["{g}"', from=3-3, to=2-5]
        \end{tikzcd}
    \]
    
    For every object $A$, the fiber $\P'(A)$ is the colimit of $D$ in $\BA$ restricted to the fibers.     
    Explicitly, recalling that filtered colimits of Boolean algebras are computed in $\Set$, the Boolean algebra $\P'(A)$ is the set of equivalence classes $[(\bar X,\varphi)]_{A}$ of pairs $(\bar X,\varphi)$ where $\bar X\in\mathcal{A}$ and $\varphi\in \P(H_{\bar X}\times A)$
    with respect to the equivalence relation $\sim_A$ defined as follows: ${(\bar{X},\varphi \in \P(H_{\bar X}\times A))} \sim_A {(\bar{Y}, \psi \in \P(H_{\bar Y}\times A))}$ if and only if there is $\bar{Z}\in \mathcal{A}$ with $\bar{X}\subseteq\bar{Z}\supseteq\bar{Y}$ such that in $\P(H_{\bar Z}\times A)$
    \[
    \P(\pr_{H_{\bar{X}} \times A}^{H_{\bar{Z}} \times A})(\varphi)=\P(\pr_{H_{\bar{Y}} \times A}^{H_{\bar{Z}} \times A})(\psi).
    \]
    We will write $[a,b]$ instead of $[(a,b)]$.
    
    The reindexings are defined via common upper bounds: for $[{\bar{X}, f}]_{A,B}\colon A\dashrightarrow B$ and $[{\bar{Y},\psi}]_B\in\P'(B)$, we set
    \[
    \P'([{\bar{X}, f}]_{A,B})([\bar{Y}, \psi]_B)=[\bar{Z},{\P(\ple{\pr^{H_{\bar Z}\times A}_{H_{\bar Y}},f\circ\pr^{H_{\bar Z} \times A}_{H_{\bar X} \times A}})(\psi)}]_A\in\P'(A),
    \]
    where $\bar{Z}$ is any element of $\mathcal{A}$ such that $\bar{X}\subseteq\bar{Z}\supseteq\bar{Y}$.
    \[\begin{tikzcd}
        {H_{\bar Z}\times A} &&& {H_{\bar X}\times A} & B \\
        {H_{\bar Z}\times B} &&& {H_{\bar Y}\times B}
        \arrow["{\pr^{H_{\bar Z} \times A}_{H_{\bar X} \times A}}", from=1-1, to=1-4]
        \arrow["f", from=1-4, to=1-5]
        \arrow["{\pr^{H_{\bar Z} \times B}_{H_{\bar Y} \times B}}"', from=2-1, to=2-4]
        \arrow["{\ple{\pr^{H_{\bar Z}\times A}_{H_{\bar Z}},f\circ\pr^{H_{\bar Z} \times A}_{H_{\bar X} \times A} }}"', from=1-1, to=2-1]
    \end{tikzcd}\]

   We call $(R,\mathfrak{r})\colon\P\to\P'$ the colimit map from $D(\varnothing) = \P$ to the colimit $\P'$.
In particular, recalling that the objects of $\C$ and $\C'$ are the same, for every object $A \in \C$ we have $R(A) = A$, for every morphism $f \colon A \to B$ in $\C$ we have $R(f) = [\varnothing, f]_{A,B}$, and for every $A \in \C$ and $\alpha \in \P(A)$ we have $\mathfrak{r}_A(\alpha) = [\varnothing, \alpha]_A$.
Moreover, for every $A \in \C$, $\alpha \in \P(A)$ and $\bar{X} \in \mathcal{A}$, 
\[
\mathfrak{r}_A(\alpha) = [\bar{X},\P(\pr_{A}^{H_{\bar{X}} \times A})(\alpha)]_A.
\]

    For $X\in\C$, we set $I_X \coloneqq \P(X)\setminus F_X$, and, for $\sigma \in I_X$, $c_\sigma$ denotes the morphism $[\{(X,\sigma)\},\id_X]_{\tmn,X} \colon \tmn \dashrightarrow X$ in $\C'$.
    We let $G$ be the universal filter for $\P'$ generated by $(\mathfrak{r}_X[F_X])_{X\in\C}$ and the following subset of $\P'(\tmn)$:
    \[
        \{\lnot\P'(c_\sigma)(\mathfrak{r}_X(\sigma))\mid X \in \C,\ \sigma \in I_X\}.
    \]
    We let $J$ be the universal ideal generated by $(\mathfrak{r}_X[I_X])_{X\in\C}$.
    Our next goal is to show that $G$ and $J$ are componentwise disjoint.

    \begin{claim} \label{cl:filter-generated}
        For every $X \in \C$, an element $\varphi' \in \P'(X)$ belongs to $G_X$ if and only if there are $\bar{S}\in\mathcal{A}$ and $\varphi \in \P(H_{\bar{S}} \times X)$ such that $\varphi' = [\bar{S}, \varphi]_X$ and
        \begin{equation}\label{eq:bohh}
            \varphi \lor \bigvee_{(S, \sigma) \in \bar{S}} \P(\pr_{S}^{H_{\bar{S}} \times X})(\sigma) \in F_{H_{\bar{S}} \times X}.
        \end{equation}
    \end{claim}
    \begin{claimproof}
        By \cref{l:desc-filter-ideal-generated}\eqref{i:descr-filter-generated}, for $X \in \C$, $G_X$ is the set of $\varphi' \in \P'(X)$ such that there are $n,p \in \N$, $Y_1,\dots,Y_n, V_1,\dots, V_p \in \C$, $(\alpha_i\in F_{Y_i})_{i=1,\dots,n}$, $(\sigma_k\in I_{V_k})_{k=1,\dots,p}$ and $([\bar{U}_i,f_i]_{X,Y_i}\colon X \dashrightarrow Y_i)_{i=1,\dots,n}$ such that in $\P'(X)$
        \begin{equation*}
            \bigwedge_{i=1}^n \P'([\bar{U}_i,f_i]_{X,Y_i})(\mathfrak{r}_{Y_i}(\alpha_i)) \land \bigwedge_{k=1}^p\lnot \P'(!_{X})\P'(c_{\sigma_k})(\mathfrak{r}_{V_k}(\sigma_k))
            \leq \varphi',
        \end{equation*}
        or, equivalently,
        \begin{equation}\label{eq:bohh1}
            \bigwedge_{i=1}^n \P'([\bar{U}_i,f_i]_{X,Y_i})(\mathfrak{r}_{Y_i}(\alpha_i))  
            \leq \varphi' \lor \P'(!_{X})\mleft(\bigvee_{k=1}^p \P'(c_{\sigma_k})(\mathfrak{r}_{V_k}(\sigma_k))\mright),
        \end{equation}
        which means, unfolding the definitions, that there is an upper bound $\bar{S}$ of $\bar{U}_1 \cup \dots \cup \bar{U}_n \cup \{(V_1,\sigma_1),\dots,(V_p,\sigma_p)\}$ in $\mathcal{A}$ and $\varphi \in \P(H_{\bar{S}} \times X)$ such that $\varphi' = [\bar{S}, \varphi]_X$ and
        \begin{equation} \label{eq:unfolded}
            \bigwedge_{i=1}^n\P(f_i \circ \pr^{H_{\bar S} \times X}_{H_{\bar{U}_i} \times X})(\alpha_i) \leq \varphi \lor \bigvee_{k=1}^p\P(\pr^{H_{\bar{S}} \times X}_{V_k})(\sigma_k).
        \end{equation}

        Let $\varphi'$ possess this property, and let us prove that the formula \eqref{eq:bohh} holds.
        Using that $\alpha_i\in F_{Y_i}$ for each $i$, that universal filters are closed under reindexings and conjunctions and are upwards closed, and that $\bigvee_{k=1}^p\P(\pr^{H_{\bar{S}}\times X}_{V_k})(\sigma_k)\leq\bigvee_{(S, \sigma) \in \bar{S}} \P(\pr_{S}^{H_{\bar{S}} \times X})(\sigma)$, we obtain \eqref{eq:bohh}.

         Conversely, let $\bar{S}\in\mathcal{A}$ and $\varphi\in \P(H_{\bar S}\times X)$ be such that \eqref{eq:bohh} holds. 
         Then $[\bar S,\varphi]_X\in G_X$, because \eqref{eq:bohh1} is obtained immediately by taking $n = 1$, $Y_1 = H_{\bar{S}} \times X$, $(V_1, \sigma_1), \dots, (V_p, \sigma_p)$ an enumeration of the elements of $\bar{S}$, $\alpha_1 = \varphi \lor \bigvee_{(S, \sigma) \in \bar{S}} \P(\pr_{S}^{H_{\bar{S}} \times X})(\sigma)$, $\bar{U}_1 = \bar{S}$ and $f_1$ the identity in $\C$ on $H_{\bar{S}} \times X$. 
    \end{claimproof}

    We now show that $(\mathfrak{r}_X[I_X])_{X\in\C}$ is closed under ``joins over disjoint variables'': for all $Z_1,\dots, Z_m \in \C$ and $(\gamma_j\in I_{Z_j})_{j=1,\dots,m}$, using the fact that $R$ preserves products and using naturality of $\mathfrak{r}$, we have
    \begin{align*}
        \bigvee_{j=1}^m\P'(\pr^{\Pi_hZ_h}_{Z_j})(\mathfrak{r}_{Z_j}(\gamma_j))&=\bigvee_{j=1}^m\P'(R(\pr^{\Pi_hZ_h}_{Z_j}))(\mathfrak{r}_{Z_j}(\gamma_j)) \\
        &= \bigvee_{j=1}^m \mathfrak{r}_{\Pi_h Z_h} \P(\pr^{\Pi_hZ_h}_{Z_j})(\gamma_j)\\
        &= \mathfrak{r}_{\Pi_j Z_j}\mleft(\bigvee_{j=1}^m\P(\pr^{\Pi_hZ_h}_{Z_j})(\gamma_j)\mright),
    \end{align*}
    and we recall that $\bigvee_{j=1}^m\P(\pr^{\Pi_hZ_h}_{Z_j})(\gamma_j) \in I_{\Pi_h Z_h}$.

    Therefore, by \cref{c:id-gen-intersects-filter}, $G$ intersects $J$ if and only if $G$ intersects the family $(\mathfrak{r}_X[I_X])_{X\in\C}$.
    Therefore, by \cref{cl:filter-generated}, $G$ intersects $J$ if and only if there are $X \in \C$, $\gamma \in I_X$, $\bar{S} \in \mathcal{A}$ and $\varphi \in \P(H_{\bar{S}} \times X)$ such that $\mathfrak{r}_X(\gamma) = [\bar{S}, \varphi]_X$ and
    \begin{equation} \label{eq:join-in-F}
        \varphi \lor \bigvee_{(S, \sigma) \in \bar{S}} \P(\pr_{S}^{H_{\bar{S}} \times X})(\sigma) \in F_{H_{\bar{S} }\times X}.
    \end{equation}
    To prove that $G$ does not intersect $J$, we assume that this condition holds and we seek a contradiction.
    Since $\mathfrak{r}_X(\gamma) = [\bar{S}, \varphi]_X$, there is an upper bound $\bar{U}$ of $\bar{S}$ such that 
    \begin{equation} \label{eq:link-between-gamma-and-phi}
        \P(\pr_X^{H_{\bar{U}} \times X})(\gamma) = \P(\pr^{H_{\bar{U}} \times X}_{H_{\bar{S}} \times X})(\varphi).
    \end{equation}
    Then, since $F$ is closed under reindexings, applying $\P(\pr_{H_{\bar{S}} \times X}^{H_{\bar{U}} \times X})$ to the element in \eqref{eq:join-in-F} and using \eqref{eq:link-between-gamma-and-phi}, we get
    \[
        \P(\pr^{H_{\bar{U}} \times X}_{{X}})(\gamma) \lor \bigvee_{(S, \sigma) \in \bar{S}} \P(\pr_{S}^{H_{\bar{U}} \times X})(\sigma) \in F_{H_{\bar{U}} \times X}.
    \]
    Then, since $F$ is upwards closed and $\bar{S} \subseteq \bar{U}$, we have
    \begin{equation} \label{eq:in-F}
        \P(\pr^{H_{\bar{U}} \times X}_{X})(\gamma) \lor \bigvee_{(U, \mu) \in \bar{U}} \P(\pr_{U}^{H_{\bar{U}} \times X})(\mu) \in F_{H_{\bar{U}} \times X}.
    \end{equation}
    The element in \eqref{eq:in-F} is a disjunction of reindexings of elements of $I$ along projections; by \cref{d:univ-ideal}\eqref{i:ideal-closed-binary-joins}, it belongs to $I_{H_{\bar{U}} \times X}$.
    This shows that $F$ intersects $I$, a contradiction.
    Hence, $G$ does not intersect $J$.
    
    Since $G$ and $J$ are disjoint, by \cref{t:gen-ult-lem} there is a universal ultrafilter $F'$ for $\P'$ that extends $G$ and is disjoint from $J$.    
    We check that $F'$ satisfies the desired properties. 
    Since $F'$ extends $G$, which extends $(\mathfrak{r}_X[F_X])_{X\in\C}$, for every $X\in \C$ we have $F_X\subseteq \mathfrak{r}_X^{-1}[F'_X]$. For the converse inclusion, let $\alpha\in \P(X)$ and suppose $\alpha\notin F_X$, i.e.\ $\alpha\in I_X$. Then $\mathfrak{r}_X(\alpha)\in J_X$, which implies  $\mathfrak{r}_X(\alpha)\notin F'_X$ because $J_X$ and $F'_X$ are disjoint.
    This proves that for all $X \in \C$ we have $F_X = \mathfrak{r}_X^{-1}[F'_X]$, which is condition \eqref{i:extension-embedding} in the statement.
    We are left to check the condition \eqref{i:extension-richness} of the statement, i.e.\ that for all $X\in\C$ and $\alpha\in\P(X)\setminus F_X$ there is a morphism $c \colon \tmn \dashrightarrow X$ in $\C'$ such that $\P'(c)(\mathfrak{r}_X(\alpha))\notin F'_\tmn$. Since $\alpha\in I_X$, we can take $c\coloneqq c_\alpha=[\{(X,\alpha)\}, \id_X]_{\tmn,X}$. Recall that $\lnot \P'(c_\alpha)(\mathfrak{r}_X(\alpha))\in\P'(\tmn)$ is a generator of $G$, and so $\lnot \P'(c_\alpha)(\mathfrak{r}_X(\alpha)) \in F'_\tmn$; since $F'_\tmn$ is an ultrafilter of the Boolean algebra $\P'(\tmn)$ (\cref{r:F_tmn-is-uf}), $\P'(c_\alpha)(\mathfrak{r}_X(\alpha))\notin F'_\tmn$, as desired.
\end{proof}

In the following theorem, we show how to produce a rich theory from an arbitrary one.
We accomplish this using \cref{l:extension} $\omega$ many times.
The desired rich theory is obtained as the colimit.

This $\omega$-step construction differs slightly from Henkin's proof of G\"odel's completeness theorem, where many constants are added all at once to the language, and is instead in line with the proof of G\"odel's completeness theorem in \cite[Thm.~3.7]{Johnstone1987}, which uses a similar $\omega$-step construction.
The difference with Henkin's proof lies in our more economical approach: we only add constants that witness specific existential statements. In some cases, this may result in no new constants being added at all, as, for instance, if we start with a theory whose only model is the empty structure; in this case, adding constants would lead to inconsistency.
This cautious approach reflects the fact that---roughly speaking---we allow empty models.

\begin{theorem} [Extension to richness] \label{t:extension-to-rich}
    Let $(F_X)_{X \in \C}$ be a universal ultrafilter for a Boolean doctrine $\P \colon \C\op \to \BA$ with $\C$ small.
    There are a category $\C'$ with the same objects as $\C$, a Boolean doctrine $\P'\colon {\C'}\op \to \BA$, a Boolean doctrine morphism $(R,\mathfrak r) \colon \P\to \P'$ such that $R \colon \C \to \C'$ is the identity on objects, and a universal ultrafilter $(F'_X)_{X \in {\C}'}$ for $\P'$ such that $\P'$ is rich with respect to $(F'_X)_{X\in\C'}$, and such that, for all $X\in\C$, $F_X=\mathfrak{r}^{-1}_X[F'_{X}]$.
\end{theorem}

\begin{proof}
    We define a sequence $(\P^n \colon (\C^{n})\op \to \BA)_{n\in\N}$ of Boolean doctrines with $\C^n$ small and having the same objects as $\C$, together with a sequence $((R^n, \mathfrak{r}^n) \colon \P^n \to \P^{n +1})_{n \in \N}$ of Boolean doctrine morphisms where each $R^n\colon \C^n \to \C^{n+1}$ is the identity on objects, and a sequence $(F^n)_{n \in \N}$ with $F^n$ a universal ultrafilter for $\P^n$ with the following properties.
    \begin{enumerate}
        \item \label{i:extension-base} 
        $\C^0 = \C$, $\P^0=\P$ and $F^0 = F$.
        
        \item \label{i:extension-embedding-inductive}
        For all $n \in \N$ and $X\in\C^n$, $F^n_X = {(\mathfrak{r}^n_X)}^{-1}[F^{n +1}_{X}]$.
        
        \item \label{i:extension-richnessinductive}
        For all $n\in\N$, $X \in \C^n$ and $\alpha \in \P^n(X) \setminus F^n_X$, there is a morphism $c \colon \tmn_{\C^{n+1}} \to X$ in $\C^{n+1}$ such that $\P^{n+1}(c)(\mathfrak{r}^n_X(\alpha))\notin F^{n+1}_{\tmn_{\C^{n+1}}}$.
    \end{enumerate}
    
    We define these sequences inductively (with the aid of the axiom of dependent choice).
    For the base case, set $\C^0 \coloneqq \C$, $\P^0 \coloneqq \P$ and $F^0 \coloneqq F$.
    For the inductive case, for any $n \in \N$, given $\C^n$, $\P^n$ and $F^n$, apply \cref{l:extension} to produce $\C^{n+1}$, $\P^{n+1}$, $(R^{n}, \mathfrak{r}^n)$ and $F^{n+1}$.
    This gives us the sequences with the desired properties.
    The sequence allows us to define a directed diagram of Boolean doctrines indexed by the poset $(\N, \leq)$ of natural numbers:
    \[
        \P^0 \xrightarrow{(R^0, \mathfrak{r}^0)} \P^1 \xrightarrow {(R^1, \mathfrak{r}^1)} \P^2\to\dots\to\P^n\xrightarrow {(R^n, \mathfrak{r}^n)}\P^{n+1}\to\dots
    \]
    For every $n\leq m\in\N$, we call $(R^{n;m},\mathfrak{r}^{n;m})$ the composite
    \[
    (R^{m-1},\mathfrak{r}^{m-1}) \circ \dots \circ (R^{n+1},\mathfrak{r}^{n+1})\circ (R^n,\mathfrak{r}^n) \colon \P^n \longrightarrow \P^m.
    \]
    
    Let $\P'\colon{\C'}\op\to\BA$ be the colimit of this diagram, and $(Q^n,\mathfrak{q}^n)\colon \P^n\to\P'$ the colimit morphism for each $n\in\N$. In particular, we define the desired Boolean doctrine morphism $(R,\mathfrak{r}) \colon \P \to \P'$ as $(R,\mathfrak{r})\coloneqq (Q^0,\mathfrak{q}^0)$.
    
    We collect here some properties of this colimit. 
    We can choose the colimit in a way such that $\C'$ has the same objects as $\C$, and for each $n$ the functor $Q^n \colon \C^n \to \C'$ is the identity on objects.
    Since the colimit category $\C'$ is computed in $\mathbf{Cat}$ by \cite[Sec.~2.2]{GuffRich}, for every morphism $g \colon X \to Y$ in $\C'$ there are $n\in \N$ and $f\colon X\to Y$ in $\C^n$ such that $g = Q^n(f)$.
    \[
    \begin{tikzcd}
        {\P^0(X)} & {\P^1(X)} & {\P^2(X)} & \dots & {\P^n(X)} & \dots \\
        \\
        && {\P'(X)}
        \arrow["{\mathfrak{r}^0_X}", from=1-1, to=1-2]
        \arrow["{\mathfrak{r}^2_X}", from=1-3, to=1-4]
        \arrow["{\mathfrak{r}^{n-1}_X}", from=1-4, to=1-5]
        \arrow["{\mathfrak{r}^n_X}", from=1-5, to=1-6]
        \arrow["{\mathfrak{q}^0_X}"'{description}, from=1-1, to=3-3]
        \arrow["{\mathfrak{r}^1_X}", from=1-2, to=1-3]
        \arrow["{\mathfrak{q}^1_X}"'{description}, from=1-2, to=3-3]
        \arrow["{\mathfrak{q}^2_X}"'{description}, from=1-3, to=3-3]
        \arrow["{\mathfrak{q}^n_X}"'{description}, from=1-5, to=3-3]
        \arrow["{\mathfrak{r}^{0;2}_X}", curve={height=-18pt}, from=1-1, to=1-3]
    \end{tikzcd}
    \]
    
    Moreover, for every $X \in \C$ and $\beta \in \P'(X)$, there are $n \in \N$ and $\alpha \in \P^n(X)$ such that $\beta = \mathfrak{q}^n_X(\alpha)$.
    Furthermore, for every $\alpha\in \P^n(X)$ and $\beta\in \P^m(X)$,
    \[
        \mathfrak{q}^n_X(\alpha)\leq \mathfrak{q}^m_X(\beta)\iff \text{there is } k\geq n,m\text{ such that }\mathfrak{r}^{n;k}_X(\alpha)\leq\mathfrak{r}^{m;k}_X(\beta).
    \]
        
    For each $X\in\C$, we set $F'_X\coloneqq \bigcup_{n\in\N}\mathfrak{q}^n_X[F^n_X]$ and $I'_X\coloneqq \bigcup_{n\in\N}\mathfrak{q}^n_X[I^n_X]$, where $I^n_X\coloneqq \P^n(X)\setminus F^n_X$.
    
    Next, we prove that $F'$ is a universal filter and $I'$ a universal ideal.
    Roughly speaking, these facts hold because universal filters and ideals are defined by closure conditions involving finitely many elements, and thus they are preserved in a directed colimit.

    We first prove that $F'$ is a universal filter. Let $f \colon X\to Y$ be a morphism in $\C'$ and $\alpha\in F'_Y$. 
    We prove that $\P'(f)(\alpha)\in F'_X$. There are $n,m\in\N$, $g\colon X\to Y$ in $\C^n$ and $\beta\in F^m_Y$ such that $Q^n(g)=f$ and $\mathfrak{q}^m_Y(\beta)=\alpha$. We take $k\geq n,m$ and compute
    \[
        \P'(f)(\alpha) = \P'(Q^n(g))(\mathfrak{q}^m_Y(\beta))
        =\P'(Q^k(R^{n;k}(g)))(\mathfrak{q}^k_Y(\mathfrak{r}^{m;k}_Y(\beta)))=\mathfrak{q}^k_X(\P^k(R^{n;k}(g))(\mathfrak{r}^{m;k}_Y(\beta))).
    \]
    By \eqref{i:extension-embedding-inductive}, $\mathfrak{r}^{m;k}_Y(\beta)\in F^k_Y$. Then, since $F^k$ is closed under reindexing, $\P^k(R^{n;k}(g))(\mathfrak{r}^{m;k}_Y(\beta))\in F^k_X$; hence, by definition of $F_X'$, $\P'(f)(\alpha) = \mathfrak{q}^k_X(\P^k(R^{n;k}(g))(\mathfrak{r}^{m;k}_Y(\beta)))\in F'_X$.
    
    We show that, for every $X\in\C$, $F'_X$ is a filter. Since $\top_{\P(X)}\in F^0_X$, we have $\top_{\P'(X)} = \mathfrak{q}^0_X(\top_{\P(X)})\in F'_X$.     
    Let $n,m\in\N$, $\alpha\in F^n_X$ and $\beta\in F^m_X$. For $k\geq n,m$ we have
    \begin{equation}\label{eq:q-meet}
        \mathfrak{q}^n_X(\alpha) \land \mathfrak{q}^m_X(\beta) = \mathfrak{q}^k_X(\mathfrak{r}^{n;k}_X(\alpha)) \land \mathfrak{q}^k_X(\mathfrak{r}^{m;k}_X(\beta)) = \mathfrak{q}^k_X(\mathfrak{r}^{n;k}_X(\alpha) \land \mathfrak{r}^{m;k}_X(\beta)).
    \end{equation}
    By \eqref{i:extension-embedding-inductive}, we deduce  $\mathfrak{r}^{n;k}_X(\alpha),\mathfrak{r}^{m;k}_X(\beta) \in F^k_X$. Since $F^k_X$ is a filter, the conjunction in \eqref{eq:q-meet} belongs to $F^k_X$.
    Thus, $F'_X$ is closed under binary meets.
    We next show that $F'_X$ is upward closed.
    Let $n,m\in\N$, $\alpha\in F^n_X$ and $\beta\in\P^m(X)$ be such that $\mathfrak{q}^n_X(\alpha)\leq\mathfrak{q}^m_X(\beta)$. There is $k\geq n,m$ such that $\mathfrak{r}^{n;k}_X(\alpha)\leq\mathfrak{r}^{m;k}_X(\beta)$. By \eqref{i:extension-embedding-inductive}, $\mathfrak{r}^{n;k}_X(\alpha)\in F^k_X$, and hence $\mathfrak{r}^{m;k}_X(\beta)\in F^k_X$. Thus, $\mathfrak{q}^m_X(\beta)=\mathfrak{q}^k_X(\mathfrak{r}^{m;k}_X(\beta))\in F'_X$, as desired.
    This shows that $F'$ is a universal filter.

    We now show that $I'$ is a universal ideal.
    Let $n,m\in\N$, $X,Y\in\C$, $(f_j\colon X\to Y)_{j=1,\dots,m}$ morphisms in $\C'$ and $\alpha\in \P^n(Y)$ with $\bigwedge_{j=1}^m\P'(f_j)(\mathfrak{q}_Y^n(\alpha))\in I'_X$.
    We prove that $\mathfrak{q}^n_Y(\alpha)\in I'_Y$. For all $j=1,\dots,m$, there are $n_j\in\N$ and a morphism $g_j \colon X \to Y$ in $\C^{n_j}$ with $f_j=Q^{n_j}(g_j)$. So, for all $k\geq n,n_1,\dots,n_m$, we compute:
    \begin{align*}
        \bigwedge_{j=1}^m\P'(f_j)(\mathfrak{q}_Y^n(\alpha))&= \bigwedge_{j=1}^m\P'(Q^{n_j}(g_j))(\mathfrak{q}^n_Y(\alpha))\\
        &=\bigwedge_{j=1}^m\P'(Q^k(R^{n_j;k}(g_j)))(\mathfrak{q}^k_Y(\mathfrak{r}^{n;k}_Y(\alpha)))\\
        &=\bigwedge_{j=1}^m\mathfrak{q}^k_X(\P^k(R^{n_j;k}(g_j))(\mathfrak{r}^{n;k}_Y(\alpha)))\\
        &=\mathfrak{q}^k_X\mleft(\bigwedge_{j=1}^m\P^k(R^{n_j;k}(g_j))(\mathfrak{r}^{n;k}_Y(\alpha))\mright).
    \end{align*}
    Since $\mathfrak{q}^k_X\mleft(\bigwedge_{j=1}^m\P^k(R^{n_j;k}(g_j))(\mathfrak{r}^{n;k}_Y(\alpha))\mright)\in I'_X$, there are $t\in\N$ and $\beta\in I^t_X$ such that, in $\P'(X)$,
    \[
        \mathfrak{q}^k_X\mleft(\bigwedge_{j=1}^m\P^k(R^{n_j;k}(g_j))(\mathfrak{r}^{n;k}_Y(\alpha))\mright)=\mathfrak{q}^t_X(\beta).
    \]
    Hence there is $s\geq k,t$ such that
    \[
        \mathfrak{r}^{k;s}_X\mleft(\bigwedge_{j=1}^m\P^k(R^{n_j;k}(g_j))(\mathfrak{r}^{n;k}_Y(\alpha))\mright)=\mathfrak{r}^{t;s}_X(\beta).
    \]
    Moreover,
    \begin{align*}
        \mathfrak{r}^{k;s}_X\mleft(\bigwedge_{j=1}^m \P^k(R^{n_j;k}(g_j))(\mathfrak{r}^{n;k}_Y(\alpha))\mright) & = \bigwedge_{j=1}^m \mathfrak{r}^{k;s}_X(\P^k(R^{n_j;k}(g_j))(\mathfrak{r}^{n;k}_Y(\alpha)))\\
        & = \bigwedge_{j=1}^m\P^s(R^{k;s}(R^{n_j;k}(g_j)))(\mathfrak{r}^{k;s}_Y(\mathfrak{r}^{n;k}_Y(\alpha)))\\
        & = \bigwedge_{j=1}^m\P^s(R^{n_j;s}(g_j))(\mathfrak{r}^{n;s}_Y(\alpha)).
    \end{align*}
    We observe that $\mathfrak{r}^{t;s}_X(\beta)\in I^s_X$. Indeed, if we had $\mathfrak{r}^{t;s}_X(\beta)\notin I^s_X$, then we would have $\mathfrak{r}^{t;s}_X(\beta)\in F^s_X$, and hence, by \eqref{i:extension-embedding-inductive}, $\beta\in F^t_X$, a contradiction.
    
    It follows that $\bigwedge_{j=1}^m\P^s(R^{n_j;s}(g_j))(\mathfrak{r}^{n;s}_Y(\alpha))\in I^s_X$. Since $I^s$ is a universal ideal for $\P^s$, we get $\mathfrak{r}^{n;s}_Y(\alpha)\in I^s_Y$. Hence, $\mathfrak{q}^n_Y(\alpha) = \mathfrak{q}^s_Y(\mathfrak{r}^{n;s}_Y(\alpha)) \in I'_Y$, as desired.

    The proof that $I'$ is componentwise downward closed is similar to the proof that $F'$ is componentwise upward closed seen above.
    To check the condition \eqref{i:ideal-closed-binary-joins} in \cref{d:univ-ideal}, we take $\alpha_1\in I^n_{X_1}$ and $\alpha_2\in I^m_{X_2}$, and we prove that $\P'(\pr^{X_1\times X_2}_{X_1})(\mathfrak{q}^n_{X_1}(\alpha_1))\lor\P'(\pr^{X_1\times X_2}_{X_2})(\mathfrak{q}^m_{X_2}(\alpha_2))\in I'_{X_1\times X_2}$. Let $k\geq n,m$.
    Using again the condition \eqref{i:extension-embedding-inductive}, we have $\mathfrak{r}^{n;k}_{X_1}(\alpha_1)\in I^k_{X_1}$ and $\mathfrak{r}^{m;k}_{X_2}(\alpha_2)\in I^k_{X_2}$. Since $I^k$ is a universal ideal for $\P^k$, it follows that $\P^k(\pr^{X_1\times X_2}_{X_1})(\mathfrak{r}^{n;k}_{X_1}(\alpha_1))\lor\P^k(\pr^{X_1\times X_2}_{X_2})(\mathfrak{r}^{m;k}_{X_2}(\alpha_2))\in I^k_{X_1\times X_2}$, and hence
    \[
    \mathfrak{q}^k_{X_1\times X_2}(\P^k(\pr^{X_1\times X_2}_{X_1})(\mathfrak{r}^{n;k}_{X_1}(\alpha_1))\lor\P^k(\pr^{X_1\times X_2}_{X_2})(\mathfrak{r}^{m;k}_{X_2}(\alpha_2)))\in I'_{X_1\times X_2}.
    \]
    Moreover,
    \begin{align*}
        &\mathfrak{q}^k_{X_1\times X_2}(\P^k(\pr^{X_1\times X_2}_{X_1})(\mathfrak{r}^{n;k}_{X_1}(\alpha_1))\lor\P^k(\pr^{X_1\times X_2}_{X_2})(\mathfrak{r}^{m;k}_{X_2}(\alpha_2)))\\
        &=\mathfrak{q}^k_{X_1\times X_2}(\P^k(\pr^{X_1\times X_2}_{X_1})(\mathfrak{r}^{n;k}_{X_1}(\alpha_1)))\lor\mathfrak{q}^k_{X_1\times X_2}(\P^k(\pr^{X_1\times X_2}_{X_2})(\mathfrak{r}^{m;k}_{X_2}(\alpha_2)))\\
        &=\P'(\pr^{X_1\times X_2}_{X_1})(\mathfrak{q}^k_{X_1}(\mathfrak{r}^{n;k}_{X_1}(\alpha_1)))\lor\P'(\pr^{X_1\times X_2}_{X_2})(\mathfrak{q}^k_{X_2}(\mathfrak{r}^{m;k}_{X_2}(\alpha_2)))\\
        &=\P'(\pr^{X_1\times X_2}_{X_1})(\mathfrak{q}^n_{X_1}(\alpha_1))\lor\P'(\pr^{X_1\times X_2}_{X_2})(\mathfrak{q}^m_{X_2}(\alpha_2)).
    \end{align*}

    To conclude, the proof that the condition \eqref{i:ideal-bot-tmn} in \cref{d:univ-ideal} is met is similar to the proof that $F'$ contains the top elements of each fiber.
    This shows that $I'$ is indeed a universal ideal for $\P'$.

    We next show that, for every $X \in \C$, $F'_X$ and $I'_X$ are complementary. 
    Fix $X \in \C$. 
    To prove $F'_X \cup I'_X=\P'(X)$, let $\alpha\in \P'(X)$.
    There are $n\in\N$ and $\beta \in \P^n(X)$ such that $\alpha=\mathfrak{q}^n_X(\beta)$. 
    Since $F^n_X$ and $I^n_X$ are complementary in $\P^n(X)$, $\beta \in F^n_X$ or $\beta \in I^n_X$, and hence $\alpha\in F'_X$ or $\alpha\in I'_X$. 
    This proves $F'_X \cup I'_X=\P'(X)$. 
    We now show that $F'_X \cap I'_X = \varnothing$. 
    We take $\alpha\in F'_X\cap I'_X$ and we seek a contradiction. 
    Since $\alpha\in F'_X$, there are $n\in\N$ and $\beta \in F^n_X$ such that $\alpha=\mathfrak{q}^n_X(\beta)$. 
    Since $\alpha\in I'_X$, there are $m\in\N$ and $\gamma \in I^m_X$ such that $\alpha=\mathfrak{q}^m_X(\gamma)$. 
    Since $\mathfrak{q}^n_X(\beta)=\mathfrak{q}^m_X(\gamma)$, there is $k\geq n,m$ such that $\mathfrak{r}^{n;k}_X(\beta)=\mathfrak{r}^{m;k}_X(\gamma)$, contradicting $F^k_X \cap I^k_X = \varnothing$.

    This proves that $F'$ is a universal ultrafilter (see \cref{r:ultrafilter-as-pair}).

    We prove that $\P'$ is rich with respect to $F'$. Let $X\in \C$ and $\alpha \in \P'(X) \setminus F'_X$.
    We seek a morphism $c \colon \tmn \to X$ in $\C'$ such that  $\P'(c)(\alpha)\notin F'_\tmn$. 
    Since $\alpha\notin F'_X$, we have $\alpha\in I'_X$. So there are $n\in \N$ and $\beta\in I^n_X$ such that $\alpha = \mathfrak{q}^n_X(\beta)$.
    Since $\beta \in \P^n(X)\setminus F^n_X$, we use the property \eqref{i:extension-richnessinductive} to get a morphism $d\colon \tmn \to X$ in $\C^{n+1}$ such that $\P^{n+1}(d)(\mathfrak{r}^n_X(\beta))\notin F^{n+1}_\tmn$, and so $\P^{n+1}(d)(\mathfrak{r}^n_X(\beta))\in I^{n+1}_\tmn$.
    Setting $c \coloneqq Q^{n+1}(d)$, since $\alpha=\mathfrak{q}^n_X(\beta)=\mathfrak{q}^{n+1}_X(\mathfrak{r}^n_X(\beta))$, we have
    \[
    \P'(c)(\alpha)= \P'(Q^{n+1}(d))(\mathfrak{q}^{n+1}_X(\mathfrak{r}^n_X(\beta)))=\mathfrak{q}^{n+1}_\tmn ( \P^{n+1}(d)(\mathfrak{r}^n_X(\beta))) \in I'_\tmn.
    \]
    Therefore, $\P'(c)(\alpha)\notin F'_\tmn$, as desired.

    Finally, we prove $F_X=\mathfrak{r}_X^{-1}[F'_{X}]$ for each $X\in\C$. 
    Let $\alpha\in \P(X)$. If $\alpha\in F_X$, then $\mathfrak{r}_X(\alpha)=\mathfrak{q}^0_X(\alpha) \in F'_X$ by definition. Conversely, if $\alpha\notin F_X$, then $\alpha\in I_X$, so $\mathfrak{r}_X(\alpha)=\mathfrak{q}^0_X(\alpha)\in I'_X$, and thus $\mathfrak{r}_X(\alpha)\notin F'_X$.
\end{proof}

\subsection{Characterization of classes of universally valid formulas}
We have all the ingredients to prove the main theorem of this section, which shows that universal ultrafilters are precisely the classes of universally valid formulas.

\begin{theorem}\label{t:characterization-ultrafilters}
    Let $\P \colon \C\op \to \BA$ be a Boolean doctrine, with $\C$ small, and let $F = (F_X)_{X \in \C}$ be a family with $F_X \subseteq \P(X)$ for each $X \in \C$.
    The following are equivalent.
    \begin{enumerate}
        \item \label{i:exists-model}
        There is a propositional model $(M,\m)$ of $\P$ such that, for every $X \in \C$,
        \[
        F_X=\{\alpha \in \P(X) \mid \text{for all }x \in M(X),\, x \in \m_X(\alpha)\}.
        \] 
        \item \label{i:is-universal-ultrafilter}
        $F$ is a universal ultrafilter for $\P$.
    \end{enumerate}
\end{theorem}
\begin{proof}
    \eqref{i:exists-model} $\Rightarrow$ \eqref{i:is-universal-ultrafilter}. We check that the conditions in \cref{d:uf} are satisfied.
    
    First, we prove that $F$ is closed under reindexings. Take a morphism $f\colon X\to Y$ and $\alpha\in F_Y$, i.e.\ $\m_Y(\alpha)=M(Y)$. Then $\P(f)(\alpha)\in F_X$ if and only if $\m_X(\P(f)(\alpha))=M(X)$. By naturality of $\m$,
    \[\m_X(\P(f)(\alpha))=M(f)^{-1}[\m_Y(\alpha)]=M(f)^{-1}[M(Y)]=M(X).\]
    Thus, $F$ is closed under reindexing.

    Second, we prove that $F$ is fiberwise a filter. For every $X\in\C$ we have $F_X=\m_X^{-1}[\{M(X)\}]=\m_X^{-1}[\{\top_{\mathscr{P}(M(X))}\}]$, and this is a filter since it is the preimage under the Boolean homomorphism $\m_X$ of the filter $\{M(X)\}$ of $\mathscr{P}(M(X))$.
    
    Next, let $\alpha_1 \in \P(X_1) \setminus F_{X_1}$ and $\alpha_2 \in \P(X_2) \setminus F_{X_2}$.
    Then, there are $x_1\in M(X_1)\setminus \m_{X_1}(\alpha_1)$ and $x_2\in M(X_2)\setminus \m_{X_2}(\alpha_2)$.
    We show that $\P(\pr^{X_1\times X_2}_{X_1})(\alpha_1)\lor \P(\pr^{X_1\times X_2}_{X_2})(\alpha_2)\notin F_{X_1\times X_2}$.
    We have
    \begin{align*}
        &\m_{X_1\times X_2}(\P(\pr^{X_1\times X_2}_{X_1})(\alpha_1)\lor \P(\pr^{X_1\times X_2}_{X_2})(\alpha_2))\\
        &= \m_{X_1\times X_2}(\P(\pr^{X_1\times X_2}_{X_1})(\alpha_1))\cup\m_{X_1\times X_2}( \P(\pr^{X_1\times X_2}_{X_2})(\alpha_2))\\
        &=(\pr^{M(X_1)\times M(X_2)}_{M(X_1)})^{-1}[\m_{X_1}(\alpha_1)]\cup (\pr^{M(X_1)\times M(X_2)}_{M(X_2)})^{-1}[\m_{X_2}(\alpha_2)].
    \end{align*}
    Observe that the element $(x_1,x_2)\in M(X_1)\times M(X_2)=M(X_1\times X_2)$ does not belong to the set $\m_{X_1\times X_2}(\P(\pr^{X_1\times X_2}_{X_1})(\alpha_1)\lor \P(\pr^{X_1\times X_2}_{X_2})(\alpha_2))$.
    
    Finally, we have $\bot_{\P(\tmn)}\notin F_\tmn$ since $\m_\tmn(\bot_{\P(\tmn)})=\varnothing$.
    
    \eqref{i:is-universal-ultrafilter} $\Rightarrow$ \eqref{i:exists-model}.
    By \cref{t:extension-to-rich}, there are a category $\C'$, a Boolean doctrine $\P'\colon {\C'}\op \to \BA$, a Boolean doctrine morphism $(R,\mathfrak r) \colon \P\to \P'$ and a universal ultrafilter $(F'_Y)_{Y \in {\C}'}$ for $\P'$ such that $\P'$ is rich with respect to $(F'_Y)_{Y\in\C'}$, and moreover, for all $X\in\C$, $F_X=\mathfrak{r}^{-1}_X[F'_{R(X)}]$.
    By \cref{p:rich-has-model}, there is a propositional model $(M',\m')$ of $\P'$ such that, for all $Y\in\C'$,
    \[
    F'_Y = \{\alpha \in \P'(Y) \mid \text{for all }x \in M'(Y),\, x \in \m'_Y(\alpha)\}.
    \]
    Let $(M, \m)$ be the composite $(M', \m') \circ (R, \mathfrak{r})$ of the morphisms $(M', \m')$ and $(R, \mathfrak{r})$. 
    Clearly, $(M, \m)$ is a propositional model of $\P$.
    Moreover, for every $X \in \C$ and $\alpha \in \P(X)$,
    \begin{align*}
        \alpha \in F_X & \iff \mathfrak{r}_X(\alpha) \in F'_{R(X)}\\
        & \iff \text{for all }x \in M'(R(X)),\, x \in \m'_{R(X)}(\mathfrak{r}_X(\alpha))\\
        & \iff \text{for all }x \in M(X),\, x \in \m_X(\alpha). \qedhere
    \end{align*}
\end{proof}

Note that the model obtained from the universal ultrafilter is not canonical: indeed, its existence was established using the extension to richness, which uses the axiom of choice.

\begin{remark}
    We translate \cref{t:characterization-ultrafilters} to the classical syntactic setting.
    Let $\mathcal{T}$ be a universal theory.
    Let $(F_n)_{n \in \N}$ be a family with $F_n$ a set of quantifier-free formulas with $x_1, \dots, x_n$ as free (possibly dummy) variables.
    The following are equivalent.
    \begin{enumerate}
        \item There is a model $M$ of $\mathcal{T}$ such that, for every $n \in \N$,
        \[
        F_n = \{\alpha(x_1, \dots, x_n) \text{ quantifier-free} \mid M \vDash \forall x_1 \dots \forall x_n\, \alpha(x_1, \dots, x_n)\}.
        \]

        \item $(F_n)_{n \in \N}$ is a universal ultrafilter for $\mathcal{T}$ (in the sense of \cref{r:uf-translation-to-classic}).
    \end{enumerate}
\end{remark}

\begin{corollary} \label{c:sequent}
    Let $\P \colon \C\op \to \BA$ be a Boolean doctrine, with $\C$ small.
    Let $\bar{i}, \bar{j} \in \N$, $Y_1, \dots, Y_{\bar{i}}, Z_1,\dots, Z_{\bar{j}} \in \C$, $(\alpha_i \in \P(Y_i))_{i = 1, \dots, \bar{i}}$, and $(\beta_j \in \P(Z_j))_{j = 1, \dots, {\bar{j}}}$.
    The following are equivalent.
    \begin{enumerate}
        \item \label{i:sequent-model}
        For every propositional model $(M, \mathfrak{m})$ of $\P$, if for every $i \in \{1, \dots, \bar{i}\}$ we have $\m_{Y_i}(\alpha_i)=M(Y_i)$, then there is $j \in \{1, \dots, {\bar{j}}\}$ such that $\m_{Z_j}(\beta_j)=M(Z_j)$.
        
        \item \label{i:sequent-terms}
        There are $n \in \N$, $l_1, \dots, l_n \in \{1, \dots, \bar{i}\}$, and $(g_{i} \colon \prod_{j=1}^{\bar{j}} Z_{j} \to Y_{l_i})_{i = 1, \dots, n}$ such that
        in $\P(\prod_{j=1}^{\bar{j}} Z_{j})$
        \[
        \bigwedge_{i = 1}^{n}\P(g_i)(\alpha_{l_i}) \leq \bigvee_{j=1}^{\bar{j}}\P(\pr^{\Pi_p Z_p}_{Z_j})(\beta_j).
        \]
    \end{enumerate}
\end{corollary}
\begin{proof}
    By \cref{t:characterization-ultrafilters}, condition \eqref{i:sequent-model} is equivalent to 
    \begin{enumerate}[label = (1'), ref = 1']
        \item \label{i:model'}
        For every universal ultrafilter $(F_X)_{X \in \C}$, if for every $i \in \{1, \dots, \bar{i}\}$ we have $\alpha_i \in F_{Y_i}$, then there is $j \in \{1, \dots, \bar{j}\}$ such that $\beta_j \in F_{Z_j}$.
    \end{enumerate}

    By \cref{l:intersect-multiple}, condition \eqref{i:sequent-terms} is equivalent to 
    \begin{enumerate}[label = (2'), ref = 2']
        \item \label{i:intersects'}The universal filter generated by $\alpha_1, \dots, \alpha_{\bar{i}}$ intersects the universal ideal generated by $\beta_1,\dots,\beta_{\bar{j}}$.
    \end{enumerate}
    
    We prove that \eqref{i:model'} is equivalent to \eqref{i:intersects'}.
    To do so, we prove that the negation of \eqref{i:model'} is equivalent to the negation of \eqref{i:intersects'}.

    \begin{enumerate}[label = ($\lnot$1'), ref = $\lnot$1']
        \item \label{i:not-model'}
        There is a universal ultrafilter $(F_X)_{X \in \C}$ such that for all $i \in \{1, \dots, \bar{i}\}$ we have $\alpha_i \in F_{Y_i}$ and for all $j \in \{1, \dots, \bar{j}\}$ we have $\beta_j \notin F_{Z_j}$.
    \end{enumerate}

    \begin{enumerate}[label = ($\lnot$2'), ref = $\lnot$2']
        \item \label{i:not-intersects'}The universal filter generated by $\alpha_1, \dots, \alpha_{\bar{i}}$ and the universal ideal generated by $\beta_1,\dots,\beta_{\bar{j}}$ are fiberwise disjoint.
    \end{enumerate}

    \eqref{i:not-model'} $\Rightarrow$ \eqref{i:not-intersects'}. 
    This is immediate since a universal ultrafilter is a universal filter whose fiberwise complement is a universal ideal.

    \eqref{i:not-intersects'} $\Rightarrow$ \eqref{i:not-model'}. This follows from \cref{t:gen-ult-lem}.
\end{proof}

\begin{remark}\label{r:sequent-classical-setting}
    We translate \cref{c:sequent} to the classical syntactic setting.
    Let $\mathcal{T}$ be a universal theory.
    Let $\bar{i}, \bar{j} \in \N$, let $p_1, \dots, p_{\bar{i}}, q_1, \dots, q_{\bar{j}} \in \N$, let $(\alpha_i(x_1, \dots, x_{p_i}))_{i = 1, \dots, \bar{i}}$ and $(\beta_j(x_1, \dots, x_{q_j}))_{j = 1, \dots, \bar{j}}$ be tuples of quantifier-free formulas.
    The following are equivalent.
    \begin{enumerate}
        \item \label{i:sequent}For every model $M$ of $\mathcal{T}$ we have 
        \[
        M \vDash \mleft(\bigwedge_{i =1}^{\bar{i}} \forall x_1 \dots \forall x_{p_i} \, \alpha_i(x_1, \dots, x_{p_i})\mright) \rightarrow \mleft(\bigvee_{j=1}^{\bar{j}} \forall x_1 \dots \forall x_{q_j}\, \beta_j(x_1, \dots, x_{q_j})\mright).
        \]

       \item \label{i:explicit-sequent}
       There are $n \in \N$, $l_1, \dots, l_n \in \{1, \dots, \bar{i}\}$ and terms $(g^{h}_{i}(x_1, \dots, x_{\Sigma_{j}q_j}))_{i \in\{ 1, \dots, n\},\,h\in\{1,\dots,p_{l_i}\}}$ such that
        \begin{align} \label{eq:sequent-explicit}
        \bigwedge_{i = 1}^{n}\alpha_{l_i}( g^{1}_{i}(x_1, \dots, x_{\Sigma_{j}q_j}), \dots, g^{p_{l_i}}_{i}(x_1, \dots, x_{\Sigma_{j}q_j})) \vdash_\mathcal{T} \bigvee_{j = 1}^{\bar{j}}\beta_{j}( x_{1 + \Sigma_{t =1}^{j-1} q_t}, \dots, x_{\Sigma_{t =1}^{j} q_t}).
        \end{align}
    \end{enumerate}
\end{remark}
The unpleasant game of subscripts in \eqref{i:explicit-sequent} is just a way to ensure that the disjuncts $\beta_j$ on the right-hand side of \eqref{eq:sequent-explicit} have no variables in common.
Note that, by the soundness and completeness theorems for first-order logic, \eqref{i:sequent} is equivalent to 
\[
    \bigwedge_{i =1}^{\bar{i}} \forall x_1 \dots \forall x_{p_i} \, \alpha_i(x_1, \dots, x_{p_i}) \vdash_\mathcal{T} \bigvee_{j=1}^{\bar{j}} \forall x_1 \dots \forall x_{q_j}\, \beta_j(x_1, \dots, x_{q_j}).
\]

\Cref{r:sequent-classical-setting} characterizes when a finite conjunction of universal sentences implies a finite disjunction of universal sentences modulo a universal theory.

\begin{example}\label{ex:forall-alpha-implies-beta}
    For example, if $\alpha(x)$ is a quantifier-free formula and $\beta$ is a quantifier-free closed formula, when does $\forall x\, \alpha(x)$ imply $\beta$ modulo a given universal theory $\mathcal{T}$? 
    To be more precise, we are in the setting of \cref{r:sequent-classical-setting} with $\bar{i} = 1$, $p_1=1$, $\bar{j} = 1$, $q_1 = 0$, $\alpha(x_1)$ a quantifier-free formula, and $\beta$ a closed quantifier-free formula.
    \Cref{r:sequent-classical-setting} tells us that $\forall x\, \alpha(x) \vdash_\mathcal{T} \beta$ occurs precisely when there are $n \in \N$ and nullary terms $g_1, \dots, g_n$ (i.e.\ term-definable constants) such that
    \begin{align*}
    \bigwedge_{i = 1}^{n}\alpha(g_{i})\vdash_\mathcal{T} \beta.
    \end{align*}
    In other words, if we know that $\alpha(x)$ holds for all $x$, the only way to prove $\beta$ is to instantiate $\alpha(x)$ on finitely many term-definable constants $g_1, \dots, g_n$ and then prove $\beta$ from $\alpha(g_1), \dots, \alpha(g_n)$.
\end{example}

\begin{remark}
    Note that, in \cref{ex:forall-alpha-implies-beta} above, it is important that $n$ can also be given the value $0$. For example, if $\beta = \top$ and the language has no term-definable constants, it is true that $\forall x\, \alpha (x)$ implies $\top$, but we cannot instantiate $\alpha (x)$ on any term-definable constants, and so we need permission to take $n = 0$.
\end{remark}

\begin{remark}
    Note that, in \cref{ex:forall-alpha-implies-beta} above, it is also important that we are allowed to take $n\geq2$. For example, let $\mathcal{L}$ be the language with two constant symbols $\{a,b\}$ and a unary predicate symbol $R$, and let $\mathcal{T}=\{R(a)\lor R(b)\}$. When does the theory $\mathcal{T}$ prove the formula $\exists x\, R(x)$? Or equivalently, when does $\forall x\, \lnot R(x)$ imply $\bot$ modulo $\mathcal{T}$? By \cref{ex:forall-alpha-implies-beta} with $\alpha=\lnot R(x)$ and $\beta=\bot$, the ways to prove $\bot$ would be:
    \begin{enumerate}
        \item \label{i:H1} $\top \vdash_{\mathcal{T}} \bot$;
        \item \label{i:H2}
        $\lnot R(a) \vdash_{\mathcal{T}} \bot$;
        \item \label{i:H3}
        $\lnot R(b) \vdash_{\mathcal{T}} \bot$;
        \item \label{i:H4}
        $\lnot R(a)\land \lnot R(b) \vdash_{\mathcal{T}} \bot$.
    \end{enumerate}
    However,
    \begin{itemize}
        \item \eqref{i:H1} does not hold (as witnessed by the model $M=\{\bar{a}\}$, $\mathbb{I}(a)=\mathbb{I}(b)=\bar{a}$, $\mathbb{I}(R)=\{\bar{a}\}$);
        \item \eqref{i:H2} is equivalent to $\lnot R(a)\land (R(a)\lor R(b))\vdash \bot$, which in turn is equivalent to $R(b)\vdash R(a)$, which does not hold (take $M=\{\bar{a},\bar{b}\}$, $\mathbb{I}(a)=\bar{a}$, $\mathbb{I}(b)=\bar{b}$, $\mathbb{I}(R)=\{\bar{b}\}$);
        \item \eqref{i:H3} is equivalent to $\lnot R(b)\land (R(a)\lor R(b))\vdash \bot$, which in turn is equivalent to $R(a)\vdash R(b)$, which does not hold (take $M=\{\bar{a},\bar{b}\}$, $\mathbb{I}(a)=\bar{a}$, $\mathbb{I}(b)=\bar{b}$, $\mathbb{I}(R)=\{\bar{a}\}$);
        \item \eqref{i:H4} is equivalent to $\lnot R(a)\land \lnot R(b)\land (R(a)\lor R(b)) \vdash \bot$, which in turn is equivalent to $R(a)\lor R(b) \vdash R(a)\lor R(b)$, which holds.
    \end{itemize}
    So, in this example, it is necessary to take $n\geq 2$.
\end{remark}

\begin{example}\label{ex:forall-alpha-implies-forall-beta}
    For example, if $\alpha(x)$ and $\beta(y)$ are quantifier-free formulas, when does $\forall x\, \alpha(x)$ imply $\forall y\,\beta(y)$ modulo a given universal theory $\mathcal{T}$? 
    To be more precise, we are in the setting of \cref{r:sequent-classical-setting} with $\bar{i} = 1$, $p_1=1$, $\bar{j} = 1$, $q_1 = 1$, $\alpha(x_1), \beta(x_1)$ quantifier-free formulas.
    \Cref{r:sequent-classical-setting} tells us that $\forall x\, \alpha(x) \vdash_\mathcal{T} \forall y\,\beta(y)$ occurs precisely when there are $n \in \N$ and unary terms $g_1(y), \dots, g_n(y)$ such that
    \begin{align*}
    \bigwedge_{i = 1}^{n}\alpha(g_{i}(y))\vdash_\mathcal{T} \beta(y).
    \end{align*}
    In other words, if we know that $\alpha(x)$ holds for all $x$, the only way to prove $\beta(y)$ for an arbitrary $y$ is to instantiate $\alpha(x)$ on finitely many terms $g_1(y), \dots, g_n(y)$ depending solely on $y$ and then prove $\beta(y)$ from $\alpha(g_1(y)), \dots, \alpha(g_n(y))$.
\end{example}

\begin{example}\label{ex:forall-alpha-implies-forall-beta-forall-gamma}
    For example, if $\alpha(x)$, $\beta(y)$ and $\gamma(z)$ are quantifier-free formulas, when does $\forall x\, \alpha(x)$ imply $\forall y\,\beta(y)\lor \forall z\,\gamma(z)$ modulo a given universal theory $\mathcal{T}$? 
    To be more precise, we are in the setting of \cref{r:sequent-classical-setting} with $\bar{i} = 1$, $p_1=1$, $\bar{j} = 2$, $q_1 = q_2= 1$, $\alpha(x_1), \beta(x_1), \gamma(x_1)$ quantifier-free formulas.
    \Cref{r:sequent-classical-setting} tells us that $\forall x\, \alpha(x) \vdash_\mathcal{T} \forall y\,\beta(y)\lor \forall z\,\gamma(z)$ occurs precisely when there are $n \in \N$ and binary terms $g_1(y,z), \dots, g_n(y,z)$ such that
    \begin{align*}
    \bigwedge_{i = 1}^{n}\alpha(g_{i}(y,z))\vdash_\mathcal{T} \beta(y) \lor \gamma(z),
    \end{align*}
    where here it is important that $y$ and $z$ are distinct variables.
    Note that $\forall y\,\beta(y)\lor \forall z\,\gamma(z)$ is equivalent to $\forall y\,\forall z\,(\beta(y)\lor \gamma(z))$ (using that $y$ and $z$ are distinct).
    Then, if we know that $\alpha(x)$ holds for all $x$, the only way to prove $\beta(y)\lor \gamma(z)$ for arbitrary $y$ and $z$ is to instantiate $\alpha(x)$ on finitely many terms $g_1(y,z), \dots, g_n(y,z)$ depending solely on $y$ and $z$ and then prove $\beta(y)\lor \gamma(z)$ from $\alpha(g_1(y,z)), \dots, \alpha(g_n(y,z))$.
\end{example}

\begin{example}\label{ex:forall-alpha-forall-beta-implies-gamma}
    For example, if $\alpha_1(x)$ and $\alpha_2(y)$ are quantifier-free formulas and $\beta$ is a quantifier-free closed formula, when does $\forall x\, \alpha_1(x) \land \forall y\, \alpha_2(y)$ imply $\beta$ modulo a given universal theory $\mathcal{T}$? 
    To be more precise, we are in the setting of \cref{r:sequent-classical-setting} with $\bar{i} = 2$, $p_1 = p_2 =1$, $\bar{j} = 1$, $q_1 = 0$, $\alpha_1(x_1), \alpha_2(x_1)$ quantifier-free formulas, and $\beta$ a closed quantifier-free formula.
    \Cref{r:sequent-classical-setting} tells us that $\forall x\, \alpha_1(x) \land \forall y\, \alpha_2(y) \vdash_\mathcal{T} \beta$ occurs precisely when there are $n \in \N$, $l_1,\dots,l_n\in\{1,2\}$ and nullary terms $(g_i)_{i = 1,\dots,n}$ (i.e.\ term-definable constants) such that
    \begin{align*}
    \bigwedge_{i = 1}^{n}\alpha_{l_i}(g_{i})\vdash_\mathcal{T} \beta.
    \end{align*}
    Equivalently, this happens when there are $n_1,n_2 \in \N$, and nullary terms $(f_i)_{i= 1,\dots,n_1}$, $(f'_j)_{j = 1,\dots,n_2}$ (i.e.\ term-definable constants) such that
    \begin{align*}
    \bigwedge_{i = 1}^{n_1}\alpha_{1}(f_{i})\land \bigwedge_{j = 1}^{n_2}\alpha_{2}(f'_{j})\vdash_\mathcal{T} \beta.
    \end{align*}
    In other words, if we know that $\alpha_1(x)$ holds for all $x$ and that $\alpha_2(y)$ holds for all $y$, the only way to prove $\beta$ is to instantiate $\alpha_1(x)$ on finitely many constants $f_1, \dots, f_{n_1}$ and  $\alpha_2(y)$ on finitely many constants $f'_1, \dots, f'_{n_2}$ and then prove $\beta$ from $\alpha_1(f_1), \dots, \alpha_1(f_{n_1}), \alpha_2(f'_1), \dots, \alpha_2(f'_{n_2})$.
\end{example}

\begin{theorem} \label{t:sequent-at-S}
    Let $\P \colon \C\op \to \BA$ be a Boolean doctrine, with $\C$ small.
    Let $\bar{i}, \bar{j} \in \N$, $S,Y_1, \dots, Y_{\bar{i}}$, $Z_1,\dots, Z_{\bar{j}} \in \C$, $(\alpha_i \in \P(S\times Y_i))_{i = 1, \dots, \bar{i}}$, and $(\beta_j \in \P(S\times Z_j))_{j = 1, \dots, {\bar{j}}}$.
    The following are equivalent.
    \begin{enumerate}
        \item 
        For every propositional model $(M, \mathfrak{m}, s)$ of $\P$ at $S$, if for every $i \in \{1, \dots, \bar{i}\}$ we have that for every $y\in M(Y_i)$, $(s,y)\in \m_{S\times Y_i}(\alpha_i)$, then there is $j \in \{1, \dots, {\bar{j}}\}$ such that for every $z\in M(Z_j)$, $(s,z)\in \m_{S\times Z_j}(\beta_j)$.
        
        \item 
        There are $n \in \N$, $l_1, \dots, l_n \in \{1, \dots, \bar{i}\}$, and $(g_{i} \colon S\times \prod_{j=1}^{\bar{j}} Z_{j} \to Y_{l_i})_{i = 1, \dots, n}$ such that
        in $\P(S\times \prod_{j=1}^{\bar{j}} Z_{j})$
        \[
        \bigwedge_{i = 1}^{n}\P(\langle \pr^{S\times \Pi_p Z_p}_{S}, g_i \rangle)(\alpha_{l_i}) \leq \bigvee_{j=1}^{\bar{j}}\P(\pr^{S\times \Pi_p Z_p}_{S\times Z_j})(\beta_j).
        \]
    \end{enumerate}
\end{theorem}
\begin{proof}
    This follows from \cref{c:sequent} applied to the Boolean doctrine $\P_S$ obtained from $\P$ by adding a constant of type $S$ and from \cref{l:model-model-at-S}.
\end{proof}

\begin{remark}
    We translate \cref{t:sequent-at-S} to the classical syntactic setting.
    For this, we fix $k \in \N$.
    Let $\{s_1, \dots, s_k, x_1, x_2, \dots\}$ be a set of variables and $\mathcal{T}$ a universal theory.
    Let $\bar{i}, \bar{j} \in \N$, let $p_1, \dots, p_{\bar{i}}, q_1, \dots, q_{\bar{j}} \in \N$, for each $i = 1, \dots, \bar{i}$ let $\alpha_i(s_1, \dots, s_k, x_1, \dots, x_{p_i})$ be a quantifier-free formula, and for each $j = 1, \dots, \bar{j}$ let $\beta_j(s_1, \dots, s_k, x_1, \dots, x_{q_{j}})$ be a quantifier-free formula.
    The following are equivalent.
    \begin{enumerate}
        \item
        For every model $M$ of $\mathcal{T}$ and $c_1, \dots, c_k \in M$, the formula
        \[
        \mleft(\bigwedge_{i=1}^{\bar{i}} \forall x_1 \dots \forall x_{p_i}\, \alpha_{i}(s_1, \dots, s_k, x_1, \dots, x_{p_i})\mright) \rightarrow \mleft(\bigvee_{j=1}^{\bar{j}} \forall x_1 \dots \forall x_{q_j}\, \beta_j(s_1, \dots, s_k, x_1, \dots, x_{q_j})\mright)
        \]
        is valid in $M$ under the variable assignment $[(s_i \mapsto c_i)_{i = 1, \dots, k}]$.

       \item There are $n \in \N$, $l_1, \dots, l_n \in \{1, \dots, \bar{i}\}$ and terms $(g^{h}_{i}(s_1, \dots, s_k,x_1, \dots, x_{\Sigma_{j}q_j}))_{i \in\{ 1, \dots, n\},\,h\in\{1,\dots,p_{l_i}\}}$ such that
        \begin{align*}
        &\bigwedge_{i = 1}^{n}\alpha_{l_i}(s_1, \dots, s_k, g^{1}_{i}(s_1, \dots, s_k,x_1, \dots, x_{\Sigma_{j}q_j}), \dots, g^{p_{l_i}}_{i}(s_1, \dots, s_k,x_1, \dots, x_{\Sigma_{j}q_j})) \\
        &\vdash_\mathcal{T} \bigvee_{j = 1}^{\bar{j}}\beta_{j}(s_1, \dots, s_k, x_{1 + \Sigma_{t =1}^{j-1} q_t}, \dots, x_{\Sigma_{t =1}^{j} q_t}).
        \end{align*}
    \end{enumerate}
\end{remark}

\section{Herbrand's theorem}\label{s:Herbrand}

Let us quickly recall that a simple case of the classical version of Herbrand's theorem states that given a universal theory $\T$ and a quantifier-free formula $\alpha(x)$, the condition $\top \vdash_\mathcal{T} \exists x\, \alpha(x)$ holds if and only if there are term-definable constants $c_1,\dots, c_n$ such that $\top \vdash_\mathcal{T} \bigvee_{k = 1}^{n}\alpha(c_k)$.
Herbrand's theorem essentially allows a certain kind of reduction of first-order logic to propositional logic: while in the first condition there is an existential formula, in the second condition only quantifier-free formulas are involved.

In this section, we obtain a doctrinal version of Herbrand's theorem for formulas of quantifier alternation depth at most one.
Let $\P$ be a Boolean doctrine over a small category $\C$, and let $(\id_\C,\mathfrak{i})\colon\P\hookrightarrow\P^\EA$ be its quantifier completion; we think of $\P$ as the set of quantifier-free formulas in $\P^\EA$. We will describe explicitly the functor $\P^\EA_1$ consisting of the ``formulas in $\P^\EA$ of quantifier alternation depth $\leq 1$'', introduced in the following notation.

\begin{notation}[$\P^\EA_1$]\label{n:first-layer}
    For a Boolean doctrine $\P \colon \C \op \to \BA$, with $\C$ small, let $\P^\EA_1$ be the subfunctor of $\P^\EA$ defined as follows: for every $S\in\C$, the fiber $\P^\EA_1(S)$ is the Boolean subalgebra of $\P^\EA(S)$ generated by the union of the images of $\P(S\times Y)$ under $\fa{Y}{S}\colon \P^\EA(S\times Y)\to\P^\EA(S)$, for $Y$ ranging over objects of $\C$.
\end{notation}

Intuitively, $\P^\EA_1$ freely adds one layer of quantifier alternation depth to $\P$.
It is easy to see that reindexings of $\P^\EA$ restrict to $\P^\EA_1$, and thus $\P^\EA_1$ defines a subfunctor of $\P^\EA$.

The main results of this section are \cref{t:description-P1,t:description-P1-forall-exists}, which characterize the order in the fibers of $\P^\EA_1$ in terms of $\P$, providing a doctrinal version of Herbrand's theorem. 
To this aim, we use results from our detour in 
\cref{sec:completenes-first-layer}.

Getting back to the question ``When should a formula $(\forall x\, \alpha(x)) \land (\forall y\, \beta(y))$ be below a formula $(\forall z\, \gamma(z)) \lor (\forall w\, \delta(w))$?'' proposed at the beginning of \cref{sec:completenes-first-layer}, we will obtain the following answer:
\[(\forall x\, \alpha(x)) \land (\forall y\, \beta(y))\leq (\forall z\, \gamma(z)) \lor (\forall w\, \delta(w))\]
\[\Updownarrow\]
\begin{center}
    every model of $\P$ satisfying $\forall x\, \alpha(x)$ and $\forall y\, \beta(y)$ also satisfies $\forall z\, \gamma(z)$ or $\forall w\, \delta(w)$
\end{center}
\[\Updownarrow\]
\begin{center}
    there are terms $t_1(z,w),\dots,t_n(z,w),s_1(z,w),\dots,s_m(z,w)$ such that $\bigwedge_{i=1}^n\alpha(t_i(z,w)) \land \bigwedge_{j=1}^m\beta(s_j(z,w)) \leq \gamma(z)\lor\delta(w)$.
\end{center}
(It is important here that $z$ and $w$ are distinct variables.)

The interest of this equivalence lies in the fact that, while the first condition concerns formulas with quantifier alternation depth $\leq 1$ and thus can be written in $\P_1^\EA$, the last condition only concerns quantifier-free formulas, and thus can be written in $\P$.

\begin{lemma}\label{l:universal-join-universal}
    Let $\P \colon \C\op \to  \BA$ be a first-order Boolean doctrine, $X,Y \in \C$, $\alpha,\gamma\in\P(X)$ and $\beta\in \P(X \times Y)$.
    \[
    \alpha \leq \gamma \lor \fa{Y}{X} \beta \text{ in $\P(X)$} \Longleftrightarrow \P(\pr^{X\times Y}_X)(\alpha) \leq \P(\pr^{X\times Y}_X)(\gamma) \lor \beta \text{ in $\P(X \times Y)$}.
    \]
\end{lemma}

\begin{proof}
    \begin{align*}
        &\alpha \leq \gamma \lor \fa{Y}{X} \beta  &&\text{in $\P(X)$}\\
        & \Longleftrightarrow  \alpha \land \lnot \gamma \leq \fa{Y}{X} \beta &&\text{in $\P(X)$}\\
        &\Longleftrightarrow \P(\pr^{X\times Y}_X)(\alpha\land \lnot \gamma) \leq \beta &&\text{in $\P(X \times Y)$}\\
        &\Longleftrightarrow \P(\pr^{X\times Y}_X)(\alpha) \land \lnot \P(\pr^{X\times Y}_X)(\gamma) \leq \beta &&\text{in $\P(X \times Y)$}\\
        &\Longleftrightarrow \P(\pr^{X\times Y}_X)(\alpha) \leq \P(\pr^{X\times Y}_X)(\gamma) \lor \beta &&\text{in $\P(X \times Y)$}.\qedhere
    \end{align*}
\end{proof}

\begin{lemma}[$\forall$ distributes over $\bigvee$ with disjoint variables]\label{l:forall-join-univ-at-S}
    Let $\P \colon \C\op \to \BA$ be a first-order Boolean doctrine, let $S,X_1,\dots,X_n\in\C$, and let $\alpha_i\in \P(S\times  X_i)$ for $i=1,\dots,n$. Then, in $\P(S)$,
    \[
    \bigvee_{i=1}^n \fa{X_i}{S} \alpha_i = \fa{\Pi_j X_j}{S}\mleft(\bigvee_{i=1}^n\P(\pr^{S\times \Pi_j X_j}_{S\times X_i})(\alpha_i)\mright).
    \]
\end{lemma}

\begin{proof}
    For $n=0$, we shall check that $\bot_{\P(S)} = \fa{\tmn}{S}\bot_{\P(S)}$,
    but this follows from the fact that $\fa{\tmn}{S}$ is the right adjoint of the identity on $\P(S)$, and hence it is the identity.
    
    For $n=1$, the statement is trivially true.
    
   Let $n=2$. We shall obtain the equality
    \begin{equation} \label{eq:binary}
        \fa{X_1}{S} \alpha_1 \lor\fa{X_2}{S} \alpha_2 = \fa{X_1\times X_2}{S}(\P(\pr^{S \times X_1\times X_2}_{S \times X_1})(\alpha_1)\lor\P(\pr^{S \times X_1\times X_2}_{S \times X_2})(\alpha_2)).
    \end{equation}
    By definition of $\forall$, \eqref{eq:binary} holds if and only if for all $\beta\in \P(S)$ we have 
    \[
      \beta \leq \fa{X_1}{S} \alpha_1 \lor\fa{X_2}{S} \alpha_2 \iff \P(\pr^{S \times X_1\times X_2}_S)(\beta)\leq \P(\pr^{S \times X_1\times X_2}_{S \times X_1})(\alpha_1)\lor\P(\pr^{S \times X_1\times X_2}_{S \times X_2})(\alpha_2).
    \]
    So, let $\beta \in \P(S)$. We have the following equivalences:
    \begin{align*}
        &\P(\pr^{S \times X_1\times X_2}_S)(\beta)\leq \P(\pr^{S \times X_1\times X_2}_{S \times X_1})(\alpha_1)\lor\P(\pr^{S \times X_1\times X_2}_{S \times X_2})(\alpha_2)\\
    &\iff\P(\pr_{S \times  X_1}^{S \times X_1\times X_2})(\P(\pr^{S \times X_1}_{S})(\beta))
    \leq \P(\pr^{S \times X_1\times X_2}_{S \times X_1})(\alpha_1)\lor \P(\pr^{S \times X_1 \times X_2}_{S \times X_2})(\alpha_2)\\
    &\iff\P(\pr^{S \times X_1}_{S})(\beta)
    \leq \alpha_1\lor \fa{X_2}{S \times X_1}\P(\pr^{S \times X_1 \times X_2}_{S \times X_2})(\alpha_2)&&\text{by \cref{l:universal-join-universal}}\\
    &\iff\P(\pr^{S \times X_1}_{S})(\beta)
    \leq \alpha_1\lor \P(\pr^{S \times X_1}_{S})(\fa{X_2}{S}\alpha_2)&&\text{by Beck--Chevalley}\\
    &\iff\beta\leq \fa{X_1}{S}\alpha_1\lor \fa{X_2}{S}\alpha_2&&\text{by \cref{l:universal-join-universal}.}
    \end{align*}
    
    The statement (for an arbitrary $n$) follows by induction.
\end{proof}

\begin{lemma} \label{l:implication}
    Let $\P\colon\C\op \to \BA$ be a Boolean doctrine, let $\R \colon \cat{D} \op \to \BA$ be a first-order Boolean doctrine, and let $(M, \m) \colon \P \to \R$ be a Boolean doctrine morphism. Let $\bar{i}, \bar{j}\in\N$, let $S, Y_1, \dots, Y_{\bar{i}}, Z_1, \dots, Z_{\bar{j}} \in \C$, $(\alpha_i \in \P(S \times Y_i))_{i = 1, \dots,  \bar{i}}$ and $(\beta_j \in \P(S \times Z_j))_{j = 1, \dots,  \bar{j}}$. Suppose there are $n \in \N$, $l_1, \dots, l_n \in \{1, \dots, \bar{i}\}$ and $(g_{i} \colon S \times \prod_{j =1}^{\bar{j}}Z_j \to Y_{l_i})_{i = 1, \dots, n}$ such that in $\P(S \times \prod_{j =1}^{\bar{j}}Z_j)$
    \[
        \bigwedge_{i = 1}^{n}\P(\ple{\pr^{S\times \Pi_h Z_h}_S, g_i})(\alpha_{l_i}) \leq \bigvee_{j = 1}^{\bar{j}}\P(\pr^{S\times \Pi_h Z_h}_{S\times Z_j})(\beta_j).
    \]
    Then in $\R(M(S))$ we have
    \[
        \bigwedge_{i=1}^{\bar{i}} \fa{M(Y_i)}{M(S)}\m_{S\times Y_i}(\alpha_{i}) \leq \bigvee_{j=1}^{\bar{j}} \fa{M(Z_j)}{M(S)}\m_{S\times Z_j}(\beta_{j}).
    \]
\end{lemma}
\begin{proof}
    First, observe that in $\R(M(S))$ we have
    \begin{equation}\label{eq:index}
        \bigwedge_{i=1}^{\bar{i}} \fa{M(Y_i)}{M(S)}\m_{S\times Y_i}(\alpha_{i}) \leq \bigwedge_{i=1}^{n} \fa{M(Y_{l_i})}{M(S)}\m_{S\times Y_{l_i}}(\alpha_{l_i}).
    \end{equation}
    For every $i=1,\dots,n$, by the adjunction $\R(\pr^{M(S)\times M(Y_{l_i})}_{M(S)})\dashv \fa{M(Y_{l_i})}{M(S)}$ we have in $\R(M(S)\times M(Y_{l_i}))$
    \[
    \R(\pr^{M(S)\times M(Y_{l_i})}_{M(S)})(\fa{M(Y_{l_i})}{M(S)}(\m_{S\times Y_{l_i}}(\alpha_{l_i})))\leq \m_{S\times Y_{l_i}}(\alpha_{l_i}).
    \]
    Then, we apply $\R(\ple{\pr^{M(S)\times \Pi_h M(Z_h)}_{M(S)},M(g_i)})$ to both sides of the inequality to get in $\R(M(S) \times \prod_{j =1}^{\bar{j}}M(Z_j))$
    \begin{align}
        & \R(\pr^{M(S)\times \Pi_h M(Z_h)}_{M(S)})(\fa{M(Y_{l_i})}{M(S)}(\m_{S\times Y_{l_i}}(\alpha_{l_i}))) \notag\\
        & = \R(\ple{\pr^{M(S)\times \Pi_h M(Z_h)}_{M(S)},M(g_i)})(\R(\pr^{M(S)\times M(Y_{l_i})}_{M(S)})(\fa{M(Y_{l_i})}{M(S)}(\m_{S\times Y_{l_i}}(\alpha_{l_i}))))\notag\\
        & \leq \R(\ple{\pr^{M(S)\times \Pi_h M(Z_h)}_{M(S)},M(g_i)})(\m_{S\times Y_{l_i}}(\alpha_{l_i}))\notag\\
        & = \m_{S\times \Pi_h Z_h}(\P(\ple{\pr^{S\times \Pi_h Z_h}_{S},g_i})(\alpha_{l_i})).\label{eq:R-forall-m-P-g}
    \end{align}
    It follows that in $\R(M(S) \times \prod_{j =1}^{\bar{j}}M(Z_j))$
    \begin{align*}
        & \R(\pr^{M(S)\times \Pi_h M(Z_h)}_{M(S)})\mleft(\bigwedge_{i=1}^{\bar{i}} \fa{M(Y_i)}{M(S)}\m_{S\times Y_i}(\alpha_{i})\mright) \\
        & \leq \R(\pr^{M(S)\times \Pi_h M(Z_h)}_{M(S)})\mleft(\bigwedge_{i=1}^{n} \fa{M(Y_{l_i})}{M(S)}\m_{S\times Y_{l_i}}(\alpha_{l_i})\mright) &&\text{by \eqref{eq:index}}\\
        & \leq \bigwedge_{i=1}^{n}\m_{S\times \Pi_h Z_h}(\P(\ple{\pr^{S\times \Pi_h Z_h}_{S},g_i})(\alpha_{l_i})) &&\text{by \eqref{eq:R-forall-m-P-g}}\\
        & \leq \m_{S\times \Pi_h Z_h}\mleft(\bigvee_{j = 1}^{\bar{j}}\P(\pr^{S\times \Pi_h Z_h}_{S\times Z_j})(\beta_j)\mright) &&\text{by assumption}\\
        &= \bigvee_{j = 1}^{\bar{j}} \R(\pr^{M(S)\times \Pi_h M(Z_h)}_{M(S)\times M(Z_j)})(\m_{S\times Z_j} (\beta_j))&&\text{by naturality of $\m$}.
    \end{align*}
    By the adjunction $\R(\pr^{M(S)\times \Pi_h M(Z_h)}_{M(S)})\dashv\fa{\Pi_h M(Z_h)}{M(S)}$, we get in $\R(M(S))$
    \begin{equation} \label{eq:last-equation}
        \bigwedge_{i=1}^{\bar{i}} \fa{M(Y_i)}{M(S)}\m_{S\times Y_i}(\alpha_{i}) \leq \fa{\Pi_h M(Z_h)}{M(S)}\mleft(\bigvee_{j=1}^{\bar{j}} \R(\pr^{M(S)\times \Pi_h M(Z_h)}_{M(S)\times M(Z_j)})(\m_{S\times Z_j}(\beta_{j}))\mright).
    \end{equation}
    Then we apply \cref{l:forall-join-univ-at-S} to the right-hand side of \eqref{eq:last-equation} to conclude the proof. 
\end{proof}

\begin{theorem} \label{t:description-P1}
    Let $\P\colon\C\op \to \BA$ be a Boolean doctrine with $\C$ small, and let $(\id_\C, \mathfrak{i}) \colon \P \hookrightarrow \P^\EA$ be a quantifier completion of $\P$.
    For all $\bar i, \bar j\in\N$, $S, Y_1, \dots, Y_{\bar{i}}, Z_1, \dots, Z_{\bar{j}} \in \C$, $(\alpha_i \in \P(S \times Y_i))_{i = 1, \dots,  \bar{i}}$ and $(\beta_j \in \P(S \times Z_j))_{j = 1, \dots,  \bar{j}}$, the following are equivalent.
    \begin{enumerate}
        \item \label{i:inequality-in-free}
        In $\P^\EA(S)$ we have
        \[
            \bigwedge_{i=1}^{\bar{i}} \fa{Y_i}{S}\mathfrak{i}_{S\times Y_i}(\alpha_{i}) \leq \bigvee_{j=1}^{\bar{j}} \fa{Z_j}{S}\mathfrak{i}_{S\times Z_j}(\beta_{j}).
        \]
        \item \label{i:existence-of-terms-in-P}
        There are $n \in \N$, $l_1, \dots, l_n \in \{1, \dots, \bar{i}\}$ and $(g_{i} \colon S \times \prod_{j =1}^{\bar{j}}Z_j \to Y_{l_i})_{i = 1, \dots, n}$ such that
        in $\P(S \times \prod_{j =1}^{\bar{j}}Z_j)$
        \[
        \bigwedge_{i = 1}^{n}\P(\ple{\pr^{S\times \Pi_h Z_h}_S, g_i})(\alpha_{l_i}) \leq \bigvee_{j = 1}^{\bar{j}}\P(\pr^{S\times \Pi_h Z_h}_{S\times Z_j})(\beta_j).
        \]
    \end{enumerate}
\end{theorem}

\begin{proof}
    \eqref{i:inequality-in-free} $\Rightarrow$ \eqref{i:existence-of-terms-in-P}.
    We prove the contrapositive.
    Suppose \eqref{i:existence-of-terms-in-P} does not hold.
   
    By \cref{t:sequent-at-S}, there is a propositional model $(M, \m, s)$ of $\P$ at $S$ such that for all $i=1,\dots,\bar{i}$ and for all $y\in M(Y_i)$ we have $(s,y)\in \m_{S \times Y_i}(\alpha_{i})$ and there is no $j\in\{1,\dots,\bar{j}\}$ such that for all $z \in M(Z_j)$ $(s,z)\in \m_{S \times Z_j}(\beta_{j})$. Thus
     \begin{align}
        &\bigcap_{i=1}^{\bar{i}} \{s' \in M(S) \mid \text{for all } y\in M(Y_i),\, (s',y)\in \m_{S \times Y_i}(\alpha_{i} )\} \notag\\
        &\nsubseteq \bigcup_{j=1}^{\bar{j}} \{ s' \in M(S) \mid \text{for all } z \in M(Z_j), \, (s',z)\in \m_{S \times Z_j}(\beta_{j} )\}, \label{eq:two-lines-in-free}
    \end{align}
    because $s$ belongs to the intersection on the left of \eqref{eq:two-lines-in-free} but not to the union on the right.
    
    By the universal property of the quantifier completion, there is a unique first-order Boolean doctrine morphism $(M,\n)$ such that the following triangle commutes:
    \[
    \begin{tikzcd}
         \P\arrow[r,"{(\id_{\C},\mathfrak{i})}"',swap,hook]\arrow[dr,"{(M,\m)}",swap] & \P^\EA\arrow[d,"{(M,\n)}"',dashed,swap]\\& \mathscr{P}.
    \end{tikzcd}
    \]
    The condition in \eqref{eq:two-lines-in-free} is equivalent to 
    \begin{equation*}
        \bigcap_{i=1}^{\bar{i}}\fa{M(Y_i)}{M(S)} \m_{S \times Y_i}(\alpha_{i})\nsubseteq \bigcup_{j=1}^{\bar{j}} \fa{M(Z_j)}{M(S)}\m_{S \times Z_j}(\beta_{j}),
    \end{equation*}
    which we rewrite as
    \begin{equation*}
        \bigcap_{i=1}^{\bar{i}}\fa{M(Y_i)}{M(S)} \n_{S \times Y_i}(\mathfrak{i}_{S \times Y_i}(\alpha_{i}))\nsubseteq \bigcup_{j=1}^{\bar{j}} \fa{M(Z_j)}{M(S)}\n_{S \times Z_j}(\mathfrak{i}_{S \times Z_j}(\beta_{j})),
    \end{equation*}
    which, by the commutativity of \eqref{eq:preserves-forall} in the definition of a first-order Boolean doctrine morphism (\cref{d:ABA-morphism}) is equivalent to
    \begin{equation*}
        \bigcap_{i=1}^{\bar{i}}\n_{S }(\fa{Y_i}{S} \mathfrak{i}_{S \times Y_i}(\alpha_{i}))\nsubseteq \bigcup_{j=1}^{\bar{j}} \n_{S }(\fa{Z_j}{S}\mathfrak{i}_{S \times Z_j}(\beta_{j})),
    \end{equation*}    
    which, since ${\n}_S$ is a Boolean homomorphism, is equivalent to
    \[
        {\n}_S\mleft(\bigwedge_{i=1}^{\bar{i}} \fa{Y_i}{S}\mathfrak{i}_{S \times Y_i}(\alpha_{i})\mright) \nsubseteq {\n}_S \mleft(\bigvee_{j=1}^{\bar{j}} \fa{Z_j}{S}\mathfrak{i}_{S \times Z_j}(\beta_{j})\mright).
    \]
    Thus, since $\n_S$ is order-preserving, in $\P^\EA(S)$ we have
    \[
        \bigwedge_{i=1}^{\bar{i}} \fa{Y_i}{S}\mathfrak{i}_{S\times Y_i}(\alpha_{i}) \nleq \bigvee_{j=1}^{\bar{j}} \fa{Z_j}{S}\mathfrak{i}_{S\times Z_j}(\beta_{j}).
    \]

    \eqref{i:existence-of-terms-in-P} $\Rightarrow$ \eqref{i:inequality-in-free}. This implication follows from \cref{l:implication}.
\end{proof}

\begin{remark}\label{r:translate-description-of-order}
    We translate \cref{t:description-P1} to the classical syntactic setting.
    For this, we fix $k \in \N$.
    Let $\{s_1, \dots, s_k, x_1, x_2, \dots\}$ be a set of variables and $\mathcal{T}$ a universal theory.
    Let $\bar{i}, \bar{j}\in\N$, let $p_1, \dots, p_{\bar{i}}$, $q_1, \dots, q_{\bar{j}} \in \N$, for each $i = 1, \dots, \bar{i}$ let $\alpha_i(s_1, \dots, s_k, x_1, \dots, x_{p_i})$ be a quantifier-free formula, and for each $j = 1, \dots, \bar{j}$ let $\beta_j(s_1, \dots, s_k, x_1, \dots, x_{q_{j}})$ be a quantifier-free formula.
    The following are equivalent.
    \begin{enumerate}
        \item 
        \[
            \bigwedge_{i=1}^{\bar{i}} \forall x_1 \dots \forall x_{p_i}\, \alpha_{i}(s_1, \dots, s_k, x_1, \dots, x_{p_i}) \vdash_\mathcal{T} \bigvee_{j=1}^{\bar{j}} \forall x_1 \dots \forall x_{q_j}\, \beta_{j}(s_1, \dots, s_k, x_1, \dots, x_{q_j}).
        \]
        \item There are $n \in \N$, $l_1, \dots, l_n \in \{1, \dots, \bar{i}\}$ and terms $(g^{h}_{i}(s_1, \dots, s_k,x_1, \dots, x_{\Sigma_{j}q_j}))_{i \in\{ 1, \dots, n\},\,h\in\{1,\dots,p_{l_i}\}}$ such that
        \begin{align*}
        &\bigwedge_{i = 1}^{n}\alpha_{l_i}(s_1, \dots, s_k, g^{1}_{i}(s_1, \dots, s_k,x_1, \dots, x_{\Sigma_{j}q_j}), \dots, g^{p_{l_i}}_{i}(s_1, \dots, s_k,x_1, \dots, x_{\Sigma_{j}q_j})) \\
        &\vdash_\mathcal{T} \bigvee_{j = 1}^{\bar{j}}\beta_{j}(s_1, \dots, s_k, x_{1 + \Sigma_{t =1}^{j-1} q_t}, \dots, x_{\Sigma_{t =1}^{j} q_t}).
        \end{align*}
    \end{enumerate}
\end{remark}

The following gives a doctrinal version of Herbrand's theorem \cite{Herbrand1930} for formulas of quantifier alternation depth at most one, modulo a universal theory.

\begin{theorem} [Doctrinal version of Herbrand's theorem] \label{t:description-P1-forall-exists}
    Let $\P\colon\C\op \to \BA$ be a Boolean doctrine with $\C$ small, and let $(\id_\C, \mathfrak{i}) \colon \P \hookrightarrow \P^\EA$ be a quantifier completion of $\P$.
    For all $\bar i, \bar h,\bar j,\bar k\in\N$, $S, Y_1, \dots, Y_{\bar{i}},W_1,\dots,W_{\bar{h}}$, $Z_1, \dots, Z_{\bar{j}}$, $V_1,\dots,V_{\bar{k}} \in \C$, $(\alpha_i \in \P(S \times Y_i))_{i = 1, \dots,  \bar{i}}$, $(\gamma_h \in \P(S \times W_h))_{h = 1, \dots,  \bar{h}}$, $(\beta_j \in \P(S \times Z_j))_{j = 1, \dots,  \bar{j}}$ and $(\delta_k \in \P(S \times V_k))_{k = 1, \dots,  \bar{k}}$, the following are equivalent.
    \begin{enumerate}
        \item \label{i:inequality-in-free-with-exists}
        In $\P^\EA(S)$ we have
        \[
            \mleft(\bigwedge_{i=1}^{\bar{i}} \fa{Y_i}{S}\mathfrak{i}_{S\times Y_i}(\alpha_{i})\mright)
            \land
            \mleft(\bigwedge_{h=1}^{\bar{h}} \ex{W_h}{S}\mathfrak{i}_{S\times W_h}(\gamma_{h}) \mright)
            \leq 
            \mleft(\bigvee_{j=1}^{\bar{j}} \fa{Z_j}{S}\mathfrak{i}_{S\times Z_j}(\beta_{j})\mright)
            \lor 
            \mleft(\bigvee_{k=1}^{\bar{k}} \ex{V_k}{S}\mathfrak{i}_{S\times V_k}(\delta_{k})\mright).
        \]
        
        \item \label{i:existence-of-terms-in-P-with-exists}
        There are $n,n' \in \N$, $l_1, \dots, l_n \in \{1, \dots, \bar{i}\}$, $l_1', \dots, l_{n'}' \in \{1, \dots, \bar{k}\}$, $(g_{i} \colon S \times \prod_{j =1}^{\bar{j}}Z_j \times \prod_{h = 1}^{\bar{h}} W_h \to Y_{l_i})_{i = 1, \dots, n}$ and $(g_{k}' \colon S \times \prod_{j =1}^{\bar{j}}Z_j  \times \prod_{h = 1}^{\bar{h}} W_h \to V_{l'_k})_{k = 1, \dots, n'}$ such that
        in $\P(S \times \prod_{j =1}^{\bar{j}}Z_j\times \prod_{h =1}^{\bar{h}}W_h)$
        \begin{align*}
        & \bigwedge_{i = 1}^{n}\P(\ple{\pr^{S\times \Pi_p Z_p\times \Pi_p W_p}_S, g_i})(\alpha_{l_i})
        \land 
        \bigwedge_{h = 1}^{\bar{h}} \P(\pr^{S\times \Pi_p Z_p\times \Pi_p W_p}_{S\times W_h})(\gamma_h) \\
        & \leq \bigvee_{j = 1}^{\bar{j}}\P(\pr^{S\times \Pi_p Z_p\times \Pi_p W_p}_{S\times Z_j})(\beta_j)
        \lor
        \bigvee_{k = 1}^{n'}\P(\langle \pr^{S\times \Pi_p Z_p\times \Pi_p W_p}_S, g'_k\rangle)(\delta_{l'_k}).
        \end{align*}
    \end{enumerate}
\end{theorem}

\begin{proof}
    Item \eqref{i:inequality-in-free-with-exists} holds if and only if
    \[
        \mleft(\bigwedge_{i=1}^{\bar{i}} \fa{Y_i}{S}\mathfrak{i}_{S \times Y_i}(\alpha_{i})\mright) 
        \land 
        \mleft(\bigwedge_{k=1}^{\bar{k}} \fa{V_k}{S}\mathfrak{i}_{S \times V_k}(\lnot\delta_{k})\mright)
        \leq 
        \mleft(\bigvee_{j=1}^{\bar{j}} \fa{Z_j}{S}\mathfrak{i}_{S \times Z_j}(\beta_{j})\mright) 
        \lor
        \mleft(\bigvee_{h=1}^{\bar{h}} \fa{W_h}{S}\mathfrak{i}_{S \times W_h}(\lnot\gamma_{h})\mright).
    \]    
   By \cref{t:description-P1}, this is equivalent to the existence of $n,n' \in \N$, $l_1, \dots, l_n \in \{1, \dots, \bar{i}\}$, $l_1', \dots, l_{n'}' \in \{1, \dots, \bar{k}\}$, $(g_{i} \colon S \times \prod_{j =1}^{\bar{j}}Z_j \times \prod_{h = 1}^{\bar{h}} W_h \to Y_{l_i})_{i = 1, \dots, n}$ and $(g_{k}' \colon S \times \prod_{j =1}^{\bar{j}}Z_j  \times \prod_{h = 1}^{\bar{h}} W_h \to V_{l'_k})_{k = 1, \dots, n'}$ such that in $\P(S \times \prod_{j =1}^{\bar{j}}Z_j\times \prod_{h = 1}^{\bar{h}} W_h)$
   \begin{align*}
        & \bigwedge_{i = 1}^{n}\P(\ple{\pr^{S\times \Pi_p Z_p\times \Pi_p W_p}_S, g_i})(\alpha_{l_i})
        \land 
        \bigwedge_{k = 1}^{n'} \P(\langle \pr^{S\times \Pi_p Z_p \times \Pi_p W_p}_S, g'_k\rangle)(\lnot\delta_{l'_k})\\
        & \leq 
        \bigvee_{j =1}^{\bar{j}} \P(\pr^{S\times \Pi_p Z_p \times \Pi_p W_p}_{S\times Z_j})(\beta_j)
        \lor
        \bigvee_{h = 1}^{\bar{h}} \P(\pr^{S\times \Pi_p Z_p \times \Pi_p W_p}_{S\times W_h})(\lnot\gamma_h),
    \end{align*}
    which is equivalent to \eqref{i:existence-of-terms-in-P-with-exists}.
\end{proof}

Roughly speaking, in the simple case with $\bar{i} = \bar{h} = \bar{j}= \bar{k} = 1$ and no free variables, this means that, given a universal theory $\mathcal{T}$, the statement
\[
    \forall y\,\alpha(y)
    \land
    \exists w \,\gamma(w)
    \vdash_\mathcal{T} 
    \forall z\,\beta(z)
    \lor 
    \exists v\, \delta(v)
\]
holds (where $\alpha$, $\gamma$, $\beta$ and $\delta$ are quantifier-free) if and only if (supposing the variables $w$ and $z$ to be distinct) there are terms $g_1(z,w), \dots, g_n(z,w)$ and $g'_1(z,w), \dots, g'_{n'}(z,w)$ such that
\[
    \bigwedge_{i = 1}^{n}\alpha(g_i(z,w))
    \land 
    \gamma(w)
    \vdash_\mathcal{T}
    \beta(z)
    \lor
    \bigvee_{k = 1}^{n'}\delta(g'_k(z,w)).
\]
In the more general case with arbitrary $\bar{i}, \bar{h}, \bar{j}, \bar{k} \in \N$, \cref{t:description-P1-forall-exists} means that, given a universal theory $\mathcal{T}$, the statement
\[
    \mleft(\bigwedge_{i=1}^{\bar{i}} \forall y_i\,\alpha_{i}(y_i)\mright)
    \land
    \mleft(\bigwedge_{h=1}^{\bar{h}} \exists w_h \,\gamma_{h}(w_h) \mright)
    \vdash_\mathcal{T} 
    \mleft(\bigvee_{j=1}^{\bar{j}} \forall z_j\,\beta_{j}(z_j)\mright)
    \lor 
    \mleft(\bigvee_{k=1}^{\bar{k}} \exists v_k\, \delta_{k}(v_k)\mright)
\]
holds (where all the $\alpha_i$'s, $\gamma_h$'s, $\beta_j$'s and $\delta_k$'s are quantifier-free) if and only if (supposing the variables $w_1,\dots,w_{\bar h}, z_1, \dots, z_{\bar j}$ to be distinct) there are $n,n' \in \N$, $l_1, \dots, l_n \in \{1, \dots, \bar{i}\}$, $l_1', \dots, l_{n'}' \in \{1, \dots, \bar{k}\}$, $g_{1} (\vec z,\vec w), \dots, g_{n} (\vec z,\vec w)$ and $g'_{1} (\vec z,\vec w), \dots ,g'_{n'} (\vec z,\vec w)$
such that
\[
    \bigwedge_{i = 1}^{n}\alpha_{l_i}(g_i(\vec z,\vec w))
    \land 
    \bigwedge_{h = 1}^{\bar{h}}\gamma_h(w_h)
    \vdash_\mathcal{T}
    \bigvee_{j = 1}^{\bar{j}} \beta_j(z_j)
    \lor
    \bigvee_{k = 1}^{n'}\delta_{l'_k}(g'_k(\vec z,\vec w)).
\]

\begin{example}[Herbrand's theorem]\label{ex:herbrand}
    Let $\mathcal{T}$ be a universal theory and let $\delta(x_1, \dots, x_v)$
    be a quantifier-free formula, where $v \in \N$.
    By \cref{t:description-P1-forall-exists}, the condition
    \[
        \top \vdash_\mathcal{T} \exists x_1 \dots \exists x_v \, \delta(x_1, \dots, x_v)
    \]
    holds if and only if there are $n \in \N$ and lists of nullary terms $(g^{1}_{k}, \dots, g^v_k)_{k =1, \dots, n}$
    such that
    \[
        \top \vdash_\mathcal{T} \bigvee_{k = 1}^{n}\delta(g^{1}_{k}, \dots, g^{v}_{k}).
    \]
    This is the well-known restriction to existential sentences of Herbrand's theorem. (We recall that Herbrand's theorem in its full generality covers all first-order formulas.)
\end{example}

\begin{remark} \label{r:super-description-congruence-P-forall}
The results obtained so far allow us to characterize when a Boolean combination of universal formulas implies another Boolean combination of universal formulas, i.e., to completely describe the order on $\P^\EA_1$.
Indeed, it is enough to rewrite the premise in disjunctive normal form and the conclusion in conjunctive normal form.
For example, the inequality
\[
\forall y\, \alpha(y) \lor \lnot \forall w\, \gamma(w) \leq \forall z\, \beta(z) \land \lnot \forall v \, \delta (v)
\]
(in which the left- and right-hand sides are already in disjunctive and conjunctive normal form, respectively)
is equivalent to the system of four inequalities
\[
    \begin{cases}
        \forall y\, \alpha(y) \leq \forall z\, \beta(z)\\
        \forall y\, \alpha(y) \leq \exists v\, \lnot\delta (v)\\
        \exists w\, \lnot\gamma(w) \leq \forall z\, \beta(z) \\
        \exists w\, \lnot\gamma(w) \leq \exists v\, \lnot\delta (v),
    \end{cases}
\]
for each of which we can apply \cref{t:description-P1-forall-exists}.
\end{remark}

\begin{example}[Failure of the usual reduction to prenex normal form]
    In the usual reduction to prenex normal form, one uses the fact that a formula of the form
    \[
    (\exists x \,\alpha(x))\lor\beta
    \]
    is equivalent to
    \[
    \exists x\, (\alpha(x)\lor\beta).
    \]
    However, in the version of classical first-order logic whose semantics admits the empty structure, only one of the two implications always holds, namely
    \[
    \exists x \, (\alpha(x)\lor\beta) \vdash_{\mathcal{T}} (\exists x \,\alpha(x)) \lor\beta,
    \]
    with $\T$ any theory.
    
    We characterize when the converse holds in the following setting: let $\mathcal{T}$ be a universal theory,
    let $\alpha(x)$ be a quantifier-free formula with free variable $x$ and let $\beta$ be a closed quantifier-free formula.
    We write $\beta(x)$ when we consider $x$ to be a dummy variable of $\beta$.
    We show that the converse direction, namely
    \begin{equation} \label{eq:pnf}
        (\exists x \,\alpha(x)) \lor\beta \vdash_{\mathcal{T}} \exists x \, (\alpha(x)\lor\beta(x)),
    \end{equation}
    holds in the empty context if and only if $\beta\vdash_\T \bot$ or the language has at least one constant symbol.
    (Note that, if the language has a constant symbol, then every model is nonempty.)
    
    First, we rewrite \eqref{eq:pnf} in a simpler form:
    \begin{align}
        (\exists x \,\alpha(x))\lor\beta\vdash_{\mathcal{T}}\exists x\, (\alpha(x)\lor\beta(x)) &\iff \exists x \,\alpha(x)\vdash_{\mathcal{T}}\exists x\, (\alpha(x)\lor\beta(x))\text{ and } \beta\vdash_{\mathcal{T}}\exists x \,(\alpha(x)\lor\beta(x)) \notag\\
        & \iff \beta\vdash_{\mathcal{T}}\exists x\, (\alpha(x)\lor\beta(x)) \label{eq:last}.
    \end{align}
    By \cref{t:description-P1-forall-exists}, seeing $\beta$ as $\exists() \,\beta$, \eqref{eq:last} is equivalent to the existence of $n\in\N$ and nullary terms $(f_i)_{i=1,\dots,n}$ such that
    \[
    \beta \vdash_{\mathcal{T}} \bigvee_{i=1}^n(\alpha(f_i)\lor\beta).
    \]
    This trivially holds if $\beta\vdash_{\mathcal{T}}\bot$ (take $n= 0$). If instead $\beta\not\vdash_{\mathcal{T}}\bot$ (so that we cannot take $n = 0$), this holds if and only if there is at least one nullary term in the language, i.e., if the language has at least one constant symbol.
\end{example}

\section{The construction of the free QA-one-step Boolean doctrine} \label{s:construction}
In \cite{AbbadiniGuffanti}, we introduced \emph{QA-one-step Boolean doctrines}.
This notion is meant to capture abstractly the pairs $(\P_0, \P_1)$ where $\P_0$ is the fragment of quantifier-free formulas modulo some theory $\T$, and $\P_1$ (which extends $\P_0$) is the fragment of formulas of quantifier alternation depth at most $1$ modulo $\T$.

In this section, we describe the free QA-one-step Boolean doctrine $(\P, \Free_1^\P)$ over a Boolean doctrine $\P$ (over a small category), i.e., the value at $\P$ of the left adjoint of the forgetful functor $(\P_0, \P_1) \mapsto \P_0$ from QA-one-step Boolean doctrines to Boolean doctrines.

This section does \emph{not} depend on \cref{sec:completenes-first-layer,s:Herbrand}; in particular, it does not rely on our discussion on universal ultrafilters and Herbrand's theorem.

In the next section, we will combine the results of this section with the doctrinal version of Herbrand's theorem to show that the free QA-one-step Boolean doctrine over a given Boolean doctrine $\P$ coincides with $(\P, \P_1^\EA)$.

\subsection{QA-one-step Boolean doctrines}

\begin{definition}[QA-one-step Boolean doctrine, {\cite[Def.~3.11]{AbbadiniGuffanti}}] \label{d:one-step-doctrine}
    A \emph{QA-one-step Boolean doctrine}\footnote{In future work, we plan to show that every QA-one-step Boolean doctrine over a small category embeds in some first-order Boolean doctrine. (\cref{t:section-6} below entails that this happens at least for the free QA-one-step Boolean doctrines.) This will characterize QA-one-step Boolean doctrines as the pairs $(\P_0, \P_1)$ appearing in some \emph{QA-stratification} (\cite[p.~12]{AbbadiniGuffanti}) of some first-order Boolean doctrine.  To reach these aims, we plan to adapt the proof of G\"odel's completeness theorem to show that every QA-one-step Boolean doctrine has enough models (= morphisms to the QA-one-step subset doctrine) to order-separate the formulas. This will establish that the notion of QA-one-step Boolean doctrine is ``the right one'', i.e., no axiom is missing.} over a category $\C$ with finite products is a componentwise injective natural transformation $i \colon \P_0 \hookrightarrow \P_1$ between two functors $\P_0, \P_1 \colon \C\op \to  \BA$ with the following properties.
    \begin{enumerate}
        \item \label{i:one-step-universal}
        (One-step universal)
        For every projection $\pr^{X\times Y}_X \colon X \times Y \to X$ in $\C$ and every $\beta \in \P_0(X \times Y)$, there is an element $\fa{Y}{X} \beta \in \P_{1}(X)$ such that, for all $\alpha \in \P_{1}(X)$, 
        \[
            \alpha \leq \fa{Y}{X} \beta \text{ in $\P_{1}(X)$} \Longleftrightarrow \P_{1}(\pr^{X\times Y}_X)(\alpha) \leq i_{X \times Y}(\beta) \text{ in $\P_{1}(X \times Y)$}.
        \]
        (Note that one such element $\fa{Y}{X} \beta$ is unique.)
        
        In other words, for every $X,Y\in\C$, $\P_1(\pr^{X\times Y}_X)\colon\P_1(X)\to\P_1(X\times Y)$ has a right $i_{X\times Y}$-relative adjoint, namely $\fa{Y}{X}\colon \P_0(X\times Y)\to \P_1(X)$ (\cite[Def.~2.2]{Ulmer1968}).
        
        \item \label{i:one-step-Beck--Chevalley}
        (One-step Beck--Chevalley)
        For every morphism $f\colon X'\to X$ in $\C$ and every $Y \in \C$, the following diagram in $\Pos$ commutes.
        \[
            \begin{tikzcd}
                {\P_0(X\times Y)} & {\P_{1}(X)}  \\
                {\P_0(X'\times Y)} & {\P_{1}(X')} 
                \arrow["{\P_0(f\times\id_{Y})}", from=1-1, to=2-1, swap]
                \arrow["{\P_{1}(f)}", from=1-2, to=2-2]
                \arrow["{\fa{Y}{X'}}"', from=2-1, to=2-2]
                \arrow["{\fa{Y}{X}}", from=1-1, to=1-2]
            \end{tikzcd}
        \]
        \item (One-step generation)
        For all $X \in \C$, $\P_{1}(X)$ is generated as a Boolean algebra by the union of the images of the functions $\fa{Y}{X} \colon \P_0(X \times Y) \to \P_{1}(X)$ for $Y$ ranging over objects of $\C$.
    \end{enumerate}

\end{definition}

Since the natural transformation $i \colon \P_0 \hookrightarrow \P_1$ is componentwise injective, we will often suppose it to be a componentwise inclusion; then, we will avoid mentioning $i$, or we will present $i \colon \P_0 \hookrightarrow \P_1$ as a pair $(\P_0, \P_1)$.

\begin{example}
    Let $\LT^\T$ be the syntactic doctrine of first-order formulas modulo a theory $\T$ in a first-order language $(\F, \mathbb{P})$ as in \cref{fbf}. For every $n \in \N$ and context $\vec{x}$, we define the Boolean algebra $\LT^\T_n(\vec{x})$ as the set of equivalence classes of first-order formulas with quantifier alternation depth modulo $\T$ at most $n$. For every $n$, the assignment $\LT^\T_n$ can be extended to a functor
    \[
        \LT^\T_n\colon\Ctx_\F\op\to \BA   
    \]
    by defining the reindexing $\LT^\T_n(\vec t(\vec x))\colon \LT^\T_n(\vec{y})\to\LT^\T_n(\vec{x})$ along the tuple of terms $\vec{t}(\vec{x})\colon \vec x\to \vec{y}$ as the restriction of $\LT^\T(\vec t(\vec x))$.
    
    The pair $(\LT^\T_0,\LT^\T_1)$ is the motivating example of a QA-one-step Boolean doctrine.
    We note that, for every $n \in \N$, the pair $(\LT^\T_n,\LT^\T_{n +1})$ is a QA-one-step Boolean doctrine, too.
\end{example}

\begin{remark}
    For every Boolean doctrine $\P \colon \C\op \to \BA$ with $\C$ small, the pair $(\P, \P_1^\EA)$ is a QA-one-step Boolean doctrine. 
    (See \cref{n:first-layer} for the definition of $\P_1^\EA$.)
    We will see that it is the free QA-one-step Boolean doctrine over $\P$.
\end{remark}

\begin{definition}[Morphism of QA-one-step Boolean doctrines, {\cite[Def.~3.15]{AbbadiniGuffanti}}]
    Let $\P_0 \hookrightarrow \P_1$ and $\R_0 \hookrightarrow \R_1$ be QA-one-step Boolean doctrines over $\C$ and $\D$, respectively.
    A \emph{morphism of QA-one-step Boolean doctrines} from $(\P_0, \P_1)$ to $(\R_0, \R_1)$ is a triple $(M, j_0,j_1)$ with $M \colon \C \to \D$ a functor preserving finite products and $j_0 \colon \P_0 \to \R_0 \circ M\op$ and $j_1 \colon \P_1 \to \R_1 \circ M\op$ natural transformations making the following diagram commute
    \[
    \begin{tikzcd}
        \P_0 \arrow[hook]{r}{} \arrow[swap]{d}{j_0} & \P_1 \arrow{d}{j_1}\\
        \R_0\circ M\op \arrow[swap,hook]{r}{} & \R_1\circ M\op
    \end{tikzcd}
    \]
    and such that, for all $X, Y \in \C$, the following diagram in $\Pos$ commutes.
    \[
        \begin{tikzcd}[column sep = 5em]
            \P_0(X \times Y) \arrow{r}{\fa{Y}{X}} \arrow[swap]{d}{(j_0)_{X \times Y}} & \P_{1}(X)\arrow{d}{(j_1)_{X}}\\
            \R_{0}(M(X) \times M(Y)) \arrow[swap]{r}{\fa{M(Y)}{M(X)}} & \R_{1}(M(X))
        \end{tikzcd}
    \]
\end{definition}

We let $\QA_{\leq 1}$ denote the category of
QA-one-step Boolean doctrines and morphisms.

\subsection{The construction of the fibers (of the free QA-one-step Boolean doctrine)}\label{s:fibers-free}

We now exhibit how to freely add one layer of quantifier alternation depth to a Boolean doctrine over a small base category via generators and relations.
Let $\P \colon \C\op \to \BA$ be a Boolean doctrine, with $\C$ small, and let $S \in \C$.
Let $B_S$ be the free Boolean algebra over 
\begin{equation*}
A_S \coloneqq \bigsqcup_{Y \in \C}\P(S \times Y),
\end{equation*}
and, for each $Y \in \C$ and $\alpha \in \P(S \times Y)$, let $\fa{Y}{S} \alpha$ denote the image of $\alpha$ under the free map $A_S \to B_S$.
Let $\sim_S$ be the Boolean congruence on $B_S$ generated by the following relations:
for all $\bar{i}, \bar{j} \in \N$, $Y_1, \dots, Y_{\bar{i}}, Z_1,\dots,Z_{\bar{j}}\in\C$, $(\alpha_i \in \P(S \times Y_i))_{i = 1, \dots,  \bar{i}}$, $(\beta_j \in \P(S \times Z_j))_{j = 1, \dots,  \bar{j}}$ we impose the relation
\begin{equation} \label{eq:relation}
  \mleft[\bigwedge_{i=1}^{\bar{i}} \fa{Y_i}{S}\alpha_{i} \mright] \leq \mleft[\bigvee_{j=1}^{\bar{j}} \fa{Z_j}{S}\beta_{j}\mright]  
\end{equation}
in $B_S / {\sim_S}$ whenever there are
$n\in \N$, $l_1, \dots, l_n \in \{1, \dots, \bar{i}\}$ and $(g_{i} \colon S \times \prod_{j =1}^{\bar{j}}Z_j \to Y_{l_i})_{i = 1, \dots, n}$ such that in $\P(S \times \prod_{j =1}^{\bar{j}}Z_j)$
\begin{equation}\label{eq:def-sim-S}
    \bigwedge_{i = 1}^{n}\P(\ple{\pr^{S\times \Pi_h Z_h}_S, g_i})(\alpha_{l_i}) \leq \bigvee_{j = 1}^{\bar{j}}\P(\pr^{S\times \Pi_h Z_h}_{S\times Z_j})(\beta_j).
\end{equation}

\begin{notation}[Free one-step construction on objects]\label{n:free-on-obj}
    We let 
    \[
    \mathrm{Free}_1^{\P}(S)
    \]
    denote the quotient $B_S / {{\sim}_S}$. 
\end{notation}

\begin{lemma} \label{l:vdash-is-entailment-relation}
    Let $\P \colon \C\op \to \BA$ be a Boolean doctrine, with $\C$ small, and let $S \in \C$.
    Consider the following binary relation $\vdash$ on the set of finite subsets of the set $A_S \coloneqq \bigsqcup_{Y \in \C}\P(S \times Y)$:
    \begin{align*}
        &\{\alpha_i \in \P(S \times Y_i)\}_{i=1,\dots,\bar i}\vdash\{\beta_j \in \P(S \times Z_j)\}_{j=1,\dots,\bar j}\\
        &\iff \text{there are } n  \in \N, l_1, \dots, l_n \in \{1, \dots, \bar{i}\}  \text{ and } (g_{i} \colon S \times \prod_{j =1}^{\bar{j}}Z_j \to Y_{l_i})_{i = 1, \dots, n}\text{ such that in  }\P(S \times \prod_{j =1}^{\bar{j}}Z_j)\\
        & \phantom{\iff}  \bigwedge_{i = 1}^{n}\P(\ple{\pr^{S\times \Pi_h Z_h}_S, g_i})(\alpha_{l_i}) \leq \bigvee_{j = 1}^{\bar{j}}\P(\pr^{S\times \Pi_h Z_h}_{S\times Z_j})(\beta_j).
    \end{align*}
    The relation $\vdash$ is an entailment relation (in the sense of \cref{d:entailment-relation}).
\end{lemma}

\begin{proof}
    Let us prove that $\vdash$ is an entailment relation, i.e.\ it satisfies reflexivity, monotonicity and the cut rule.

    Let us prove reflexivity, i.e.\ that for every $Y\in\C$ and $\alpha\in\P(S\times Y)$, $\{\alpha\}\vdash\{\alpha\}$.
    Take $n = 1$, $l_1=1$ and $g_1=\pr^{S\times Y}_Y \colon S\times Y\to Y$.
    Then,
    \[\P(\ple{\pr^{S\times Y}_S, \pr^{S\times Y}_Y})(\alpha) = \alpha = \P(\pr^{S\times Y}_{S\times Y})(\alpha).
    \]

    Let us prove monotonicity: suppose $\{\alpha_i \in \P(S \times Y_i)\}_{i=1,\dots,\bar i}\vdash\{\beta_j \in \P(S \times Z_j)\}_{j=1,\dots,\bar j}$, and let us show $\{\alpha_i\}_{i=1,\dots,\bar i}\cup \{\gamma_h \in \P(S \times W_h)\}_{h=1,\dots,\bar h}\vdash\{\beta_j \}_{j=1,\dots,\bar j}\cup\{\delta_k \in \P(S \times V_k)\}_{k=1,\dots,\bar k}$. The latter trivially holds by taking the same $n$, $l_i$'s and (the precomposition with the projections forgetting the factors $V_k$ of the) $g_i$'s witnessing $\{\alpha_i \}_{i=1,\dots,\bar i}\vdash\{\beta_j \}_{j=1,\dots,\bar j}$.

    Let us prove the cut rule: let $\alpha_i \in \P(S \times Y_i)$ for $i \in \{1, \dots, \bar{i}\}$, $\beta_j \in \P(S \times Z_j)$ for $j \in \{1, \dots, \bar j\}$ and $\sigma \in \P(S \times W)$ and let us suppose that
    \begin{equation} \label{eq:hypothesis-of-cut}
        \{\alpha_i\}_{i=1,\dots,\bar i} \vdash\{\sigma\} \cup \{\beta_j \}_{j=1,\dots,\bar j} \quad \text{ and } \quad \{\alpha_i\}_{i=1,\dots,\bar i} \cup \{\sigma\} \vdash \{\beta_j \}_{j=1,\dots,\bar j}.
    \end{equation}
    Then, by the first condition in \eqref{eq:hypothesis-of-cut} there are
    $n  \in \N$, $l_1, \dots, l_n \in \{1, \dots, \bar{i}\}$  and $(g_{i} \colon S \times W\times \prod_{j =1}^{\bar{j}}Z_j \to Y_{l_i})_{i = 1, \dots, n}$ such that in $\P(S \times W\times \prod_{j =1}^{\bar{j}}Z_j)$
        \begin{equation}\label{eq:transitivity-1}
        \bigwedge_{i = 1}^{n}\P(\ple{\pr^{S\times W\times \Pi_h Z_h}_S, g_i})(\alpha_{l_i}) \leq \P(\pr^{S\times W\times \Pi_h Z_h}_{S\times W})(\sigma)\lor \bigvee_{j = 1}^{\bar{j}}\P(\pr^{S\times W\times \Pi_h Z_h}_{S\times Z_j})(\beta_j).
        \end{equation}
    Moreover, by the second condition in \eqref{eq:hypothesis-of-cut} there are 
    $n',m  \in \N$, $l'_1, \dots, l'_{n'} \in \{1, \dots, \bar{i}\}$, $(g'_{i} \colon S \times \prod_{j =1}^{\bar{j}}Z_j \to Y_{l'_i})_{i = 1, \dots, n'}$ and $(t_{k} \colon S \times \prod_{j =1}^{\bar{j}}Z_j \to W)_{k = 1, \dots, m}$ such that in $\P(S \times \prod_{j =1}^{\bar{j}}Z_j)$
        \begin{equation}\label{eq:a'-s-b}
        \underbracket{\bigwedge_{i = 1}^{n'}\P(\ple{\pr^{S\times \Pi_h Z_h}_S, g'_i})(\alpha_{l'_i})}_{\eqqcolon A'}\land \bigwedge_{k = 1}^{m}\underbracket{\P(\ple{\pr^{S\times \Pi_h Z_h}_S, t_k})(\sigma)}_{\eqqcolon S_k} \leq \underbracket{\bigvee_{j = 1}^{\bar{j}}\P(\pr^{S\times \Pi_h Z_h}_{S\times Z_j})(\beta_j)}_{\eqqcolon B}.
        \end{equation}
    For every $k=1,\dots, m$, by applying $\P(\ple{\pr^{S\times\Pi_jZ_j}_S,t_k,\pr^{S\times\Pi_jZ_j}_{\Pi_jZ_j}})$ to \eqref{eq:transitivity-1}, we get in $\P(S\times \prod_{j=1}^{\bar j}Z_j)$ 
    \begin{equation*}
        \underbracket{\bigwedge_{i = 1}^{n}\P(\ple{\pr^{S\times \Pi_h Z_h}_S, g_i\circ \ple{\pr^{S\times\Pi_hZ_h}_S,t_k,\pr^{S\times\Pi_hZ_h}_{\Pi_hZ_h}}})(\alpha_{l_i})}_{\eqqcolon A_k} 
        \leq \P(\ple{\pr^{S\times \Pi_h Z_h}_{S},t_k})(\sigma)\lor \bigvee_{j = 1}^{\bar{j}}\P(\pr^{S\times \Pi_h Z_h}_{S\times Z_j})(\beta_j),
        \end{equation*}
        i.e.
        \begin{equation}\label{eq_a-s-b}
            A_k\leq S_k\lor B.
        \end{equation}
        Let us prove that $A'\land\bigwedge_{k = 1}^{m}A_k\leq B$:
        \begin{align*}
            A'\land\bigwedge_{k = 1}^{m}A_k&\leq A'\land\bigwedge_{k = 1}^{m} (S_k\lor B)&&\text{by \eqref{eq_a-s-b}}\\
            &=(A'\land\bigwedge_{k = 1}^{m}S_k)\lor(A'\land B)\\
            &\leq B&&\text{by \eqref{eq:a'-s-b}},
        \end{align*}
        as desired.
        Let us rewrite the inequality $A'\land\bigwedge_{k = 1}^{m}A_k\leq B$ as:
        \begin{align*}
        &\bigwedge_{i = 1}^{n'}\P(\ple{\pr^{S\times \Pi_h Z_h}_S, g'_i})(\alpha_{l'_i})\land \bigwedge_{k = 1}^{m}\bigwedge_{i = 1}^{n}\P(\ple{\pr^{S\times \Pi_h Z_h}_S, g_i\circ \ple{\pr^{S\times\Pi_hZ_h}_S,t_k,\pr^{S\times\Pi_hZ_h}_{\Pi_hZ_h}}})(\alpha_{l_i})\\
        &\leq \bigvee_{j = 1}^{\bar{j}}\P(\pr^{S\times \Pi_h Z_h}_{S\times Z_j})(\beta_j),
        \end{align*}
    which proves $\{\alpha_i \}_{i=1,\dots,\bar i}\vdash\{\beta_j \}_{j=1,\dots,\bar j}$, as desired.    
\end{proof}

\begin{theorem}\label{t:characterization-fibers}
    Let $\P \colon \C\op \to \BA$ be a Boolean doctrine, with $\C$ small, and let $S \in \C$. Let $\bar{i}, \bar{j} \in \N$, $Y_1, \dots, Y_{\bar{i}}, Z_1,\dots,Z_{\bar{j}}\in\C$, $(\alpha_i \in \P(S \times Y_i))_{i = 1, \dots,  \bar{i}}$, $(\beta_j \in \P(S \times Z_j))_{j = 1, \dots,  \bar{j}}$. The following are equivalent.
    \begin{enumerate}
        \item\label{i:meets-under-joins} In $B_S / {\sim_S}$
        \begin{equation*} 
          \mleft[\bigwedge_{i=1}^{\bar{i}} \fa{Y_i}{S}\alpha_{i} \mright] \leq \mleft[\bigvee_{j=1}^{\bar{j}} \fa{Z_j}{S}\beta_{j}\mright].
        \end{equation*}
        
        \item \label{i:existence-of-things}
        There are $n  \in \N$, $l_1, \dots, l_n \in \{1, \dots, \bar{i}\}$  and $(g_{i} \colon S \times \prod_{j =1}^{\bar{j}}Z_j \to Y_{l_i})_{i = 1, \dots, n}$ such that in $\P(S \times \prod_{j =1}^{\bar{j}}Z_j)$
        \begin{equation*}
        \bigwedge_{i = 1}^{n}\P(\ple{\pr^{S\times \Pi_h Z_h}_S, g_i})(\alpha_{l_i}) \leq \bigvee_{j = 1}^{\bar{j}}\P(\pr^{S\times \Pi_h Z_h}_{S\times Z_j})(\beta_j).
        \end{equation*}

    \end{enumerate}
\end{theorem}
\begin{proof}
    By \cref{l:vdash-is-entailment-relation}, together with the version for Boolean algebras of the fundamental theorem of entailment relations, which we postpone to \cref{s:entailment} (\cref{t:fundamental-theorem-of-entailment-relations-boolean-algebras}).
\end{proof}

\begin{remark}
    \Cref{t:characterization-fibers} is all we need to understand the order on $B_S/{\sim_S}$. Indeed, knowing when a finite meet of generators is below a finite join of generators of a Boolean algebra, we can deduce when any given Boolean combination of generators is below any given Boolean combination of generators.
\end{remark}

\begin{remark}
    Comparing \cref{t:characterization-fibers,t:description-P1}, we can already see that, for every $S \in \C$, $\Free_1^\P(S)$ and $\P_1^\EA(S)$ are isomorphic Boolean algebras.
    Since $(\P,\P_1^\EA)$ is a QA-one-step Boolean doctrine, one could already deduce that, via the aforementioned fiberwise isomorphisms, also the assignment $S \mapsto\Free_1^\P(S)$ can be extended to a QA-one-step Boolean doctrine $(\P, \Free_1^\P)$.
    However, in this section, we give a direct description of $\Free_1^\P$ as a Boolean doctrine and a direct proof that $(\P,\Free^\P_1)$ is indeed a QA-one-step Boolean doctrine, without relying on our results about $\P_1^\EA$.
\end{remark}

\begin{remark}
    If desired, one could obtain a more explicit construction of $\Free_1^\P(S)$ as the poset reflection of a preordered set with underlying set the finite powerset of the finite powerset of two copies of $A_S=\bigsqcup_{Y \in \C} \P(S \times Y)$; the idea being that one copy of $\alpha \in \P(S \times Y)$ means $\fa{Y}{S} \alpha$, the second copy means $\lnot \fa{Y}{S} \alpha$, and a finite set of finite sets of these elements expresses a bounded lattice combination of them in disjunctive normal form.
    See \cref{r:entailment-DNF,r:two-copies} for more details.
\end{remark}

\subsection{The construction gives a QA-one-step Boolean doctrine}\label{s:qa-free}

\begin{notation}[Free one-step construction on morphisms] \label{n:one-step-on-morphisms}
    Let $\P\colon\C\op\to\BA$ be a Boolean doctrine, with $\C$ small.
    For a morphism $f \colon S \to S'$ in $\C$ we denote by
    \[
    \Free_1^{\P}(f) \colon \mathrm{Free}_1^{\P}(S') \to \mathrm{Free}_1^{\P}(S)
    \]
    the unique Boolean homomorphism that, for all $Y\in \C$ and $\alpha \in \P(S' \times Y)$, maps the generator $[\fa{Y}{S'}\alpha]$ to $[\fa{Y}{S} (\P(f \times \id_{Y})(\alpha))]$. (The next lemma shows that such a homomorphism indeed exists and is unique.)
\end{notation}

\begin{lemma}
    Let $\P\colon\C\op\to\BA$ be a Boolean doctrine with $\C$ small and $f \colon S \to S'$ a morphism in $\C$. There is a unique Boolean homomorphism 
    \[
    \mathrm{Free}_1^{\P}(S') \to \mathrm{Free}_1^{\P}(S)
    \]
    that, for all $Y\in \C$ and $\alpha \in \P(S' \times Y)$, maps the generator $[\fa{Y}{S'}\alpha]$ to $[\fa{Y}{S} \P(f \times \id_{Y})(\alpha)]$.
\end{lemma}
\begin{proof}
    We shall prove that there is a unique Boolean homomorphism $\mathrm{Free}_1^{\P}(S') \to \mathrm{Free}_1^{\P}(S)$ such that the following diagram commutes
    \[
    \begin{tikzcd}
        A_{S'} \arrow[swap]{d}{\alpha \in \P(S' \times Y) \mapsto \P(f \times \id_Y)(\alpha) \in \P(S \times Y)} \arrow{r}{\iota_{S'}} & B_{S'}/{\sim_{S'}} = \Free_1^{\P}(S')\arrow[dashed, shift left=13pt]{d}\\
         A_S\arrow{r}{\iota_S}   & B_{S}/{\sim_{S}} = \Free_1^{\P}(S),
    \end{tikzcd}
    \]
    where $\iota_S$ maps $\alpha\in\P(S \times Y)$ to $[\fa{Y}{S}\alpha]\in\Free_1^{\P}(S)$ (and $\iota_{S'}$ has a similar definition).

    By the universal property of algebras given by generators and relations, proving this amounts to proving that, roughly speaking, every relation of the form \eqref{eq:relation} defining $\Free_1^\P(S')$ holds in $\Free_1^\P(S)$ in an appropriate sense.
    More precisely, let $n, \bar{i}, \bar{j} \in \N$, $l_1, \dots, l_n \in \{1, \dots, \bar{i}\}$, $Z_1,\dots,Z_{\bar{j}}, Y_1, \dots, Y_{\bar{i}}\in\C$, $(\alpha_i \in \P(S' \times Y_i))_{i = 1, \dots,  \bar{i}}$, $(\beta_j \in \P(S' \times Z_j))_{j = 1, \dots,  \bar{j}}$ and $(g_{i} \colon S' \times \prod_{j =1}^{\bar{j}}Z_j \to Y_{l_i})_{i = 1, \dots, n}$ be such that in $\P(S' \times \prod_{j =1}^{\bar{j}}Z_j)$
    \begin{equation}\label{eq:def-sim-S-thm}
        \bigwedge_{i = 1}^{n}\P(\ple{\pr^{S'\times \Pi_h Z_h}_{S'}, g_i})(\alpha_{l_i}) \leq \bigvee_{j = 1}^{\bar{j}}\P(\pr^{S'\times \Pi_h Z_h}_{S'\times Z_j})(\beta_j).
    \end{equation}
    (This induces the relation $[\bigwedge_{i=1}^{\bar{i}} \fa{Y_i}{S'}\alpha_{i}] \leq [\bigvee_{j=1}^{\bar{j}} \fa{Z_j}{S'}\beta_{j}]$ in $\Free_1^\P(S')$.)
    We shall prove that
    \begin{equation} \label{eq:relations-after-f}
        \mleft[\bigwedge_{i=1}^{\bar{i}} \fa{Y_i}{S}\P(f\times \id_{Y_i})(\alpha_{i}) \mright] \leq \mleft[\bigvee_{j=1}^{\bar{j}} \fa{Z_j}{S}\P(f\times \id_{Z_j})(\beta_{j})\mright]
    \end{equation}
    in $\Free_1^\P(S) = B_S/{\sim_S}$.
    By the definition of $\Free_1^\P(S)$ via generators and relations, to prove \eqref{eq:relations-after-f} it suffices to take
    $n, \bar{i}, \bar{j} \in \N$, $l_1, \dots, l_n \in \{1, \dots, \bar{i}\}$, $Z_1,\dots,Z_{\bar{j}}, Y_1, \dots, Y_{\bar{i}}\in\C$, $(\P(f\times\id_{Y_i})(\alpha_i )\in \P(S \times Y_i))_{i = 1, \dots,  \bar{i}}$, $(\P(f\times\id_{Z_j})(\beta_j) \in \P(S \times Z_j))_{j = 1, \dots,  \bar{j}}$ and $(g_{i}\circ(f\times \id_{\Pi_j Z_j}) \colon S \times \prod_{j =1}^{\bar{j}}Z_j \to Y_{l_i})_{i = 1, \dots, n}$ and show that in $\P(S \times \prod_{j =1}^{\bar{j}}Z_j)$
    \begin{equation*}
        \bigwedge_{i = 1}^{n}\P(\ple{\pr^{S\times \Pi_h Z_h}_{S}, g_{i}\circ(f\times \id_{\Pi_h Z_h})})(\P(f\times\id_{Y_{l_i}})(\alpha_{l_i})) \leq \bigvee_{j = 1}^{\bar{j}}\P(\pr^{S\times \Pi_h Z_h}_{S\times Z_j})(\P(f\times\id_{Z_j})(\beta_j)),
    \end{equation*}
    or, equivalently,
    \begin{align*}
        \bigwedge_{i = 1}^{n}\P(\ple{f\circ\pr^{S\times \Pi_h Z_h}_S,g_{i}\circ(f\times \id_{\Pi_h Z_h})})(\alpha_{l_i}) \leq \bigvee_{j = 1}^{\bar{j}}
        \P(f\times\pr^{\Pi_h Z_h}_{Z_j})(\beta_j).
    \end{align*}
    The latter holds:
    \begin{align*}
        &\bigwedge_{i = 1}^{n}\P(\ple{f\circ\pr^{S\times \Pi_h Z_h}_S,g_{i}\circ(f\times \id_{\Pi_h Z_h})})(\alpha_{l_i})\\
        &=\bigwedge_{i = 1}^{n}\P(\ple{\pr^{S'\times \Pi_h Z_h}_{S'}\circ (f\times \id_{\Pi_h Z_h}), g_i\circ (f\times \id_{\Pi_h Z_h})})(\alpha_{l_i}) \\
        &=\bigwedge_{i = 1}^{n}\P(f\times \id_{\Pi_h Z_h})(\P(\ple{\pr^{S'\times \Pi_h Z_h}_{S'}, g_i})(\alpha_{l_i}))\\
        & \leq \bigvee_{j = 1}^{\bar{j}}\P(f\times \id_{\Pi_h Z_h})(\P(\pr^{S'\times \Pi_h Z_h}_{S'\times Z_j})(\beta_j)) && \text{applying $\P(f\times \id_{\Pi_h Z_h})$ to \eqref{eq:def-sim-S-thm}}\\
        &= \bigvee_{j = 1}^{\bar{j}}    \P(f\times\pr^{\Pi_h Z_h}_{Z_j})(\beta_j).
    \end{align*}
    This concludes the proof.    
\end{proof}

\begin{lemma}
    Let $\P\colon\C\op\to\BA$ be a Boolean doctrine with $\C$ small; the assignment on objects described in \cref{n:free-on-obj} and the assignment on morphisms described in \cref{n:one-step-on-morphisms} define a functor, denoted $\Free_1^\P$.
\end{lemma}

\begin{proof}
    Let us prove that it preserves the identities.
    Let $S \in \C$.
    The morphism $\Free_1^\P(\id_S)$ is the unique Boolean homomorphism that, for every $Y \in \C$ and $\alpha \in \P(S \times Y)$, maps $[\fa{Y}{S} \alpha]$ to $[\fa{Y}{S} \P(\id_S \times \id_Y)(\alpha)]$, i.e.\ to $[\fa{Y}{S} (\alpha)]$.
    Since also the identity on $\Free_1^\P(S)$ satisfies this property, $\Free_1^\P(\id_S)$ is the identity on $\Free_1^\P(S)$.
    This proves that $\Free_1^\P$ preserves the identities.

    Let us prove that it preserves composition.
    Let $f \colon S \to S'$ and $f' \colon S' \to S''$ be morphisms in $\C$.
    The morphism $\Free_1^\P(f' \circ f)$ is the unique Boolean homomorphism that, for every $Y \in \C$ and $\alpha \in \P(S'' \times Y)$, maps $[\fa{Y}{S''} \alpha]$ to $[\fa{Y}{S} \P((f' \circ f) \times \id_Y)(\alpha)]$.
    Let us prove that also $\Free_1^\P(f) \circ \Free_1^\P(f')$ satisfies this property.
    Let $Y \in \C$ and $\alpha \in \P(S'' \times Y)$.
    The function $\Free_1^\P(f')$ maps the element $[\fa{Y}{S''} \alpha]$ to $[\fa{Y}{S'} \P(f' \times \id_Y)(\alpha)]$; the latter element is mapped by $\Free_1^\P(f)$ to $[\fa{Y}{S} \P(f \times \id_{Y})(\P(f' \times \id_Y)(\alpha))]$, i.e.\ to $[\fa{Y}{S} (\P((f' \circ f) \times \id_Y)(\alpha))]$, as desired.
\end{proof}

Recall that our goal is to define the free QA-one-step Boolean doctrine over $\P$. We have already defined the functor $\Free_1^{\P}$, but we need to define the connecting natural transformation $\forall\tmn \colon \P \hookrightarrow \Free_1^{\P}$. For every $S\in\C$, let us define the component at $S$ as
\begin{align}\label{eq:nat-transf-qa}
    \P(S) & \longrightarrow \Free_1^{\P}(S)\\
    \alpha & \longmapsto [\fa{\tmn}{S} \alpha].\notag
\end{align}
This will allow us to define a Boolean doctrine morphism $(\id_\C,\forall\tmn) \colon \P \hookrightarrow \Free_1^{\P}$.

\begin{lemma}
    Let $\P\colon \C\op \to \BA$ be a Boolean doctrine, with $\C$ small.
    For all $S \in \C$, the map 
    \begin{align*}
        [\fa{\tmn}{S}-] \colon \P(S) & \longrightarrow \Free_1^{\P}(S)\\
        \alpha & \longmapsto [\fa{\tmn}{S} \alpha]
    \end{align*}
    is an injective Boolean homomorphism.
\end{lemma}

\begin{proof}
    Let us fix $S\in\C$.
    
    Let us first prove that $[\fa{\tmn}{S}-]$ is order-preserving and order-reflecting (and hence injective).
    By \cref{t:characterization-fibers}, for all $\alpha, \beta \in \P(S)$ we have $[\fa{\tmn}{S} \alpha] \leq [\fa{\tmn}{S} \beta]$ if and only if there are $n \in \N$, $(g_i \colon S \times \tmn \to \tmn)_{i = 1, \dots, n}$ such that in $\P(S \times \tmn)$
    \[
    \bigwedge_{i=1}^n \P(\langle \pr_S^{S \times \tmn}, g_i \rangle)(\alpha) \leq \P(\pr_{S \times \tmn}^{S \times \tmn})(\beta).
    \]
    This condition amounts to existence of $n \in \N$ such that $\bigwedge_{i = 1}^n \alpha \leq \beta$. In turn, this condition is equivalent to $\alpha \leq \beta$.
    
    Let us now show that it is a Boolean homomorphism.
    For all $\alpha,\beta\in\P(S)$, the element $[\fa{\tmn}{S}(\alpha\land\beta)]$ is a lower bound of $[\fa{\tmn}{S}\alpha]$ and $[\fa{\tmn}{S}\beta]$; by \cref{t:characterization-fibers} we have $[\fa{\tmn}{S}\alpha]\land[\fa{\tmn}{S}\beta]\leq[\fa{\tmn}{S}(\alpha\land\beta)]$, since $\alpha\land\beta\leq\alpha\land\beta$ in $\P(S)$ (by taking $n=2$, $l_1=1$, $l_2=2$, $g_1=g_2={!_S}$).
    Thus binary meets are preserved, and analogously so are binary joins (by taking $n=1$, $l_1=1$, $g_1={!_S}$).
    Moreover, $[\fa{\tmn}{S}\top_{\P(S)}]=\top_{\Free_1^\P(S)}$: indeed, by \cref{t:characterization-fibers} we have $\top_{\Free_1^\P(S)}\leq[\fa{\tmn}{S}\top_{\P(S)}]$, since $\top_{\P(S)}\leq\top_{\P(S)}$ in $\P(S)$ (with $\bar i=0$).
    Similarly, $[\fa{\tmn}{S}\bot_{\P(S)}]=\bot_{\Free_1^\P(S)}$, since $\bot_{\P(S)}\leq\bot_{\P(S)}$ in $\P(S)$ (with $\bar j=0$, by taking $n=1$, $l_1=1$, $g_1={!_S}$).
    Since $[\fa{\tmn}{S}-]$ is a function between Boolean algebras preserving finite meets and finite joins, it is a Boolean homomorphism.
\end{proof}

\begin{lemma}
    Let $\P\colon \C\op \to \BA$ be a Boolean doctrine, with $\C$ small.
    $\forall\tmn \colon \P \to \Free_1^{\P}$ is a natural transformation.
    Therefore, $(\id_\C,\forall\tmn) \colon \P \to \Free_1^{\P}$ is a componentwise injective Boolean doctrine morphism.
\end{lemma}

\begin{proof}
    Let $f \colon S \to S'$ be a morphism in $\C$.
    We shall prove that the following diagram commutes.
    \[
    \begin{tikzcd}[column sep = 5em]
        \P(S') \arrow{r}{[\fa{\tmn}{S'}-]} \arrow[swap]{d}{\P(f)} & \Free_1^\P(S') \arrow{d}{\Free_1^\P(f)}\\
        \P(S) \arrow[swap]{r}{[\fa{\tmn}{S}-]} & \Free_1^\P(S)
    \end{tikzcd}
    \]
    Let $\alpha \in \P(S')$.
    We have
    \begin{equation*}
        \Free_1^\P(f)([\fa{\tmn}{S'}\alpha]) = [\fa{\tmn}{S}\P(f \times \id_\tmn)(\alpha)] = [\fa{\tmn}{S}\P(f)(\alpha)]. \qedhere
    \end{equation*}
\end{proof}

\begin{proposition}\label{p:free-is-one-step}
    For every Boolean doctrine $\P \colon \C\op \to \BA$ with $\C$ small, $\forall\tmn\colon \P \hookrightarrow \Free_1^\P$ (as in \eqref{eq:nat-transf-qa}) is a QA-one-step Boolean doctrine. For every $X,Y\in \C$, for every $\beta\in \P(X\times Y)$, its quantification $\fa{Y}{X}\beta\in\Free_1^\P(X)$ is given by $[\fa{Y}{X}\beta]$.
\end{proposition}
\begin{proof}
    We prove the one-step universal condition.
    Let $\pr^{X\times Y}_X \colon X \times Y \to X$ be a projection in $\C$, let $\beta \in \P(X \times Y)$.
    Let us prove that, for all $\alpha \in \Free_1^\P(X)$,
    \begin{equation}\label{eq:qa-one-step-univers-condition}
        \alpha \leq [\fa{Y}{X} \beta] \text{ in $\Free_1^\P(X)$} \Longleftrightarrow \Free_1^\P(\pr^{X\times Y}_X)(\alpha) \leq [\fa{\tmn}{X \times Y}\beta] \text{ in $\Free_1^\P(X \times Y)$}.
    \end{equation}
    Every $\alpha\in \Free_1^\P(X)$ is a finite join of finite meets of generators and negations of generators of the Boolean algebra $\Free_1^\P(X)$.
    It is easy to see that the set of $\alpha$'s satisfying \eqref{eq:qa-one-step-univers-condition} is closed under finite joins.
    Therefore, it is enough to prove \eqref{eq:qa-one-step-univers-condition} for a finite meet of generators and negations of generators of $\Free_1^\P(X)$.
    Let then $\alpha = \bigwedge_{i = 1}^{\bar i} [\fa{W_i}{X}\gamma_i] \land \bigwedge_{j = 1}^{\bar j}  \lnot [\fa{Z_j}{X}\delta_j]$, with $\gamma_i\in\P(X\times W_i)$ for $i=1,\dots,{\bar i}$ and $\delta_j\in\P(X\times Z_j)$ for $j=1,\dots,{\bar j}$.
    We shall then prove
    \begin{align*}
        &\bigwedge_{i = 1}^{\bar i} [\fa{W_i}{X}\gamma_i] \leq [\fa{Y}{X} \beta] \lor \bigvee_{j = 1}^{\bar j} [\fa{Z_j}{X}\delta_j] \text{ in $\Free_1^\P(X)$} \\
        &\Longleftrightarrow \bigwedge_{i = 1}^{\bar i} \Free_1^\P(\pr^{X\times Y}_X)([\fa{W_i}{X}\gamma_i]) \leq [\fa{\tmn}{X \times Y}\beta]  \lor \bigvee_{j = 1}^{\bar j} \Free_1^\P(\pr^{X\times Y}_X)([\fa{Z_j}{X}\delta_j])\text{ in $\Free_1^\P(X \times Y)$},
    \end{align*}
    i.e.
    \begin{align}
        &\bigwedge_{i = 1}^{\bar i} [\fa{W_i}{X}\gamma_i] \leq [\fa{Y}{X} \beta] \lor \bigvee_{j = 1}^{\bar j} [\fa{Z_j}{X}\delta_j] \text{ in $\Free_1^\P(X)$} \label{eq:equivalence-1}\\
        &\Longleftrightarrow \bigwedge_{i = 1}^{\bar i} [\fa{W_i}{X\times Y}\P(\pr^{X\times Y\times W_i}_{X\times W_i})(\gamma_i)] \notag\\
        &\phantom{\Longleftrightarrow}\leq [\fa{\tmn}{X \times Y}\beta]  \lor \bigvee_{j = 1}^{\bar j} [\fa{Z_j}{X\times Y}\P(\pr^{X\times Y\times Z_j}_{X\times Z_j})(\delta_j)]\text{ in $\Free_1^\P(X \times Y)$}\label{eq:equivalence-2}.
    \end{align}
    By \cref{t:characterization-fibers}, \eqref{eq:equivalence-1} is equivalent to the existence of 
    $n  \in \N$, $l_1, \dots, l_n \in \{1, \dots, \bar{i}\}$  and $(g_{i} \colon X \times Y\times  \prod_{j =1}^{\bar{j}}Z_j \to W_{l_i})_{i = 1, \dots, n}$ such that in $\P(X \times Y\times \prod_{j =1}^{\bar{j}}Z_j)$
        \begin{equation*}
        \bigwedge_{i = 1}^{n}\P(\ple{\pr^{X\times Y\times \Pi_h Z_h}_X, g_i})(\gamma_{l_i}) \leq \P(\pr^{X\times Y\times \Pi_h Z_h }_{X\times Y})(\beta)\lor \bigvee_{j = 1}^{\bar{j}}\P(\pr^{X\times Y\times \Pi_h Z_h}_{X\times Z_j})(\delta_j),
        \end{equation*}
    which we rewrite as
    \begin{align*}
        &\bigwedge_{i = 1}^{n}\P(\ple{\pr^{X\times Y\times \tmn\times \Pi_h Z_h}_{X\times Y}, g_i})(\P(\pr^{X\times Y\times W_{l_i}}_{X\times W_{l_i}})(\gamma_{l_i})) \\
        &\leq \P(\pr^{X\times Y\times \tmn\times \Pi_h Z_h }_{X\times Y\times \tmn})(\beta)\lor \bigvee_{j = 1}^{\bar{j}}\P(\pr^{X\times Y\times \tmn\times \Pi_h Z_h}_{X\times Y\times Z_j})(\P(\pr^{X\times Y\times Z_j}_{X\times Z_j})(\delta_j)),
        \end{align*}
    which allows us to see, by using \cref{t:characterization-fibers} again, that such existence is equivalent to \eqref{eq:equivalence-2}.

    The one-step Beck--Chevalley condition trivially follows from the definition of $\Free_1^\P$ on morphisms.

    The one-step generation follows from the definition of $\Free_1^\P(X)$ and the fact that, by the one-step universal condition above, the element $[\fa{Y}{X} \alpha]$ is the quantification $\fa{Y}{X}[\fa{\tmn}{X \times Y}\alpha]$ of $[\fa{\tmn}{X \times Y}\alpha]$.
\end{proof}

\subsection{The universal property}\label{s:univ-prop}

To prove that our free construction is the free QA-one-step Boolean doctrine, we need a slight modification of \cref{l:implication}, namely \cref{l:implication-QA} below.
In this version, instead of a first-order Boolean doctrine, we have a QA-one-step Boolean doctrine. The proof is essentially the same as the proof of \cref{l:implication}. To prove \cref{l:implication-QA}, similar variations of \cref{l:universal-join-universal,l:forall-join-univ-at-S} may be used, which we state as \cref{l:universal-join-universal-QA,l:forall-join-univ-at-S-QA}, respectively. They, too,  are proved similarly to their first-order versions.

In passing, we note that the versions for first-order Boolean doctrines (\cref{l:universal-join-universal,l:forall-join-univ-at-S,l:implication}) follow from the ones for QA-one-step Boolean doctrines (\cref{l:universal-join-universal-QA,l:forall-join-univ-at-S-QA,l:implication-QA}) by viewing a first-order Boolean doctrine $\R$ as a QA-one-step Boolean doctrine $(\R, \R)$.

\begin{lemma}\label{l:universal-join-universal-QA}
    Let $(\P_0,\P_1)$ be a QA-one-step Boolean doctrine over a category $\C$. Let $X,Y \in \C$, $\alpha,\gamma\in\P_1(X)$ and $\beta\in \P_0(X \times Y)$. Then
    \[
    \alpha \leq \gamma \lor \fa{Y}{X} \beta \text{ in $\P_1(X)$} \Longleftrightarrow \P_1(\pr^{X\times Y}_X)(\alpha) \leq \P_1(\pr^{X\times Y}_X)(\gamma) \lor \beta \text{ in $\P_1(X \times Y)$}.
    \]
\end{lemma}
\begin{proof}
    The proof follows line by line the one of \cref{l:universal-join-universal}, using the relative adjunction \eqref{i:one-step-universal} in \cref{d:one-step-doctrine} in place of the adjunction \eqref{i:h3} in \cref{d:bool_ex_doc} and observing that the elements to be quantified lie in the $0$-th layer of a QA-one-step Boolean doctrine.
\end{proof}

\begin{lemma}[$\forall$ distributes over $\bigvee$ with disjoint variables]\label{l:forall-join-univ-at-S-QA}
    Let $(\P_0,\P_1)$ be a QA-one-step Boolean doctrine over a category $\C$, let $S,X_1,\dots,X_n\in\C$, and let $\alpha_i\in \P_0(S\times  X_i)$ for $i=1,\dots,n$. Then in $\P_1(S)$
    \[
    \bigvee_{i=1}^n \fa{X_i}{S} \alpha_i = \fa{\Pi_j X_j}{S}\mleft(\bigvee_{i=1}^n\P_0(\pr^{S\times \Pi_j X_j}_{S\times X_i})(\alpha_i)\mright).
    \]
\end{lemma}
\begin{proof}
    The proof follows line by line the one of \cref{l:forall-join-univ-at-S}, using the relative adjunction \eqref{i:one-step-universal} in \cref{d:one-step-doctrine} in place of the adjunction \eqref{i:h3} in \cref{d:bool_ex_doc} and observing that the elements to be quantified lie in the $0$-th layer of a QA-one-step Boolean doctrine.
    In the case $n = 2$, one has to use \cref{l:universal-join-universal-QA} in place of \cref{l:universal-join-universal}.
\end{proof}

\begin{lemma} \label{l:implication-QA}
    Let $\P\colon\C\op \to \BA$ be a Boolean doctrine, let $(\R_0,\R_1)$ be a QA-one-step Boolean doctrine over a category $\D$, and let $(M, \m) \colon \P \to \R_0$ be a Boolean doctrine morphism. Let $\bar{i}, \bar{j}\in\N$, let $S, Y_1, \dots, Y_{\bar{i}}, Z_1, \dots, Z_{\bar{j}} \in \C$, $(\alpha_i \in \P(S \times Y_i))_{i = 1, \dots,  \bar{i}}$ and $(\beta_j \in \P(S \times Z_j))_{j = 1, \dots,  \bar{j}}$. Suppose there are $n \in \N$, $l_1, \dots, l_n \in \{1, \dots, \bar{i}\}$ and $(g_{i} \colon S \times \prod_{j =1}^{\bar{j}}Z_j \to Y_{l_i})_{i = 1, \dots, n}$ such that in $\P(S \times \prod_{j =1}^{\bar{j}}Z_j)$
    \[
        \bigwedge_{i = 1}^{n}\P(\ple{\pr^{S\times \Pi_h Z_h}_S, g_i})(\alpha_{l_i}) \leq \bigvee_{j = 1}^{\bar{j}}\P(\pr^{S\times \Pi_h Z_h}_{S\times Z_j})(\beta_j).
    \]
    Then in $\R_1(M(S))$ we have
    \[
        \bigwedge_{i=1}^{\bar{i}} \fa{M(Y_i)}{M(S)}\m_{S\times Y_i}(\alpha_{i}) \leq \bigvee_{j=1}^{\bar{j}} \fa{M(Z_j)}{M(S)}\m_{S\times Z_j}(\beta_{j}).
    \]
\end{lemma}
\begin{proof}
    The proof follows line by line the one of \cref{l:implication}, using the relative adjunction \eqref{i:one-step-universal} in \cref{d:one-step-doctrine} in place of the adjunction \eqref{i:h3} in \cref{d:bool_ex_doc} and observing that the elements to be quantified lie in the $0$-th layer of a QA-one-step Boolean doctrine. In the last step, one has to use \cref{l:forall-join-univ-at-S-QA} in place of \cref{l:forall-join-univ-at-S}.
\end{proof}

There is a forgetful functor $U$ from the category $\QA_{\leq 1}$ of QA-one-step Boolean doctrines to the category of Boolean doctrines mapping a morphism $(M,j_0,j_1)\colon(\P_0, \P_1)\to (\R_0,\R_1)$ to $(M,j_0)\colon\P_0\to \R_0$. In what follows, we shall prove that the restriction of $U$ to the subcategories in which the base categories are small has a left adjoint, defined thanks to our free one-step construction.

\begin{theorem}\label{t:univ-prop-counit-QA}
    For every Boolean doctrine $\P \colon \C\op \to \BA$ with $\C$ small, every QA-one-step Boolean doctrine $(\R_0, \R_1)$ over a category $\D$, and every Boolean doctrine morphism $(M, \m) \colon \P \to \R_0$, there is a unique QA-one-step Boolean doctrine morphism $(\bar{M},\bar{\m}_0,\bar{\m}_1) \colon (\P, \Free_1^\P) \to (\R_0, \R_1)$  such that $(\bar{M},\bar{\m}_0)=(M,\m)$.
    \[
        \begin{tikzcd}[column sep = 4em, row sep = 4em]
            \P \arrow[swap]{dr}{(M, \m)} \arrow{r}{\id_\P}& \P = U(\P, \Free_1^\P) \arrow[dashed]{d}{(\bar{M}, \bar{\m}_0) = U(\bar{M}, \bar{\m}_0, \bar{\m}_1)} & (\P, \Free_1^\P) \arrow[dashed]{d}{\exists! (\bar{M}, \bar{\m}_0, \bar{\m}_1)}\\
            & \R_0 = U(\R_0, \R_1) & (\R_0, \R_1)
        \end{tikzcd}
    \]
\end{theorem}

\begin{proof}
    Let $\P \colon \C\op \to \BA$ be a Boolean doctrine with $\C$ small, let  $(\R_0, \R_1)$ be a QA-one-step Boolean doctrine over a category $\D$, and let $(M, \m) \colon \P \to \R_0$ be a Boolean doctrine morphism.
    We shall prove that there is a unique natural transformation $\bar{\m}_1\colon \Free_1^\P \to \R_1\circ M\op$ such that $({M},{\m},\bar{\m}_1) \colon (\P, \Free_1^\P) \to (\R_0, \R_1)$ is a QA-one-step Boolean doctrine morphism.

    Let us prove the existence.
    For every $S\in\C$, let us show that there is a unique Boolean homomorphism $(\bar\m_1)_S\colon\Free_1^\P(S)\to \R_1(M(S))$ that, for every $Y\in\C$ and $\alpha\in\P(S \times Y)$, maps the generator $[\fa{Y}{S}\alpha]$ to $\fa{M(Y)}{M(S)}\m_{S\times Y}(\alpha)$; in other words, such that the following diagram in $\Set$ commutes.
    \[
    \begin{tikzcd}[column sep = 6em, row sep = 6em]
        A_S \coloneqq \bigsqcup_{Y \in \C} \P(S \times Y) \arrow{r}{\alpha \mapsto [\fa{Y}{S}\alpha]} \arrow[swap]{rd}{\alpha \mapsto \fa{M(Y)}{M(S)}\m_{S\times Y}(\alpha)} & \Free_1^\P(S) \arrow[dashed]{d}{\exists! (\bar\m_1)_S}\\
        & \R_1(M(S))
    \end{tikzcd}
    \]

    For this, we use the universal property of the Boolean algebra $\Free_1^\P(S)$ as given via generators and relations.
    It is enough to show that the relations defining $\Free_1^\P(S)$ hold in $\R_1(M(S))$ once the generators are interpreted in $\R_1(M(S))$ rather than in $\Free_1^\P(S)$.
    We recall the definition of $\Free_1^\P$ on an object $S\in\C$ from \cref{n:free-on-obj} and the paragraph that precedes it.
    Let $\sim_S$ be the Boolean congruence on $B_S$ generated by the following relations: for all $n, \bar{i}, \bar{j} \in \N$, $l_1, \dots, l_n \in \{1, \dots, \bar{i}\}$, $Z_1,\dots,Z_{\bar{j}}$, $Y_1, \dots, Y_{\bar{i}}\in\C$, $(\alpha_i \in \P(S \times Y_i))_{i = 1, \dots,  \bar{i}}$, $(\beta_j \in \P(S \times Z_j))_{j = 1, \dots,  \bar{j}}$ and $(g_{i} \colon S \times \prod_{j =1}^{\bar{j}}Z_j \to Y_{l_i})_{i = 1, \dots, n}$ such that in $\P(S \times \prod_{j =1}^{\bar{j}}Z_j)$
    \begin{equation*}
        \bigwedge_{i = 1}^{n}\P(\ple{\pr^{S\times \Pi_h Z_h}_S, g_i})(\alpha_{l_i}) \leq \bigvee_{j = 1}^{\bar{j}}\P(\pr^{S\times \Pi_h Z_h}_{S\times Z_j})(\beta_j),
    \end{equation*}
    we impose the relation 
    \[
        \mleft[\bigwedge_{i=1}^{\bar{i}} \fa{Y_i}{S}\alpha_{i} \mright] \leq \mleft[\bigvee_{j=1}^{\bar{j}} \fa{Z_j}{S}\beta_{j}\mright]
    \]
    in $B_S / {\sim_S}=\Free_1^{\P}(S)$.

    Let $n, \bar{i}, \bar{j} \in \N$, $l_1, \dots, l_n \in \{1, \dots, \bar{i}\}$, $Z_1,\dots,Z_{\bar{j}}$, $Y_1, \dots, Y_{\bar{i}}\in\C$, $(\alpha_i \in \P(S \times Y_i))_{i = 1, \dots,  \bar{i}}$, $(\beta_j \in \P(S \times Z_j))_{j = 1, \dots,  \bar{j}}$ and $(g_{i} \colon S \times \prod_{j =1}^{\bar{j}}Z_j \to Y_{l_i})_{i = 1, \dots, n}$ be such that in $\P(S \times \prod_{j =1}^{\bar{j}}Z_j)$
    \begin{equation*}
        \bigwedge_{i = 1}^{n}\P(\ple{\pr^{S\times \Pi_h Z_h}_S, g_i})(\alpha_{l_i}) \leq \bigvee_{j = 1}^{\bar{j}}\P(\pr^{S\times \Pi_h Z_h}_{S\times Z_j})(\beta_j).
    \end{equation*}
    By \cref{l:implication-QA}, in $\R_1(M(S))$ we have
    \[
        \bigwedge_{i=1}^{\bar{i}} \fa{M(Y_i)}{M(S)}\m_{S\times Y_i }(\alpha_{i}) \leq \bigvee_{j=1}^{\bar{j}} \fa{M(Z_j)}{M(S)}\m_{S\times Z_j}(\beta_{j}).
    \]
    This proves the existence of a unique Boolean homomorphism $(\bar\m_1)_S\colon\Free_1^\P(S)\to \R_1(M(S))$ such that, for all $Y\in\C$ and $\alpha\in\P(S \times Y)$,
    \begin{equation} \label{eq:generator-to-something}
        (\bar\m_1)_S[\fa{Y}{S}\alpha]=\fa{M(Y)}{M(S)}\m_{S\times Y}(\alpha).
    \end{equation}

    We now prove that $\bar\m_1 \colon \Free_1^\P\to \R_1\circ M\op$ is a natural transformation, i.e., such that, for every morphism $f \colon S\to S'$ in $\C$, the following diagram commutes.
    \[
    \begin{tikzcd}
        \Free_1^\P(S')\arrow[d,"\Free_1^\P(f)"']\arrow[r, "(\bar\m_1)_{S'}"]    &     \R_1(M(S'))\arrow[d,"\R_1(M(f))"]\\
        \Free_1^\P(S)\arrow[r, "(\bar\m_1)_{S}"]   &     \R_1(M(S))
    \end{tikzcd}
    \]
    It is enough to show that the two composites agree on each generator, i.e., on each element of the form $[\fa{Y}{S'}\alpha]$ for $Y \in \C$ and $\alpha \in \P(S' \times Y)$.
    Let $Y \in \C$ and $\alpha \in \P(S' \times Y)$.
    We have
    \begin{align*}
        & \R_1(M(f))((\bar{\m}_1)_{S'}([\fa{Y}{S'}\alpha]))\\
        & = \R_1(M(f))(\fa{M(Y)}{M(S')}\m_{S'\times Y}(\alpha))&& \text{by \eqref{eq:generator-to-something}}\\
        & = \fa{M(Y)}{M(S)}\R_0(M(f)\times \id_{M(Y)})(\m_{S'\times Y}(\alpha))&& \text{by one-step Beck--Chevalley}\\
        & = \fa{M(Y)}{M(S)}\R_0(M(f\times \id_{Y}))(\m_{S'\times Y}(\alpha))&& \text{since $M$ preserves finite products}\\
        & = \fa{M(Y)}{M(S)}\m_{S\times Y}(\P(f\times \id_{Y})(\alpha))&& \text{by naturality of $\m$}\\
        & = (\bar{\m}_1)_S([\fa{Y}{S}\P(f \times \id_Y)(\alpha)]) && \text{by \eqref{eq:generator-to-something}}\\
        & = (\bar{\m}_1)_S(\Free_1^\P(f)([\fa{Y}{S'}\alpha])) && \text{by def.\ of $\Free_1^\P(f)$, in \cref{n:one-step-on-morphisms}.}
    \end{align*}
    This proves that $\bar\m_1$ is a natural transformation.

    To show that $(M, \m, \bar{\m}_1)$ is a QA-one-step Boolean doctrine morphism, we are left to show that for every $S,Y\in\C$ the following diagrams commute:
    \begin{equation}\label{diag:prop-univ-free}
    \begin{tikzcd}[column sep = 5em]
        \P(S) \arrow[hook]{r}{[\fa{\tmn}{S}-]} \arrow[swap]{d}{\m_S} & \Free_1^\P(S) \arrow{d}{(\bar\m_1)_S}&  \P(S \times Y) \arrow{r}{\fa{Y}{S}} \arrow[swap]{d}{\m_{S \times Y}} & \Free_1^\P(S)\arrow{d}{(\bar\m_1)_{S}}\\
        \R_0(M(S))\arrow[swap,hook]{r}{} & \R_1(M(S))&\R_{0}(M(S) \times M(Y)) \arrow[swap]{r}{\fa{M(Y)}{M(S)}} & \R_{1}(M(S))
    \end{tikzcd}
    \end{equation}
    To prove the commutativity of the first diagram, let $\alpha\in\P(S)$. We have:
    \begin{align*}
        (\bar\m_1)_S([\fa{\tmn_\C}{S}\alpha])&=\fa{\tmn_\D}{M(S)}\m_S(\alpha)&&\text{by \eqref{eq:generator-to-something}}\\
        &=\m_S(\alpha).
    \end{align*}

    To prove the commutativity of the second diagram, let $\alpha\in\P(S\times Y)$. We have:
    \begin{align*}
        (\bar\m_1)_S(\fa{Y}{S}\alpha)&=        (\bar\m_1)_S([\fa{Y}{S}\alpha])&&\text{by \cref{p:free-is-one-step}}\\
        &=\fa{M(Y)}{M(S)}\m_{S\times Y}(\alpha)&&\text{by \eqref{eq:generator-to-something}}.
    \end{align*}

    This shows that $(M, \m, \bar{\m}_1)$ is a QA-one-step Boolean doctrine morphism, thus settling the existence.
    
    Uniqueness is clear, because such a natural transformation $\bar\m_1$ has a prescribed behavior on the generators: indeed, the commutativity of the diagram on the right in \eqref{diag:prop-univ-free} implies \eqref{eq:generator-to-something}.
\end{proof}

\begin{corollary}\label{c:left-adjoint}
    The forgetful functor from the category of QA-one-step Boolean doctrines over a small category to the category of Boolean doctrines over a small category has a left adjoint.
    On objects, the left adjoint maps $\P$ to $(\P,\Free_1^\P)$, and, on morphisms, the left adjoint maps $(M,\m)\colon \P\to\R$ to $(M,\m, \bar\m_1)\colon (\P,\Free_1^\P)\to(\R,\Free_1^\R)$, with $\bar\m_1$ defined as in the proof of \cref{t:univ-prop-counit-QA}.
\end{corollary}

\begin{remark}[Non-monadicity]\label{r:non-monadicity}
    The forgetful functor $U$ of \cref{c:left-adjoint} is not monadic, since it does not reflect isomorphisms.\footnote{We thank the referee for suggesting that we include this fact and for suggesting the proof (which we have slightly reworded).}
    Let us consider the category $\Ctx_\varnothing$ of contexts over the empty functional language $\F\coloneqq\varnothing$, and let $\P\colon\Ctx\op_\varnothing\to\BA$ be the constant
    two-element Boolean doctrine.
    The pair $(\P,\P)$ is a QA-one-step Boolean doctrine with identity quantifiers, so freeness gives a morphism
    \[
        (\id,\id,\mathfrak{q})\colon(\P,\Free_1^\P)\longrightarrow(\P,\P)
        \qquad\text{with}\qquad U (\id,\id,\mathfrak{q})=\id_\P.
    \]
    Let $X\in\Ctx_\varnothing$ be any non-empty context. The function $\mathfrak{q}_\tmn\colon\Free_1^\P(\tmn)\to\P(\tmn)$ sends
    $[\fa{X}{\tmn}\bot_{\P(X)}]$ to $\fa{X}{\tmn}\bot_{\P(X)}=\bot_{\P(\tmn)}$.
    However, $[\fa{X}{\tmn}\bot_{\P(X)}]\ne\bot_{\Free_1^\P(\tmn)}=[\fa{\tmn}{\tmn}\bot_{\P(\tmn)}]$ by
    \cref{t:characterization-fibers}, since $\top_{\P(\tmn)} \nleq\bot_{\P(\tmn)}$ and there is no morphism
    $\tmn\to X$ in $\Ctx_\varnothing$.
    Thus $(\id,\id,\mathfrak{q})$ is not an isomorphism, while its image $U(\id, \id, \mathfrak{q}) = \id_\P$ is an isomorphism.
    
    We observe that the Boolean doctrine $\P$ is isomorphic to $\LT_0^\T$ (defined as in \cref{ex:sanitycheck}) for the empty theory $\T\coloneqq\varnothing$ in the empty language $(\F,\mathbb{P})\coloneqq(\varnothing,\varnothing)$, and the QA-one-step Boolean doctrine $(\P,\Free_1^\P)$ is isomorphic to the pair $(\LT_0^\T,\LT_1^\T)$ of quantifier-free formulas and formulas with quantifier alternation depth at most $1$: $[\fa{X}{\tmn}\bot]$ represents the sentence
    $\forall x_1 \dots\forall x_n\,\bot$ (for $n\geq 1$), expressing that the domain is empty.
\end{remark}

The question whether there is an intrinsic characterization of free QA-one-step Boolean doctrines, i.e., which QA-one-step Boolean doctrines $(\P_0,\P_1)$ are such that the canonical morphism $(\P_0,\Free_1^{\P_0}) \to (\P_0,\P_1)$ is an isomorphism, remains an open question.

\section{The first layer of the quantifier completion is the free QA-one-step Boolean doctrine}\label{sec:qc-is-free-qa}

In the previous section, we described the free QA-one-step Boolean doctrine $(\P,\Free_1^\P)$ over a given Boolean doctrine $\P$ over a small category.
Using Herbrand's theorem, we now connect this description with $\P_1^\EA$.
Recall from \cref{n:first-layer} that $\P_1^\EA$ is not defined directly via a universal property, but rather as a certain subfunctor of the quantifier completion of $\P$.
In this brief section, we show that the free QA-one-step Boolean doctrine over $\P$ is nothing but the pair $(\P, \P^\EA_1)$.
Thus, $\Free_1^\P$ constitutes an explicit construction of $\P_1^\EA$ in terms of $\P$.

\begin{theorem}\label{t:section-6}
    Let $\P\colon\C\op\to\BA$ be a Boolean doctrine with $\C$ small. The QA-one-step Boolean doctrines $(\P,\Free_1^\P)$ and $(\P,\P^\EA_1)$ are isomorphic. In other words, $(\P,\P^\EA_1)$ is the free QA-one-step Boolean doctrine over $\P$ (i.e., it satisfies the universal property in \cref{t:univ-prop-counit-QA}).
\end{theorem}
\begin{proof}
    By the universal property of $(\P,\Free_1^\P)$ (\cref{t:univ-prop-counit-QA}) there is a unique natural transformation $\bar{\m}_1\colon\Free^\P_1\to\P^\EA_1$ such that $(\id_\C,\id_\P,\bar{\m}_1)\colon (\P,\Free_1^\P)\to(\P,\P^\EA_1)$ is a morphism of QA-one-step Boolean doctrines. In other words, there is a unique natural transformation $\bar{\m}_1\colon\Free^\P_1\to\P^\EA_1$ such that the following diagrams commute for all $S, Y \in \C$.
     \[
    \begin{tikzcd}[column sep = 4em]
        \P  \arrow[r, hook, "{\forall{\tmn}}"] \arrow[dr, hook, "\mathfrak{i}"']& \Free_1^{\P} \arrow[d, "{\bar{\m}_1}"]&\P(S\times Y)  \arrow[r, "{[\fa{Y}{S}-]}"] \arrow[dr, "{\fa{Y}{S}}"']& \Free_1^{\P}(S) \arrow[d, "{(\bar{\m}_1)_S}"]\\
            & \P^\EA_1&&\P^\EA_1(S)
    \end{tikzcd}
    \]
    In order to prove that it is an isomorphism, it is enough to show that it is fiberwise bijective.
    Let $S \in \C$, and let us prove that the function $(\bar{\m}_1)_S \colon \Free_1^{\P}(S) \to \P_1^\EA(S)$ is an isomorphism.
    Surjectivity follows from the fact that the Boolean homomorphism $(\bar{\m}_1)_S \colon \Free_1^{\P}(S) \to \P_1^\EA(S)$ maps the generating subset $\{[\fa{Y}{S} \alpha] \mid Y\in\C, \alpha\in\P(S\times Y)\}$ of the Boolean algebra $\Free_1^{\P}(S)$ to the generating subset $\{\fa{Y}{S} \alpha \mid Y\in\C, \alpha\in\P(S\times Y)\}$ of $\P^\EA_1(S)$.
    It is easily seen that, given a Boolean homomorphism $f \colon A \to B$ and a generating subset $T$ of $A$, the function $f$ is injective if and only if, for all finite subsets $V$ and $U$ of $T$, if $\bigwedge_{v\in V} f(v)\leq \bigvee_{u\in U} f(u)$ then $\bigwedge_{v\in V} v\leq \bigvee_{u\in U} u$. Thus, injectivity of $(\bar{\m}_1)_S$ follows from \cref{t:description-P1}[\eqref{i:inequality-in-free}$\Rightarrow$\eqref{i:existence-of-terms-in-P}] and the definition of $\Free_1^\P$ (\cref{n:free-on-obj}).
\end{proof}

\section{Future work}\label{s:future-work}

One long-term goal is to provide a step-by-step construction of the quantifier completion of a Boolean doctrine; in this paper, we have addressed the first step.
For the remaining steps, we must show how to freely add further layers of quantifiers while keeping a specified set of quantifiers already defined in the previous steps.

Towards this aim, we plan to first establish an axiomatization of the tuples of the form $(\P_0, \dots, \P_m)$, where, for a certain theory $\T$, each $\P_i$ represents the set of $\T$-equivalence classes of formulas of quantifier alternation depth less than or equal to $i$ (see \cite[Question~7.1]{AbbadiniGuffanti} for more details).

Furthermore, as happens for adding layers of modality to Boolean algebras, the step-by-step construction of the quantifier completion might also have a useful Stone dual in terms of Joyal's polyadic spaces \cite{Joyal1971,Marques2023,vanGoolMarques2024}.
We plan to use the present work to show how to freely add the first layer of quantifiers, dually.

\appendix

\makeatletter
\begingroup
\let\addcontentsline\@gobblethree

\section*{Appendices}
Let us clarify the dependency between the appendices and the previous sections: on the one hand, in \cref{s:app-char,s:char-filter-ideal-at-S,s:equality} we use notions and results from the rest of the manuscript, while results from \cref{s:app-char,s:char-filter-ideal-at-S,s:equality} are not used in the main body of the paper, and are only mentioned there for motivational purposes. On the other hand, \cref{s:entailment} does not rely on any result from the core of the manuscript, as it is an overview of some results about entailment relations, that we use in \cref{s:construction}.
\endgroup
\makeatother

\section{Semantic characterizations of universal filters and ideals} \label{s:app-char}

\Cref{t:characterization-ultrafilters} shows that the notion of universal ultrafilter is meaningful.
To prove the theorem, we used universal filters and ideals.
In this section, we will show that these two notions are not just auxiliary technical notions, but are also meaningful since they have a semantic characterization.
In \cref{t:characterization-filters} we prove that universal filters are precisely the families of all formulas that are universally valid in all members of some class of models.
Similarly, in \cref{t:characterization-ideals}, we prove that universal ideals are precisely the families of all formulas whose universal closure fails in all members of some class of models.
For the sake of completeness, we also characterize the pairs consisting of a filter and an ideal that arise from a common family of models. This is obtained in \cref{t:characterization-filter-ideal-pair}. Such pairs are called \emph{filter-ideal pairs}.

In \cref{s:char-filter-ideal-at-S} we will generalize the results in \cref{sec:completenes-first-layer} and in this appendix to the case where there are some fixed free variables exempt from universal closure.

\subsection{Semantic characterization of universal filters} \label{s:app-sub-filter}

\begin{lemma}\label{l:ultrafilter-lemma-element}
    Let $\P \colon \C\op \to \BA$ be a Boolean doctrine, with $\C$ small, $F = (F_X)_{X \in \C}$ a universal filter, $Y \in \C$ and $\alpha \in \P(Y) \setminus F_Y$.
    There is a universal ultrafilter that extends $F$ and does not contain $\alpha$.
\end{lemma}
\begin{proof}
    Let $I$ be the universal ideal generated by $\alpha\in\P(Y)$. By \cref{l:desc-filter-ideal-generated}(\ref{i:desc-generaliz-ideal}), for each $X\in\C$ we have
    \[
    I_X = \{\varphi \in \P(X) \mid \text{there are }n,m\in\N, (f_j \colon Y^n\to X)_{j=1,\dots,m} \text{ such that } \bigwedge_{j=1}^m\P(f_j)(\varphi)\leq\bigvee_{i=1}^n\P(\pr_i^{Y^n})(\alpha)\}.
    \]
    Observe that $F$ and $I$ are componentwise disjoint: indeed, suppose $\varphi\in I_X\cap F_X$, so that there are $n,m\in\N$ and $(f_j \colon Y^n\to X)_{j=1,\dots,m}$ such that $ \bigwedge_{j=1}^m\P(f_j)(\varphi)\leq\bigvee_{i=1}^n\P(\pr_i^{Y^n})(\alpha)$ in $\P(Y^n)$. By taking the reindexing along $\Delta^n\coloneqq\ple{\id_Y,\dots,\id_Y}\colon Y\to Y^n$ we get $\bigwedge_{j=1}^m\P(f_j\circ\Delta^n)(\varphi)\leq\alpha$ in $\P(Y)$ (for $n = 0$, the morphism $\Delta^0$ is $!_Y \colon Y \to \tmn$ and the empty join is $\bot_{\P(Y)}\leq \alpha$).
    Since $F$ is closed under reindexings and finite meets and is upward closed, we get $\alpha\in F_Y$, a contradiction.
    By \cref{t:gen-ult-lem}, there is a universal ultrafilter $G$ extending $F$ and disjoint from $I$. In particular, $\alpha$ does not belong to $G$, as desired.
\end{proof}

\begin{remark}
    \Cref{l:ultrafilter-lemma-element} is similar to the version of the classical ultrafilter lemma stating that every filter not containing an element $a$ can be extended to an ultrafilter not containing $a$.
\end{remark}

\begin{definition}
    A universal filter $(F_X)_{X \in \C}$ for a Boolean doctrine $\P \colon \C\op \to \BA$ is \emph{consistent} if $\bot_{\P(\tmn)}\notin F_\tmn$.
\end{definition}

\begin{lemma}\label{l:extension-consistent-filterultrafilter}
    Every consistent universal filter of a Boolean doctrine over a small category
    can be extended to a universal ultrafilter.
\end{lemma}

\begin{proof}
    It suffices to apply \cref{l:ultrafilter-lemma-element} with $Y=\tmn$ and $\alpha=\bot_{\P(\tmn)}$ (where $\P$ is the Boolean doctrine).
\end{proof}

\begin{theorem} \label{t:characterization-filters}
    Let $\P \colon \C\op \to \BA$ be a Boolean doctrine, with $\C$ small.
    Let $F = (F_X)_{X \in \C}$ be a family with $F_X \subseteq \P(X)$ for each $X \in \C$.
    The following are equivalent.
    \begin{enumerate}
        \item \label{i:exist-class-ne}
        There is a class $\mathcal{M}$ of propositional models of $\P$ such that, for every $X \in \C$,
        \[
            F_X = \{\alpha \in \P(X) \mid \text{for all }(M,\m) \in \mathcal{M},\text{ for all }x \in M(X),\, x \in \m_X(\alpha)\}.
        \]
        
        \item \label{i:prop-univ-filt-ne}
        $F$ is a universal filter for $\P$.
        
        \item \label{i:filter-is-intersection-ne}
        $F$ is the intersection of the universal ultrafilters for $\P$ containing $F$.
    \end{enumerate}
    Equivalent are also the statements obtained by additionally requiring: nonemptiness of $\mathcal{M}$ in \eqref{i:exist-class-ne}, consistency of $F$ in \eqref{i:prop-univ-filt-ne}, and the existence of a universal ultrafilter for $\P$ containing $F$ in \eqref{i:filter-is-intersection-ne}. 
\end{theorem}
\begin{proof}
    \eqref{i:exist-class-ne} $\Rightarrow$ \eqref{i:filter-is-intersection-ne}.
    Let $\mathcal{G}$ be the set of universal ultrafilters containing $F$.
    Fix $X \in \C$.
    The inclusion $F_X\subseteq\bigcap_{G\in\mathcal{G}} G_X$ is immediate by definition of $\mathcal{G}$. 
    For the converse inclusion, let $\alpha\in\bigcap_{G\in\mathcal{G}} G_X$.
    To prove $\alpha \in F_X$, we check that for all $(M,\m)\in\mathcal{M}$ and $x\in M(X)$ we have $x\in\m_X(\alpha)$. Let $(M,\m)\in\mathcal{M}$.
    For all $Y\in\C$, we set
    \begin{equation*}
        H_Y \coloneqq \{\beta \in \P(Y) \mid \text{for all }x \in M(Y),\, x \in \m_Y(\beta)\}.
    \end{equation*}
    By \cref{t:characterization-ultrafilters}, $(H_Y)_{Y \in \C}$ is a universal ultrafilter, and it is easy to see that it belongs to $\mathcal{G}$.
    Then, $\alpha\in \bigcap_{G\in\mathcal{G}} G_X \subseteq H_X$, as desired.
    
    \eqref{i:filter-is-intersection-ne} $\Rightarrow$ \eqref{i:exist-class-ne}.
     Let $\mathcal{M}$ be the class of propositional models $(M,\m)$ of $\P$ such that, for every $X \in \C$,
    \[
        F_X  \subseteq \{\alpha \in \P(X) \mid \text{for all }x \in M(X),\, x \in \m_X(\alpha)\}.
    \]
    Let 
    \[
        \beta\in \bigcap_{(M,\m)\in\mathcal{M}}\{\alpha \in \P(X) \mid \text{for all }x \in M(X),\, x \in \m_X(\alpha)\}.
    \]
    We show that $\beta\in F_X$, i.e.\ that $\beta$ belongs to all universal ultrafilters $G$ containing $F$. 
    Let $G$ be any such universal ultrafilter.
    By \cref{t:characterization-ultrafilters} there is a propositional model $(M,\m)$ of $\P$ such that, for all $Y\in\C$, $G_Y=\{\gamma \in \P(Y) \mid \text{for all }y \in M(Y),\, y \in \m_Y(\gamma)\}$. It is then easy to see that the propositional model $(M,\m)$ belongs to $\mathcal{M}$. By hypothesis on $\beta$, we have $\beta\in G_X$, as desired.

    \eqref{i:prop-univ-filt-ne} $\Rightarrow$ \eqref{i:filter-is-intersection-ne}.
    Clearly, $F$ is contained in the intersection of the universal ultrafilters containing $F$.    
    For the converse inclusion, let $Y\in\C$ and $\alpha\in \P(Y)\setminus F_Y$.
    By \cref{l:ultrafilter-lemma-element}, there is a universal ultrafilter extending $F$ and not containing $\alpha$.

    \eqref{i:filter-is-intersection-ne} $\Rightarrow$ \eqref{i:prop-univ-filt-ne}.
    The componentwise intersection of universal (ultra)filters is a universal filter.

    This proves that the statements \eqref{i:exist-class-ne}, \eqref{i:prop-univ-filt-ne} and \eqref{i:filter-is-intersection-ne} are equivalent.

    Let us now prove that the statements (1'), (2') and (3') obtained from \eqref{i:exist-class-ne}, \eqref{i:prop-univ-filt-ne} and \eqref{i:filter-is-intersection-ne} as in the final paragraph of the theorem are equivalent.

    (1') $\Rightarrow$ (3').
    Since $\mathcal{M}$ is nonempty, there is $(M,\m)\in\mathcal{M}$. The family $(H_Y)_{Y \in \C}$ defined by
    \[
        H_Y \coloneqq \{\beta \in \P(Y) \mid \text{for all }x \in M(Y),\, x \in \m_Y(\beta)\}
    \] 
    is a universal ultrafilter for $\P$ (by \cref{t:characterization-ultrafilters}) containing $F$.

    (3') $\Rightarrow$ (1').
    The class $\mathcal{M}$ is nonempty because, if $\mathcal{M}$ were empty, we would have $F_\tmn = \P(\tmn)$, contradicting the existence of a universal ultrafilter for $\P$ containing $F$.

    (2') $\Rightarrow$ (3').
    This follows from \cref{l:extension-consistent-filterultrafilter}.
    
    (3') $\Rightarrow$ (2').
    This is immediate.
\end{proof}

\subsection{Semantic characterization of universal ideals}

\begin{definition}[Universal ultraideal]
    Let $\P \colon \C\op \to \BA$ be a Boolean doctrine. A \emph{universal ultraideal for $\P$} is a family $(I_X)_{X \in \C}$, with $I_X \subseteq \P(X)$ for each $X \in \C$, such that 
    \begin{enumerate}
        \item\label{forall-reindex-ultraideal} 
        For all $f \colon X \to Y$ and $\alpha \in \P(Y)$, if $\P(f)(\alpha) \in I_X$ then $\alpha \in I_Y$.
        
        \item\label{i:uf-ideal} For all $X \in \C$, $\P(X) \setminus I_X$ is a filter of $\P(X)$.
        
        \item\label{i:uf-join-ideal}
        For all $\alpha_1 \in I_{X_1}$ and $\alpha_2\in I_{X_2}$, we have $\P(\pr^{X_1 \times X_2}_{X_1})(\alpha_1)\lor \P(\pr^{X_1 \times X_2}_{X_2})(\alpha_2)\in I_{X_1\times X_2}$.
        
        \item \label{i:uf-bot-ideal}
        $\bot_{\P(\tmn)}\in I_\tmn$. 
    \end{enumerate}
\end{definition}

\begin{remark} \label{r:complement}
    A universal ultraideal is simply the componentwise complement of a universal ultrafilter.
\end{remark}

\begin{lemma}\label{l:ultraideal-lemma-element}
    Let $\P \colon \C\op \to \BA$ be a Boolean doctrine, with $\C$ small, $I = (I_X)_{X \in \C}$ a universal ideal, $Y \in \C$ and $\alpha \in \P(Y) \setminus I_Y$.
    There is a universal ultraideal that extends $I$ and does not contain $\alpha$.
\end{lemma}

\begin{proof}
    Let $F$ be the universal filter generated by $\alpha\in\P(Y)$. By the description in \cref{l:desc-filter-ideal-generated}\eqref{i:descr-filter-generated}, for every $X\in\C$,
    \[
        F_X = \mleft\{\beta \in \P(X) \mid \text{there are } n\in\N, (f_i \colon X\to Y)_{i=1,\dots,n} \text{ such that } \bigwedge_{i=1}^n\P(f_i)(\alpha)\leq\beta\mright\}.
    \]
    Observe that $F$ and $I$ are componentwise disjoint: indeed, by way of contradiction, suppose $\beta\in I_X\cap F_X$, so that there are $n\in\N$ and $(f_i \colon X\to Y)_{i=1,\dots,n}$ such that $\bigwedge_{i=1}^n\P(f_i)(\alpha)\leq\beta$ in $\P(X)$. Since $I$ is downward closed, we get $\bigwedge_{i=1}^n\P(f_i)(\alpha)\in I_X$.
    Therefore, by \cref{d:univ-ideal}\eqref{i:conjunction}, we obtain $\alpha \in I_Y$, a contradiction.

    By the universal ultrafilter lemma (\cref{t:gen-ult-lem}), there is a universal ultrafilter $G$ containing $F$ and disjoint from $I$. By \cref{r:complement}, the complement of $G$ is an ultraideal, and it has the desired properties.
\end{proof}

\begin{definition}
    A universal ideal $(I_X)_{X \in \C}$ for a Boolean doctrine $\P \colon \C\op \to \BA$ is \emph{consistent} if $\top_{\P(\tmn)}\notin I_\tmn$.
\end{definition}

\begin{lemma}\label{l:extension-consistent-idealultraideal}
    Every consistent universal ideal of a Boolean doctrine over a small category
    can be extended to a universal ultraideal.
\end{lemma}

\begin{proof}
    Let $\P$ denote the Boolean doctrine and $I$ the ideal.
    Applying \cref{l:ultraideal-lemma-element} (with $Y=\tmn$ and $\alpha=\top_{\P(\tmn)}\notin I_\tmn$), we obtain that there is a universal ultraideal that extends $I$ (and does not contain $\top_{\P(\tmn)}$).
\end{proof}

\begin{theorem} \label{t:characterization-ideals}
    Let $\P \colon \C\op \to \BA$ be a Boolean doctrine, with $\C$ small.
    Let $I = (I_X)_{X \in \C}$ be a family with $I_X \subseteq \P(X)$ for each $X \in \C$. 
    The following are equivalent.
    \begin{enumerate}
        \item \label{i:exist-class-ne-id}
        There is a class $\mathcal{M}$ of propositional models of $\P$ such that, for every $X \in \C$,
        \[
            I_X = \{\alpha \in \P(X) \mid \text{for all }(M,\m) \in \mathcal{M},  \text{ not all }x \in M(X) \text{ satisfy } x \in \m_X(\alpha)\}.
        \]
        
        \item \label{i:prop-univ-filt-ne-id}
        $I$ is a universal ideal for $\P$.
        
        \item \label{i:filter-is-intersection-ne-id}
        $I$ is the intersection of the universal ultraideals for $\P$ containing $I$.
    \end{enumerate}
    Equivalent are also the statements obtained by additionally requiring:  nonemptiness of $\mathcal{M}$ in \eqref{i:exist-class-ne-id},  consistency of $I$ in \eqref{i:prop-univ-filt-ne-id}, and the existence of a universal ultraideal for $\P$ containing $I$ in \eqref{i:filter-is-intersection-ne-id}.    
\end{theorem}

\begin{proof}
    \eqref{i:exist-class-ne-id} $\Rightarrow$ \eqref{i:filter-is-intersection-ne-id}.    
    Let $\mathcal{J}$ be the set of universal ultraideals containing $I$.
    Fix $X \in \C$.
    The inclusion $I_X\subseteq\bigcap_{J\in\mathcal{J}} J_X$ is immediate by definition of $\mathcal{J}$. 
    For the converse inclusion, let $\alpha\in\bigcap_{J\in\mathcal{J}} J_X$.
    To prove that $\alpha \in I_X$, we check that, for all $(M,\m)\in\mathcal{M}$, not all $x\in M(X)$ satisfy $x\in\m_X(\alpha)$. Let $(M,\m) \in \mathcal{M}$.
    For all $Y \in \C$, set
    \begin{equation*}
        L_Y = \{\beta \in \P(Y) \mid \text{not all }x \in M(Y) \text{ satisfy } x \in \m_Y(\beta)\}.
    \end{equation*}
    By \cref{t:characterization-ultrafilters} and \cref{r:ultrafilter-as-pair}, $(L_Y)_{Y \in \C}$ is a universal ultraideal, and it is easy to see that it belongs to $\mathcal{J}$.
    Then, $\alpha\in \bigcap_{J\in\mathcal{J}} J_X \subseteq L_X$, as desired.

    \eqref{i:filter-is-intersection-ne-id} $\Rightarrow$ \eqref{i:exist-class-ne-id}.
    Let $\mathcal{M}$ be the class of propositional models $(M,\m)$ of $\P$ such that, for every $X \in \C$,
    \[
        I_X  \subseteq \{\alpha \in \P(X) \mid \text{not all }x \in M(X)\text{ satisfy } x \in \m_X(\alpha)\}.
    \]
    Let 
    \[
    \beta\in \bigcap_{(M,\m)\in\mathcal{M}}\{\alpha \in \P(X) \mid \text{not all }x \in M(X)\text{ satisfy } x \in \m_X(\alpha)\}.
    \]
    We show $\beta\in I_X$, i.e.\ that $\beta$ belongs to every universal ultraideal containing $I$. 
    Let $J$ be any such universal ultraideal and $G$ the componentwise complement of $J$. In particular, $G$ is a universal ultrafilter.
    By \cref{t:characterization-ultrafilters}, there is a propositional model $(M,\m)$ of $\P$ such that, for all $Y\in\C$, $G_Y=\{\gamma \in \P(Y) \mid \text{for all }y \in M(Y),\, y \in \m_Y(\gamma)\}$. It is then easy to see that the propositional model $(M,\m)$ belongs to $\mathcal{M}$. By hypothesis on $\beta$, we have $\beta\notin G_X$ and hence $\beta\in J_X$, as desired.    

    \eqref{i:prop-univ-filt-ne-id} $\Rightarrow$ \eqref{i:filter-is-intersection-ne-id}.
    It is easy to see that $I$ is contained in the intersection of the universal ultraideals that contain $I$.
    For the converse inclusion, let $Y\in\C$ and $\alpha\in \P(Y)\setminus I_Y$. Apply \cref{l:ultraideal-lemma-element} to get a universal ultraideal $J$ that extends $I$ and does not contain $\alpha$, as desired.
    
    \eqref{i:filter-is-intersection-ne-id} $\Rightarrow$ \eqref{i:prop-univ-filt-ne-id}.
    The componentwise intersection of universal (ultra)ideals is a universal ideal.

    This proves that the statements \eqref{i:exist-class-ne-id}, \eqref{i:prop-univ-filt-ne-id} and \eqref{i:filter-is-intersection-ne-id} are equivalent.
    
    Let us now prove that the statements (1'), (2') and (3') obtained from \eqref{i:exist-class-ne-id}, \eqref{i:prop-univ-filt-ne-id} and \eqref{i:filter-is-intersection-ne-id} as in the final paragraph of the theorem are equivalent.

    (1') $\Rightarrow$ (3').
    Since $\mathcal{M}$ is nonempty, there is $(M,\m)\in\mathcal{M}$. The family $(L_Y)_{Y \in \C}$ defined by
    \[
        L_Y = \{\beta \in \P(Y) \mid \text{not all }x \in M(Y) \text{ satisfy } x \in \m_Y(\beta)\}
    \] 
    is a universal ultraideal (by \cref{t:characterization-ultrafilters} and \cref{r:ultrafilter-as-pair}) containing $I$.

    (3') $\Rightarrow$ (1').
    The class $\mathcal{M}$ is nonempty because, if $\mathcal{M}$ were empty, we would have $I_\tmn = \P(\tmn)$, contradicting the existence of a universal ultraideal containing $I$.

    (2') $\Rightarrow$ (3').
    This follows from \cref{l:extension-consistent-idealultraideal}.

    (3') $\Rightarrow$ (2').
    This is immediate.
\end{proof}

\subsection{Semantic characterization of universal filter-ideal pairs}

\begin{definition}[Universal filter-ideal pair]\label{d:fil-id-pair}
    A \emph{universal filter-ideal pair for a Boolean doctrine $\P$} is a pair $(F, I)$ where $F = (F_X)_{X \in \C}$ is a universal filter for $\P$, $I = (I_X)_{X \in \C}$ is a universal ideal for $\P$, and the following conditions hold for all $Y\in\C$ and $\alpha \in \P(Y)$.
    \begin{enumerate}
        \item \label{i:connecting-one} For all $X\in\C$, $n \in \N$, $(f _i \colon X \to Y)_{i=1,\dots,n}$ and $\beta \in F_X$, if $\beta \land \bigwedge_{i = 1}^n \P(f_i)(\alpha) \in I_X$, then $\alpha \in I_Y$.

        \item\label{i:connecting-two} For all $Z\in\C$ and $\gamma \in I_Z$, if $\P(\pr^{Y\times Z}_Y)(\alpha) \lor \P(\pr^{Y\times Z}_Z)(\gamma) \in F_{Y \times Z}$, then $\alpha \in F_Y$. 
    \end{enumerate}
\end{definition}

\begin{definition} \label{d:universal-filter-ideal-pair}
    We say that a universal filter-ideal pair $(F, I)$ for a Boolean doctrine $\P \colon \C\op \to \BA$ is \emph{consistent} when for every $X \in \C$ we have $F_X \cap I_X = \varnothing$.
    Otherwise, we say it is \emph{inconsistent}.
\end{definition}

\begin{lemma} \label{l:inconsistent-is-all}
    Let $(F, I)$ be an inconsistent universal filter-ideal pair for a Boolean doctrine $\P \colon \C\op \to \BA$.
    For every $Y \in \C$, $F_Y = I_Y = \P(Y)$.
\end{lemma}

\begin{proof}
    By inconsistency, there are $Z \in \C$ and $\gamma \in F_Z \cap I_Z$.   
    Since $\gamma \in I_Z$ and
    \[
        \P(\pr^{\tmn\times Z}_\tmn)(\bot_{\P(\tmn)}) \lor \P(\pr^{\tmn\times Z}_Z)(\gamma) = \gamma \in F_Z = F_{\tmn \times Z},
    \]
    we have $\bot_{\P(\tmn)} \in F_\tmn$ by \cref{d:fil-id-pair}\eqref{i:connecting-two}.
    Let $Y \in \C$ and $\alpha \in \P(Y)$.
    The fact that $\alpha \in F_Y$ follows from $\bot_{\P(\tmn)} \in F_\tmn$ and the properties of universal filters.
    Since $\bot_{\P(\tmn)} \in F_\tmn \cap I_\tmn$, by \cref{d:fil-id-pair}\eqref{i:connecting-one} (applied with $n = 0$ and $\beta = \bot_{\P(\tmn)}$), $\alpha \in I_Y$.
\end{proof}

The following lemma explains the purpose of \eqref{i:connecting-one} and \eqref{i:connecting-two} in \cref{d:fil-id-pair}.

\begin{lemma}\label{l:const-filt-id-sep}
    Let $(F,I)$ be a universal filter-ideal pair for a Boolean doctrine $\P \colon \C\op \to \BA$, with $\C$ small, let $Y \in \C$ and let $\alpha \in \P(Y)$.
    \begin{enumerate}
        \item \label{i:separatation-1} 
        If $\alpha \notin I_Y$, then $I$ is componentwise disjoint from the universal filter generated by $F$ and $\alpha$, and thus there is a universal ultrafilter extending $F$, containing $\alpha$, and disjoint from $I$.
        
        \item \label{i:separation-2}
        If $\alpha \notin F_Y$, then $F$ is componentwise disjoint from the universal ideal generated by $I$ and $\alpha$, and thus there is a universal ultrafilter extending $F$, not containing $\alpha$, and disjoint from $I$.
    \end{enumerate}
\end{lemma}

\begin{proof}
    \eqref{i:separatation-1}.
    We prove the contrapositive of the implication ``If $\alpha \notin I_Y$, then $I$ is componentwise disjoint from the universal filter generated by $F$ and $\alpha$''.
    Suppose that $I$ intersects the universal filter generated by $F$ and $\alpha$ in some fiber.
    By \cref{l:char-gen-intersects}\eqref{i:intersect-1}, there are $X \in \C$, $n\in\N$, $(f_i \colon X \to Y)_{i= 1, \dots, n}$ and $\beta \in F_X$ such that $\beta \land \bigwedge_{i = 1}^n\P(f_i)(\alpha) \in I_X$.
    By \cref{d:fil-id-pair}\eqref{i:connecting-one}, $\alpha \in I_Y$.
    This proves that, if $\alpha \notin I_Y$, $I$ is componentwise disjoint from the universal filter generated by $F$ and $\alpha$.
    Therefore, if $\alpha \notin I_Y$, by \cref{t:gen-ult-lem} there is a universal ultrafilter extending $F$, containing $\alpha$, and disjoint from $I$.

    \eqref{i:separation-2}. 
    We prove the contrapositive of the implication ``If $\alpha \notin F_Y$, then $F$ is componentwise disjoint from the universal ideal generated by $I$ and $\alpha$''.
    Suppose that $F$ intersects the universal ideal generated by $I$ and $\alpha$.
    By \cref{l:char-gen-intersects}\eqref{i:intersect-2}, there is $X\in\C$ such that $I_X\cap F_X\neq\varnothing$ or there are $Z\in\C$ and $\gamma \in I_Z$ such that $\P(\pr^{Y\times Z}_Y)(\alpha) \lor \P(\pr^{Y\times Z}_Z)(\gamma) \in F_{Y \times Z}$.
    In the first case, $(F, I)$ is inconsistent, and thus $\alpha \in F_Y$ by \cref{l:inconsistent-is-all}.
    In the second case, by \cref{d:fil-id-pair}\eqref{i:connecting-two}, $\alpha\in F_Y$.
    If $\alpha \notin F_Y$, then $F$ is componentwise disjoint from the universal ideal generated by $I$ and $\alpha$.
    Similarly to \eqref{i:separatation-1}, one then proves that if $\alpha \notin F_Y$, then there is a universal ultrafilter extending $F$, not containing $\alpha$, and disjoint from $I$.
\end{proof}

\begin{lemma} \label{l:ultrafilter-ideal-is-pair}
    Let $\P \colon \C\op \to  \BA$ be a Boolean doctrine.
    Let $(F_X)_{X \in \C}$ be a universal ultrafilter, and for each $X \in \C$ set $I_X \coloneqq \P(X) \setminus F_X$.
    The pair $((F_X)_{X \in \C}, (I_X)_{X \in \C})$ is a universal filter-ideal pair.
\end{lemma}

\begin{proof}
    The family $(I_X)_{X \in \C}$ is a universal ideal (as already mentioned in \cref{r:ultrafilter-as-pair}).
    
    We prove the conditions \eqref{i:connecting-one} and \eqref{i:connecting-two} in \cref{d:fil-id-pair}.
    
    \eqref{i:connecting-one}.
    Let $X,Y\in\C$, $\alpha \in \P(Y)$, let $n\in\N$, let $(f_i \colon X \to Y)_{i=1,\dots,n}$, let $\beta \in F_X$, and suppose $\beta\land\bigwedge_{i=1}^n\P(f_i)(\alpha) \in I_X$.
    We shall prove $\alpha \in I_Y$, i.e., $\alpha \notin F_Y$.
    We suppose $\alpha \in F_Y$ and we seek a contradiction.
    From $\alpha \in F_Y$ we deduce that for all $i=1,\dots,n$, $\P(f_i)(\alpha) \in F_X$, and hence $\beta\land\bigwedge_{i=1}^n\P(f_i)(\alpha)\in F_X$.
    Thus $\beta\land\bigwedge_{i=1}^n\P(f_i)(\alpha)  \in F_X \cap I_X$, a contradiction.

    \eqref{i:connecting-two}.
    Let $Y, Z\in\C$, $\alpha \in \P(Y)$, $\gamma \in I_Z$ and suppose $\P(\pr^{Y\times Z}_Y)(\alpha)\lor\P(\pr^{Y\times Z}_Z)(\gamma)\in F_{Y\times Z}$. We check $\alpha\in F_Y$. We suppose $\alpha\notin F_Y$ and we seek a contradiction. From $\gamma \in I_Z$ we deduce $\gamma\notin F_Z$. Using condition \eqref{i:uf-join} in \cref{d:uf} we obtain $\P(\pr^{Y\times Z}_Y)(\alpha)\lor\P(\pr^{Y\times Z}_Z)(\gamma)\notin F_{Y \times Z}$, a contradiction.
\end{proof}

\begin{theorem} \label{t:characterization-filter-ideal-pair}
    Let $\P \colon \C\op \to \BA$ be a Boolean doctrine, with $\C$ small.
    Let $F = (F_X)_{X \in \C}$ and $I = (I_X)_{X \in \C}$ be families with $F_X \subseteq \P(X)$ and $I_X \subseteq \P(X)$ for each $X \in \C$.
    The following are equivalent.
    \begin{enumerate}
        \item \label{i:fi-arbitrary-class}
        There is a class $\mathcal{M}$ of propositional models of $\P$ such that, for every $X \in \C$,
        \begin{align*}
            F_X & = \{\alpha \in \P(X) \mid \text{for all }(M,\m) \in \mathcal{M},\text{ for all }x \in M(X),\, x \in \m_X(\alpha)\},\\
            I_X & = \{\alpha \in \P(X) \mid \text{for all }(M,\m) \in \mathcal{M},  \text{ not all }x \in M(X) \text{ satisfy } x \in \m_X(\alpha)\}.
        \end{align*}
        \item \label{i:fi-arbitrary-pair}
        $(F, I)$ is a universal filter-ideal pair for $\P$.
        \item\label{i:fi-arbitrary-intersection} $F$ is the intersection of the universal ultrafilters for $\P$ containing $F$ and disjoint from $I$, and $I$ is the intersection of the universal ultraideals for $\P$ containing $I$ and disjoint from $F$.
    \end{enumerate}
    Equivalent are also the statements obtained by additionally requiring:  nonemptiness of $\mathcal{M}$ in \eqref{i:fi-arbitrary-class}, consistency of $(F, I)$ in \eqref{i:fi-arbitrary-pair}, and the existence of a universal ultrafilter for $\P$ containing $F$ and disjoint from $I$  in \eqref{i:fi-arbitrary-intersection}. 
\end{theorem}
\begin{proof}
    \eqref{i:fi-arbitrary-class} $\Rightarrow$ \eqref{i:fi-arbitrary-intersection}. 
    We prove that $F$ is the componentwise intersection of the set $\mathcal{G}$ of all universal ultrafilters containing $F$ and disjoint from $I$.
    Fix $X \in \C$.
    The inclusion $F_X\subseteq\bigcap_{G\in\mathcal{G}} G_X$ is immediate by definition of $\mathcal{G}$. 
    For the converse inclusion, let $\alpha\in\bigcap_{G\in\mathcal{G}} G_X$.
    To prove that $\alpha \in F_X$, we check that for all $(M,\m)\in\mathcal{M}$ and $x\in M(X)$ we have $x\in\m_X(\alpha)$. Let $(M,\m)\in\mathcal{M}$.
    For all $Y\in\C$, set
    \begin{equation*}
    H_Y = \{\beta \in \P(Y) \mid \text{for all }x \in M(Y),\, x \in \m_Y(\beta)\}.
    \end{equation*}
    By \cref{t:characterization-ultrafilters}, $(H_Y)_{Y \in \C}$ is a universal ultrafilter, and it is easy to see that it belongs to $\mathcal{G}$.
    We have $\alpha\in \bigcap_{G\in\mathcal{G}} G_X \subseteq H_X$, as desired.

    A similar argument shows that $I$ is the intersection of the universal ultraideals for $\P$ containing $I$ and disjoint from $F$.
    
    \eqref{i:fi-arbitrary-intersection} $\Rightarrow$ \eqref{i:fi-arbitrary-class}.
    Let $\mathcal{M}$ be the class of propositional models $(M,\m)$ of $\P$ such that, for every $X \in \C$,
    \begin{align*}
        F_X & \subseteq \{\alpha \in \P(X) \mid \text{for all }x \in M(X),\, x \in \m_X(\alpha)\}, \text{ and}\\
        I_X & \subseteq \{\alpha \in \P(X) \mid \text{not all }x \in M(X) \text{ satisfy } x \in \m_X(\alpha)\}.
    \end{align*}
    Fix $X \in \C$, and let 
    \[
        \beta\in \bigcap_{(M,\m)\in\mathcal{M}}\{\alpha \in \P(X) \mid \text{for all }x \in M(X),\, x \in \m_X(\alpha)\}.
    \]
    We show that $\beta\in F_X$, i.e.\ that $\beta$ belongs to all universal ultrafilters $G$ containing $F$ and disjoint from $I$. 
    Let $G$ be any such universal ultrafilter.
    By \cref{t:characterization-ultrafilters} there is a propositional model $(M,\m)$ of $\P$ such that, for all $Y\in\C$, $G_Y=\{\gamma \in \P(Y) \mid \text{for all }y \in M(Y),\, y \in \m_Y(\gamma)\}$. It is then easy to see that the propositional model $(M,\m)$ belongs to $\mathcal{M}$. By hypothesis on $\beta$, we have $\beta\in G_X$, as desired.

    A similar argument shows the desired condition on $I$.

    \eqref{i:fi-arbitrary-pair} $\Rightarrow$ \eqref{i:fi-arbitrary-intersection}.
    It is easy to see that $F$ is contained in the intersection of the universal ultrafilters for $\P$ that contain $F$ and are disjoint from $I$.
    The converse inclusion is precisely \cref{l:const-filt-id-sep}\eqref{i:separation-2}.
    Analogously, $I$ is the intersection of the universal ultraideals for $\P$ containing $I$ and disjoint from $F$.

    \eqref{i:fi-arbitrary-intersection} $\Rightarrow$ \eqref{i:fi-arbitrary-pair}.
    By \cref{l:ultrafilter-ideal-is-pair}, if $(G_X)_{X \in \C}$ is a universal ultrafilter and for each $X \in \C$ we set $J_X \coloneqq \P(X) \setminus G_X$, then the pair $((G_X)_{X \in \C}, (J_X)_{X \in \C})$ is a universal filter-ideal pair.
    Thus, $(F, I)$ is an intersection of universal filter-ideal pairs $(G, J)$ (with the property that $G$ is a universal ultrafilter and $J$ is componentwise complementary), and so it is a universal filter-ideal pair.

    This proves that the statements \eqref{i:fi-arbitrary-class}, \eqref{i:fi-arbitrary-pair} and \eqref{i:fi-arbitrary-intersection} are equivalent.
    
    Let us now prove that the statements (1'), (2') and (3') obtained from \eqref{i:fi-arbitrary-class}, \eqref{i:fi-arbitrary-pair} and \eqref{i:fi-arbitrary-intersection} as in the final paragraph of the theorem are equivalent.

    (1') $\Rightarrow$ (3').
    Since $\mathcal{M}$ is nonempty, there is $(M,\m)\in\mathcal{M}$. The family $(H_Y)_{Y \in \C}$ defined by
    \[
        H_Y \coloneqq \{\beta \in \P(Y) \mid \text{for all }x \in M(Y),\, x \in \m_Y(\beta)\}
    \] 
    is a universal ultrafilter (by \cref{t:characterization-ultrafilters}) containing $F$ and disjoint from $I$.
    
    (3') $\Rightarrow$ (1'). 
    The class $\mathcal{M}$ is nonempty because, if $\mathcal{M}$ were empty, we would have $F_X = \P(X)$ for all $X \in \C$, contradicting the existence of a universal ultrafilter for $\P$ containing $F$.

    (2') $\Rightarrow$ (3').
    This follows from the universal ultrafilter lemma (\cref{t:gen-ult-lem}).

    (3') $\Rightarrow$ (2').
    This is immediate.
\end{proof}

\section{Semantic characterizations over fixed free variables}\label{s:char-filter-ideal-at-S}

In \cref{sec:completenes-first-layer}, we characterized the classes of formulas whose universal closure (with respect to \emph{all} free variables) is valid in some fixed model.
In this appendix, we do something similar, but we fix some free variables that are exempt from universal closure. 
To illustrate this, we introduce the following notation.

\subsection{Semantic characterization of universal ultrafilters over fixed free variables}
\begin{notation}\label{n:FM-IM-one-model-at-S}
    Let $(M,\m,s)$ be a propositional model of a Boolean doctrine $\P \colon \C\op \to  \BA$ at an object $S\in\C$ (see \cref{d:bool-mod-at-S}).
    For each $X \in \C$, define
    \[
        F_X^{S,(M,\m,s)} \coloneqq \{\alpha \in \P(S\times X) \mid \text{for all }x \in M(X),\, (s,x) \in \m_{S\times X}(\alpha)\},
    \]
    where we made implicit use of the isomorphism $M(S \times X) \cong M(S) \times M(X)$ in writing $(s, x) \in \m_{S\times X}(\alpha)$.
\end{notation}

\begin{remark}\label{r:univ-vald-fmlas-one-model-at-S}
    We translate \cref{n:FM-IM-one-model-at-S} to the classical syntactic setting.
    For this, we fix $k \in \N$.
    Let $\{s_1, \dots, s_k, x_1, x_2, \dots\}$ be a set of variables, $\mathcal{T}$ a universal theory, $M$ a model of $\mathcal{T}$, and $c_1, \dots, c_k \in M$.
    For each $n \in \N$ we define
    \[
        F_n^{k,M,c_1, \dots, c_k} \coloneqq \{ \alpha(s_1, \dots, s_k, x_1, \dots, x_{n}) \text{ q.-free}\mid M,[(s_i \mapsto c_i)_i] \vDash \forall x_1 \dots \forall x_n\, \alpha(s_1, \dots, s_k, x_1, \dots, x_n)\},
    \]
    where by $M,[(s_i \mapsto c_i)_i] \vDash \forall x_1 \dots \forall x_n\, \alpha(s_1, \dots, s_k, x_1, \dots, x_n)$ we mean that, under the variable assignment $[(s_i \mapsto c_i)_{i= 1, \dots, k}]$, the formula $\forall x_1 \dots \forall x_n\, \alpha(s_1, \dots, s_k, x_1, \dots, x_n)$ is valid in $M$.
\end{remark}

In \cref{t:characterization-ultrafilters-at-object} below we characterize the families of the form $(F_X^{S,(M,\m,s)})_{X \in \C}$ for some propositional model $(M,\m,s)$ at $S$, at least in the case where the base category $\C$ is small; these families are captured axiomatically by the notion of a \emph{universal ultrafilter at $S$}, introduced in \cref{d:uf-at-X} below.

\begin{definition} [Universal ultrafilter at an object] \label{d:uf-at-X}
    A \emph{universal ultrafilter for a Boolean doctrine $\P \colon \C\op \to \BA$ at $S \in \C$} is a family $(F_X)_{X \in \C}$, with $F_X \subseteq \P(S \times X)$ for $X \in \C$, with the following properties.
    \begin{enumerate}
        \item\label{i:uf-at-object-substitution} 
        For all $f\colon S\times X\to Y$ and $\alpha \in F_Y$, we have $\P(\langle \pr^{S\times X}_S,f\rangle)(\alpha) \in F_X$.
        
        \item \label{i:uf-at-object-filter}
        For all $X \in \C$, $F_X$ is a filter of $\P(S \times X)$.
        
        \item\label{i:uf-at-object-join}
        For all $\alpha_1 \in \P(S \times X_1) \setminus F_{X_1}$ and $\alpha_2 \in \P(S \times X_2) \setminus F_{X_2}$, in $\P(S\times X_1\times X_2)$ we have $\P(\pr^{S\times X_1\times X_2}_{S\times X_1})(\alpha_1)\lor \P(\pr^{S\times X_1\times X_2}_{S\times X_2})(\alpha_2)\notin F_{X_1\times X_2}$.
        
        \item \label{i:uf-at-object-bot}
        $\bot_{\P(S)}\notin F_\tmn$.
    \end{enumerate}
\end{definition}

\begin{remark} \label{r:uf-translation-to-classic-at-object}
    We translate \cref{d:uf-at-X} to the classical syntactic setting.
    For this, we fix $k \in \N$.
    Let $\{s_1, \dots, s_k, x_1, x_2, \dots\}$ be a set of variables and $\mathcal{T}$ a universal theory.
    A \emph{universal ultrafilter for $\mathcal{T}$ at $k$} is a family $(F_n)_{n \in \N}$, with $F_n$ a set of quantifier-free formulas with $s_1, \dots, s_k, x_1, \dots, x_n$ as (possibly dummy) free variables, with the following properties.
    \begin{enumerate}
        \item
        For every $n,m \in \N$, for every formula $\alpha(s_1, \dots, s_k, x_1, \dots, x_m) \in F_m$ and for every $m$-tuple $(f_i(s_1, \dots, s_k, x_1, \dots,x_n))_{i = 1, \dots, m}$ of $(k + n)$-ary terms, 
        \[
        \alpha(s_1, \dots, s_k, f_1(s_1, \dots, s_k, x_1, \dots, x_n), \dots, f_m(s_1, \dots, s_k, x_1, \dots, x_n)) \in F_n.
        \]
        
        \item For all $n \in \N$, 
        \begin{enumerate}
            \item for all quantifier-free formulas $\alpha(s_1, \dots, s_k, x_1, \dots, x_n)$ and $\beta(s_1, \dots, s_k, x_1, \dots, x_n)$, if we have $\alpha(s_1, \dots, s_k, x_1, \dots, x_n) \in F_n$ and $\alpha(s_1, \dots, s_k, x_1, \dots, x_n) \vdash_\T \beta(s_1, \dots, s_k, x_1, \dots, x_n)$, then we have $\beta(s_1, \dots, s_k, x_1, \dots, x_n) \in F_n$;
            \item 
            for all $\alpha_1(s_1, \dots, s_k, x_1, \dots, x_n), \alpha_2(s_1, \dots, s_k, x_1, \dots, x_n) \in F_n$, we have 
            \[\alpha_1(s_1, \dots, s_k, x_1, \dots, x_n) \land \alpha_2(s_1, \dots, s_k, x_1, \dots, x_n) \in F_n;\]
            \item $\top(s_1, \dots, s_k, x_1, \dots, x_n) \in F_n$.
        \end{enumerate}
            
        \item 
        For every $n_1,n_2 \in \N$ and for every pair of quantifier-free formulas $\alpha_1(s_1, \dots, s_k, x_1, \dots, x_{n_1})$ and $\alpha_2(s_1, \dots, s_k, x_1, \dots, x_{n_2})$, if
        \[
        \alpha_1(s_1, \dots, s_k, x_1, \dots, x_{n_1}) \lor \alpha_2(s_1, \dots, s_k, x_{n_1 +1}, \dots, x_{n_1 + n_2}) \in F_{n_1 + n_2},
        \]
        then $\alpha_1(s_1, \dots, s_k, x_1, \dots, x_{n_1}) \in F_{n_1}$ or $\alpha_2(s_1, \dots, s_k, x_1, \dots, x_{n_2}) \in F_{n_2}$.
        
        \item $\bot(s_1, \dots, s_k) \notin F_0$.
    \end{enumerate}
    
    For any model $M$ of $\T$ and any $c_1, \dots, c_k \in M$, it is easy to check that the family $(F_n)_{n\in\N}$ defined by
    \[
    F_n \coloneqq \{ \alpha(s_1, \dots, s_k, x_1, \dots, x_{n}) \text{ quantifier-free}\mid M, [(s_i \mapsto c_i)_i] \vDash \forall x_1 \dots \forall x_n\, \alpha(s_1, \dots, s_k, x_1, \dots, x_n)\}
    \]
    is a universal ultrafilter for $\T$ at $k$ in the sense above.
\end{remark}

\begin{remark}
    A universal ultrafilter for a Boolean doctrine $\P \colon \C\op \to  \BA$ at $S \in \C$ is a universal ultrafilter for the Boolean doctrine $\P_S\colon \C_S\op \to \BA$ obtained by adding a constant of type $S$ for $\P$ (see \cref{r:const}).
\end{remark}

\begin{theorem} \label{t:characterization-ultrafilters-at-object}
    Let $\P \colon \C\op \to \BA$ be a Boolean doctrine, with $\C$ small, and let $S \in \C$.
    Let $F = (F_X)_{X \in \C}$ be a family with $F_X \subseteq \P(S \times X)$ for each $X \in \C$.
    The following are equivalent.
    \begin{enumerate}
        \item \label{i:exists-model-at-object}
        There is a propositional model $(M, \m, s)$ of $\P$ at $S$ such that, for every $X \in \C$,
        \[
        F_X=\{\alpha \in \P(S \times X) \mid \text{for all }x \in M(X),\, (s , x) \in \m_{S \times X}(\alpha)\}.
        \] 
        
        \item \label{i:is-universal-ultrafilter-at-object}
        $F$ is a universal ultrafilter for $\P$ at $S$.
    \end{enumerate}
\end{theorem}
\begin{proof}
    This follows from \cref{t:characterization-ultrafilters} applied to the Boolean doctrine $\P_S$ obtained from $\P$ by adding a constant of type $S$ and from \cref{l:model-model-at-S}.
\end{proof}

\begin{remark}
    We translate \cref{t:characterization-ultrafilters-at-object} to the classical syntactic setting.
    For this, we fix $k \in \N$.
    Let $\{s_1, \dots, s_k, x_1, x_2, \dots\}$ be a set of variables and $\mathcal{T}$ a universal theory.
    Let $(F_n)_{n \in \N}$ be a family with $F_n$ a set of quantifier-free formulas with $s_1, \dots, s_k, x_1, \dots, x_n$ as free (possibly dummy) variables.
    The following are equivalent.
    \begin{enumerate}
        \item There are a model $M$ of $\mathcal{T}$ and $c_1, \dots, c_k \in M$ such that, for every $n \in \N$,
        \[
        F_n = \{\alpha(s_1, \dots, s_k, x_1, \dots, x_n) \text{ quantifier-free} \mid M , [(s_i \mapsto c_i)_i] \vDash \forall x_1 \dots \forall x_n\, \alpha(s_1, \dots, s_k, x_1, \dots, x_n)\}.
        \]

        \item $(F_n)_{n \in \N}$ is a universal ultrafilter for $\mathcal{T}$ at $k$ (in the sense of \cref{r:uf-translation-to-classic-at-object}).
    \end{enumerate}
\end{remark}

\subsection{Semantic characterization of universal filters over fixed free variables}

Analogously to \cref{sec:completenes-first-layer}, we introduce the notions of \emph{universal filters} and \emph{universal ideals} at a given object.
To motivate these notions, we extend \cref{n:FM-IM-one-model-at-S} as follows.

\begin{notation}\label{n:FM-IM-at-object}
    Let $\mathcal{M}$ be a class of propositional models of a Boolean doctrine $\P \colon \C\op \to  \BA$ at an object $S \in \C$.
    For each $X \in \C$, define
    \begin{align*}
        F_X^{S, \mathcal{M}} &\coloneqq \{\alpha \in \P(S \times X) \mid  \text{for all }(M,\m, s) \in \mathcal{M},\text{ for all }x \in M(X),\, (s , x) \in \m_{S\times X}(\alpha)\},\\
        I_X ^{S, \mathcal{M}} &\coloneqq \{\alpha \in \P(S \times X) \mid \text{for all }(M,\m, s) \in \mathcal{M},  \text{ not all }x \in M(X) \text{ satisfy } (s , x) \in \m_{S\times X}(\alpha)\}.
    \end{align*}
\end{notation}

Roughly speaking, 
\begin{itemize}
    \item $F^{S,\mathcal{M}}$ consists of all the formulas $\alpha(S , X)$ such that for all $(M ,s)$ in $\mathcal{M}$ we have $M,[S \mapsto s]\vDash\forall X\, \alpha(S, X)$,
    
    \item $I^{S,\mathcal{M}}$ consists of all the formulas $\alpha(S , X)$ such that for all $(M ,s)$ in $\mathcal{M}$ we have $M,[S \mapsto s]\vDash\lnot (\forall X\, \alpha(S, X))$.
\end{itemize}

\begin{remark}\label{r:univ-vald-fmlas-at-object}
    We translate \cref{n:FM-IM-at-object} to the classical syntactic setting.
    For this, we fix $k \in \N$.
    Let $\{s_1, \dots, s_k, x_1, x_2, \dots\}$ be a set of variables and $\mathcal{T}$ a universal theory.
    Let $\mathcal{M}$ be a class of tuples $(M, c_1, \dots, c_k)$ where $M$ is a model of $\mathcal{T}$ and $c_1, \dots, c_k \in M$.
    For each $n \in \N$ we define
    \begin{align*}
        F_n^{k,\mathcal{M}}
        \coloneqq \{ \alpha(s_1, \dots, s_k, x_1, \dots, x_{n}) \text{ quantifier-free}\mid{}& \text{for all } (M, c_1, \dots, c_k) \in \mathcal{M},\\
        & M, [(s_i \mapsto c_i)_i] \vDash \forall x_1 \dots \forall x_n\, \alpha(s_1, \dots, s_k, x_1, \dots, x_n)\},\\
        I_n^{k,\mathcal{M}} \coloneqq \{ \alpha(s_1, \dots, s_k, x_1, \dots, x_{n}) \text{ quantifier-free}\mid{}& \text{for all } (M, c_1, \dots, c_k) \in \mathcal{M},\\
        & M, [(s_i \mapsto c_i)_i] \nvDash \forall x_1 \dots \forall x_n\, \alpha(s_1, \dots, s_k, x_1, \dots, x_n)\}.
    \end{align*}
\end{remark}

We introduce \emph{universal filters at an object $S$}, meant to characterize the families of the form $F^{S,\mathcal{M}}$ for $\mathcal{M}$ an arbitrary class of propositional models at $S$ (see \cref{n:FM-IM-at-object}).

\begin{definition}[Universal filter at an object]\label{d:filter-at-object}
    A \emph{universal filter for a Boolean doctrine $\P \colon \C\op \to \BA$ at $S \in \C$} is a family $(F_X)_{X \in \C}$, with $F_X \subseteq \P(S \times X)$ for $X \in \C$, with the following properties.
    \begin{enumerate}
        \item\label{forall-reindex-at-object} For all $f\colon S\times X\to Y$ and $\alpha \in F_Y$, $\P(\langle \pr^{S\times X}_S,f\rangle)(\alpha) \in F_X$.
        
        \item For all $X \in \C$, $F_X$ is a filter of $\P(S \times X)$.
    \end{enumerate}
\end{definition}

\begin{remark} 
    We translate \cref{d:filter-at-object} to the classical syntactic setting.
    For this, we fix $k \in \N$.
    Let $\{s_1, \dots, s_k, x_1, x_2, \dots\}$ be a set of variables and $\mathcal{T}$ a universal theory.
    A \emph{universal filter for $\mathcal{T}$ at $k$} is a family $(F_n)_{n \in \N}$, with $F_n$ a set of quantifier-free formulas with $s_1, \dots, s_k, x_1, \dots, x_n$ as (possibly dummy) free variables, with the following properties.
    \begin{enumerate}
        \item For every $n,m \in \N$, for every formula $\alpha(s_1, \dots, s_k, x_1, \dots, x_m) \in F_m$ and for every $m$-tuple $(f_i(s_1, \dots, s_k, x_1, \dots,x_n))_{i = 1, \dots, m}$ of $(k + n)$-ary terms, 
        \[
        \alpha(s_1, \dots, s_k, f_1(s_1, \dots, s_k, x_1, \dots, x_n), \dots, f_m(s_1, \dots, s_k, x_1, \dots, x_n)) \in F_n.
        \]
        \item For all $n \in \N$, 
        \begin{enumerate}
            \item for all quantifier-free formulas $\alpha(s_1, \dots, s_k, x_1, \dots, x_n)$ and $\beta(s_1, \dots, s_k, x_1, \dots, x_n)$, if we have $\alpha(s_1, \dots, s_k, x_1, \dots, x_n) \in F_n$ and $\alpha(s_1, \dots, s_k, x_1, \dots, x_n) \vdash_\T \beta(s_1, \dots, s_k,x_1, \dots, x_n)$, then we have $\beta(s_1, \dots, s_k,x_1, \dots, x_n) \in F_n$;
            \item 
            for all $\alpha_1(s_1, \dots, s_k, x_1, \dots, x_n), \alpha_2(s_1, \dots, s_k, x_1, \dots, x_n) \in F_n$, we have 
            \[\alpha_1(s_1, \dots, s_k, x_1, \dots, x_n) \land \alpha_2(s_1, \dots, s_k, x_1, \dots, x_n) \in F_n;\]
            \item $\top(s_1, \dots, s_k, x_1, \dots, x_n) \in F_n$.
        \end{enumerate}
    \end{enumerate}
    
    These conditions are satisfied by any family $(F^{k,\mathcal{M}}_n)_{n\in\N}$ defined by a class $\mathcal{M}$ as in \cref{r:univ-vald-fmlas-at-object}.
\end{remark}

\begin{remark}
    A universal filter for a Boolean doctrine $\P \colon \C\op \to \BA$ at $S \in \C$ is a universal filter for the Boolean doctrine $\P_S\colon \C_S\op \to \BA$ obtained by adding a constant of type $S$ for $\P$ (see \cref{r:const}).
\end{remark}

\begin{definition}
    A universal filter $F$ for a Boolean doctrine $\P \colon \C\op \to  \BA$ at an object $S\in \C$  is \emph{consistent} if $\bot_{\P(S)}\notin F_\tmn$.
\end{definition}

\begin{theorem} \label{t:characterization-filters-at-object}
    Let $\P \colon \C\op \to \BA$ be a Boolean doctrine, with $\C$ small, and let $S \in \C$.
    Let $F = (F_X)_{X \in \C}$ be a family with $F_X \subseteq \P(S \times X)$ for each $X \in \C$.
    The following are equivalent.
    \begin{enumerate}
        \item \label{i:exist-class-ne-at-object}
        There is a class (resp.\ nonempty class) $\mathcal{M}$ of propositional models of $\P$ at $S$ such that, for every $X \in \C$,
        \[
            F_X = \{\alpha \in \P(S \times X) \mid{}  \text{for all }(M,\m, s) \in \mathcal{M},\text{ for all }x \in M(X),\,(s , x) \in \m_{S\times X}(\alpha)\}.
        \]
        
        \item \label{i:prop-univ-filt-ne-at-object}
        $F$ is a universal filter (resp.\ consistent universal filter) for $\P$ at $S$.
    \end{enumerate}
\end{theorem}
\begin{proof}
This follows from \cref{t:characterization-filters} applied to the Boolean doctrine $\P_S$ obtained from $\P$ by adding a constant of type $S$ and from \cref{l:model-model-at-S}.
\end{proof}

\subsection{Semantic characterization of universal ideals over fixed free variables}

We introduce \emph{universal ideals at an object $S$}, meant to characterize the families of the form $I^{S,\mathcal{M}}$ for $\mathcal{M}$ an arbitrary class of propositional models at $S$ (see \cref{n:FM-IM-at-object}).

\begin{definition}[Universal ideal at an object] \label{d:univ-ideal-at-object}
    A \emph{universal ideal for a Boolean doctrine $\P$ at $S \in \C$} is a family $(I_X)_{X \in \C}$, with $I_X \subseteq \P(S \times X)$ for $X \in \C$, with the following properties.
    \begin{enumerate}
        \item \label{i:conjunction-at-object} 
        For all $m \in \N$, $(f_j \colon S \times X \to Y)_{j = 1, \dots, m}$ and $\alpha \in \P(S \times Y)$, if $\bigwedge_{j = 1}^m \P(\langle \pr^{S\times X}_S, f_j \rangle)(\alpha) \in I_X$ then $\alpha \in I_Y$.
        
        \item 
        For all $X\in\C$, $I_X$ is downward closed.
        
        \item \label{i:ideal-closed-binary-joins-at-object} 
        For all $\alpha_1 \in I_{X_1}$ and $\alpha_2 \in I_{X_2}$, in $\P(S\times X_1\times X_2)$ we have $\P( \pr^{S\times X_1\times X_2}_{S\times X_1})(\alpha_1)\lor \P(\pr^{S\times X_1\times X_2}_{S\times X_2})(\alpha_2) \in I_{X_1 \times X_2}$.
        
        \item \label{i:ideal-bot-tmn-at-object}
        $\bot_{\P(S)} \in I_\tmn$.
    \end{enumerate}
\end{definition}

\begin{remark}
    We translate \cref{d:univ-ideal-at-object} to the classical syntactic setting.
    For this, we fix $k \in \N$.
    Let $\{s_1, \dots, s_k, x_1, x_2, \dots\}$ be a set of variables and $\mathcal{T}$ a universal theory.
    A \emph{universal ideal for $\mathcal{T}$ at $k$} is a family $(I_n)_{n \in \N}$, with $I_n$ a set of quantifier-free formulas with $s_1, \dots, s_k, x_1, \dots, x_n$ as (possibly dummy) free variables, with the following properties.
    \begin{enumerate}
        \item 
        For all $p,q,m \in \N$, every $(m\cdot q)$-tuple $(f_{j,h}(s_1, \dots, s_k, x_1, \dots, x_p))_{j \in \{1, \dots, m\},\, h \in \{1, \dots, q\}}$ of $(k + p)$-ary terms and every quantifier-free formula $\alpha(s_1, \dots, s_k, x_1, \dots, x_q)$, if 
        \[
            \bigwedge_{j = 1}^m \alpha(s_1, \dots, s_k, f_{j,1}(s_1, \dots, s_k, x_1, \dots, x_p), \dots, f_{j,q}(s_1, \dots, s_k, x_1, \dots, x_p)) \in I_p,
        \]
        then 
        \[
            \alpha(s_1, \dots, s_k, 
            x_1, \dots, x_q) \in I_q.
        \]
        
        \item 
        For all $n \in \N$, all quantifier-free formulas $\alpha(s_1, \dots, s_k, x_1, \dots, x_n)$ and $\beta(s_1, \dots, s_k, x_1, \dots, x_n)$, if $\beta(s_1, \dots, s_k, x_1, \dots, x_n) \in I_n$ and $\alpha(s_1, \dots, s_k, x_1, \dots, x_n) \vdash_\T \beta(s_1, \dots, s_k, x_1, \dots, x_n)$, then 
        \[
        \alpha(s_1, \dots, s_k, x_1, \dots, x_n) \in I_n.
        \]
        
        \item 
        For all $n_1,n_2 \in \N$, $\alpha_1(s_1, \dots, s_k, x_1, \dots, x_{n_1}) \in I_{n_1}$ and $\alpha_2(s_1, \dots, s_k, x_1, \dots, x_{n_2}) \in I_{n_2}$,  we have
        \[
        \alpha_1(s_1, \dots, s_k, x_1, \dots, x_{n_1}) \lor \alpha_2(s_1, \dots, s_k, x_{n_1 +1}, \dots, x_{n_1 + n_2}) \in I_{n_1 + n_2}.
        \]
        
        \item $\bot(s_1, \dots, s_k) \in I_0$.
    \end{enumerate}
    
    These conditions are satisfied by any family $(I^{k,\mathcal{M}}_n)_{n\in\N}$ defined by a class $\mathcal{M}$ as in \cref{r:univ-vald-fmlas-at-object}.
\end{remark}

\begin{remark}
   A universal ideal for a Boolean doctrine $\P \colon \C\op \to  \BA$ at an object $S \in \C$ is a universal ideal for the Boolean doctrine $\P_S\colon \C_S\op \to \BA$ obtained by adding a constant of type $S$ for $\P$ (see \cref{r:const}).
\end{remark}

\begin{definition}
    A universal ideal $(I_X)_{X \in \C}$ for a Boolean doctrine $\P \colon \C\op \to \BA$ at an object $S \in \C$ is \emph{consistent} if $\top_{\P(S)}\notin I_\tmn$.
\end{definition}

\begin{theorem} \label{t:characterization-ideals-at-object}
    Let $\P \colon \C\op \to \BA$ be a Boolean doctrine, with $\C$ small, and let $S \in \C$.
    Let $I = (I_X)_{X \in \C}$ be a family with $I_X \subseteq \P(S \times X)$ for each $X \in \C$. 
    The following are equivalent.
    \begin{enumerate}
        \item \label{i:exist-class-ne-id-at-object}
        There is a class (resp.\ nonempty class) $\mathcal{M}$ of propositional models of $\P$ at $S$ such that, for every $X \in \C$,
        \[
            I_X = \{\alpha \in \P(S \times X) \mid \text{for all }(M,\m, s) \in \mathcal{M},  \text{ not all }x \in M(X) \text{ satisfy }(s, x) \in \m_{S \times X}(\alpha)\}.
        \]
        
        \item \label{i:prop-univ-filt-ne-id-at-object}
        $I$ is a universal ideal (resp.\ consistent universal ideal) for $\P$ at $S$.
    \end{enumerate}
\end{theorem}

\begin{proof}
This follows from \cref{t:characterization-ideals} applied to the Boolean doctrine $\P_S$ obtained from $\P$ by adding a constant of type $S$ and from \cref{l:model-model-at-S}.
\end{proof}

\subsection{Semantic characterization of universal filter-ideal pairs over fixed free variables}

\begin{definition}[Universal filter-ideal pair at an object]\label{d:fil-id-pair-at-obj}
    A \emph{universal filter-ideal pair for a Boolean doctrine $\P \colon \C\op \to  \BA$ at an object $S \in \C$} is a pair $(F, I)$ where $F$ is a universal filter for $\P$ at $S$, $I$ is a universal ideal for $\P$ at $S$, and the following conditions hold for all $Y\in\C$ and $\alpha \in \P(S \times Y)$.
    \begin{enumerate}
        \item \label{i:connecting-oneatobj} 
        For all $X\in\C$, $n \in \N$, $(f _i \colon S \times X \to Y)_{i=1,\dots,n}$ and $\beta \in F_X$, if $\beta \land \bigwedge_{i = 1}^n \P(\langle \pr^{S\times X}_S, f_i \rangle)(\alpha) \in I_X$, then $\alpha \in I_Y$.

        \item\label{i:connecting-two-at-obj} 
        For all $Z\in\C$ and $\gamma \in I_Z$, if $\P(\pr^{S\times Y\times Z}_{S\times Y})(\alpha) \lor \P(\pr^{S\times Y\times Z}_{S\times Z})(\gamma) \in F_{Y \times Z}$, then $\alpha \in F_Y$. 
    \end{enumerate}
    
    We say that $(F,I)$ is consistent if $F_X\cap I_X=\varnothing$ for every $X\in\C$.
\end{definition}

\begin{remark}
   A universal filter-ideal pair for a Boolean doctrine $\P \colon \C\op \to  \BA$ at an object $S \in \C$ is a universal filter-ideal pair for the Boolean doctrine $\P_S\colon \C_S\op \to \BA$ obtained by adding a constant of type $S$ for $\P$ (see \cref{r:const}).
\end{remark}

\begin{theorem} \label{t:characterization-filter-ideal-pair-at-obj}
    Let $\P \colon \C\op \to \BA$ be a Boolean doctrine, with $\C$ small, and let $S \in \C$.
    Let $F = (F_X)_{X \in \C}$ and $I = (I_X)_{X \in \C}$ be families with $F_X, I_X \subseteq \P(S \times X)$ for each $X \in \C$.
    The following are equivalent.
    \begin{enumerate}
        \item \label{i:fi-arbitrary-class-at-obj}
        There is a class (resp.\ nonempty class) $\mathcal{M}$ of propositional models of $\P$ at $S$ such that, for every $X \in \C$,
        \begin{align*}
            F_X & = \{\alpha \in \P(S \times X) \mid \text{for all }(M,\m, s) \in \mathcal{M},\text{ for all }x \in M(X),\,(s , x) \in \m_{S\times X}(\alpha)\},\\
            I_X &= \{\alpha \in \P(S \times X) \mid \text{for all }(M,\m, s) \in \mathcal{M},  \text{ not all }x \in M(X) \text{ satisfy }(s, x) \in \m_{S \times X}(\alpha)\}.
        \end{align*}
        
        \item \label{i:fi-arbitrary-pair-at-obj}
        $(F, I)$ is a universal filter-ideal pair (resp.\ consistent universal filter-ideal pair) for $\P$ at $S$.
    \end{enumerate}
\end{theorem}
\begin{proof}
This follows from \cref{t:characterization-filter-ideal-pair} applied to the Boolean doctrine $\P_S$ obtained from $\P$ by adding a constant of type $S$ and from \cref{l:model-model-at-S}.
\end{proof}

\section{The case with equality (``nothing changes'')}
\label{s:equality}

In this appendix, we adapt the results of the previous sections to the context with equality.
However, ``nothing changes'':
\begin{itemize}
    \item The notions of universal ultrafilters, filters and ideals still work perfectly without any adjustment: their semantic characterization remains unchanged if we restrict to models preserving equality.

    \item
    The free one-step construction is the same:
    the way to add the first layer of quantifier alternation depth to an elementary Boolean doctrine $\P$ coincides with the way to add the first layer of quantifier alternation depth to $\P$ seen simply as a Boolean doctrine.
\end{itemize}

We start by recalling a modern definition of elementarity in the context of Boolean doctrines\footnote{We give the definitions in the Boolean case, since this is the context we are interested in in this paper. However, these definitions are usually stated in the \emph{primary} context (where a \emph{primary doctrine} is a functor $\P \colon \C\op \to \mathsf{InfSL}$ where $\C$ is a category with finite products and $\mathsf{InfSL}$ is the category of inf-semilattices), see \cite{MaiettiRosolini13,MaiettiRosolini15}. }; this notion, which originates in Lawvere's work \cite{Lawvere70}, amounts to the possibility of interpreting equality.

\begin{definition}[\cite{MaiettiRosolini13}]\label{def:equality-classical}
    A Boolean doctrine $\P\colon\C\op\to\BA$ is \emph{elementary} if for every $X\in \C$ there is an element $\delta_X\in\P(X\times X)$ such that, for every $Y\in\C$, the assignment
    \begin{align*}
       \text{\AE}^X_Y\colon \P(Y \times X)&\longrightarrow \P(Y\times X\times X)\\
       \alpha &\longmapsto\P(\ple{\pr^{Y\times X\times X}_1,\pr^{Y\times X\times X}_2})(\alpha)\land \P(\ple{\pr^{Y\times X\times X}_2,\pr^{Y\times X\times X}_3})(\delta_X)
    \end{align*}
        is a left adjoint to the function 
    \[  \P(\id_Y\times\Delta_X) \colon \P(Y\times X\times X) \to \P(Y \times X).\]
\end{definition}

An equivalent definition, as shown in \cite[Prop.~2.5]{EmPaRo20}, is the following one.

\begin{definition}\label{def:equality}
    A Boolean doctrine $\P\colon\C\op\to\BA$ is \emph{elementary} if there is a family of elements $(\delta_X)_{X\in\C}$, with $\delta_X\in\P(X\times X)$ for each $X\in \C$, such that, for all $X,Y\in\C$,
    \begin{enumerate}
        \item\label{i:def=1} $\top_{\P(X)}\leq\P(\Delta_X)(\delta_X)$ in $\P(X)$;
        \item\label{i:def=2} for every $\alpha\in\P(X)$, $\P(\pr^{X\times X}_1)(\alpha)\land \delta_X\leq \P(\pr^{X\times X}_2)(\alpha)$ in $\P(X\times X)$;
        \item\label{i:def=3} $\P(\pr^{X\times Y\times X\times Y}_{X\times X})(\delta_X)\land \P(\pr^{X\times Y\times X\times Y}_{Y\times Y})(\delta_Y)\leq \delta_{X\times Y}$ in $\P(X\times Y\times X\times Y)$.
    \end{enumerate}
\end{definition}

The proof of \cite[Prop.~2.5]{EmPaRo20} shows more: a family $(\delta_X)_{X\in\C}$ satisfies the condition in \cref{def:equality-classical} if and only if it satisfies the conditions in \cref{def:equality}.

Informally, condition \eqref{i:def=1} in \cref{def:equality} is the reflexivity of the equality relation, i.e.\ $\vdash x=x$. Condition \eqref{i:def=2} is the substitutivity property, i.e.\ $\alpha(x)\land (x=x')\vdash \alpha(x')$. Condition \eqref{i:def=3} can be roughly interpreted as ``$(x=x')\land (y=y')\vdash (x,y)=(x',y')$'', meaning that two pairs coincide if both entries coincide.

\begin{remark}[{E.g.\ \cite[Rem.~6.3]{AbbadiniGuffanti}}]
    There is a unique family $(\delta_X)_{X\in\C}$ satisfying the conditions in \cref{def:equality-classical} (or, equivalently, in \cref{def:equality}).
\end{remark}

For every $X\in\C$, the element $\delta_X$ is called \emph{fibered equality on $X$}.

\begin{example}\label{fbf=}
    For a first-order theory $\mathcal{T}$ in a language with equality, one can define a syntactic doctrine $\LT_{=}^\T$ similarly to \cref{fbf}. In this case, one has to consider also the formulas involving equality, and take the equivalence relation of equiprovability modulo $\T$ in first-order logic with equality.
    The syntactic doctrine $\LT_{=}^\T$ is an elementary first-order Boolean doctrine: for every context $\vec x = (x_1,\dots,x_n)$, the fibered equality $\delta_{\vec x}\in \LT_=^\mathcal{T}(x_1,\dots,x_n,x'_1,\dots,x'_n)$ is the formula
    \[
    (x_1=x'_1)\land\dots\land (x_n=x'_n).
    \]
\end{example}

\begin{example}
    The subset doctrine $\mathscr{P}\colon\Set\op\to\BA$ is elementary: the fibered equality on $X \in \Set$ is
    \[
    \Delta_X=\{(x,x) \in X \times X\mid x \in X\}\in \mathscr{P}(X \times X).
    \] 
\end{example}

\begin{remark}[{E.g.\ \cite[Rem.~6.4]{AbbadiniGuffanti}}]\label{r:delta-symm}
    For every elementary Boolean doctrine $\P \colon \C\op \to \BA$ and $X \in \C$,
    \[
    \delta_X \leq \P(\ple{\pr^{X\times X}_2,\pr^{X\times X}_1})(\delta_X).
    \]
\end{remark}

\begin{remark}\label{r:inverse_3_equality}
    Also the converse direction of \cref{def:equality}\eqref{i:def=3} holds. This is easily seen from the equality $\delta_{X\times Y}=\text{\AE}^{X\times Y}_\tmn(\top_{\P(X\times Y)})$ and the adjunction $\text{\AE}^{X\times Y}_\tmn\dashv\P(\ple{\pr^{X\times Y}_X,\pr^{X\times Y}_Y,\pr^{X\times Y}_X,\pr^{X\times Y}_Y})$.
\end{remark}

\begin{definition}
    Let $\P\colon\C\op\to \BA$ and $\mathbf{R}\colon \cat{D}\op\to \BA$ be two elementary Boolean doctrines. 
    An \emph{elementary Boolean doctrine morphism from $\P$ to $\R$} is a Boolean doctrine morphism $(M,\m)\colon \P\to\R$ that preserves fibered equalities, i.e.\ such that for every $X\in\C$
    \[\m_{X\times X}(\delta^\P_X)=\delta^\R_{M(X)}.\]
\end{definition}

\begin{definition}
    An \emph{elementary propositional model of a Boolean doctrine $\P \colon \C\op \to \BA$}
    is an elementary Boolean doctrine morphism from $\P$ to the subset doctrine $\mathscr{P}$.
\end{definition}

\begin{definition}
    An \emph{elementary first-order Boolean doctrine}
    $\P\colon\C\op\to\BA$ is a first-order Boolean doctrine $\P$ that is elementary.
    
    Given two elementary first-order Boolean doctrines $\P$ and $\R$, an \emph{elementary first-order Boolean doctrine morphism from $\P$ to $\R$} is a first-order Boolean doctrine morphism $(M,\m)\colon \P\to\R$ that is elementary.
\end{definition}

\begin{corollary}[{\cite[Thm.~6.16]{AbbadiniGuffanti}}]\label{c:equality-only-p0}
    Let $(\id_\C,\mathfrak{i})\colon \P\hookrightarrow\P^\EA$ be a quantifier completion of a Boolean doctrine $\P$. If $\P$ is elementary, then $\P^\EA$ is elementary, and the Boolean doctrine morphism $(\id_\C,\mathfrak{i})$ is elementary.
\end{corollary}

\begin{remark}[{\cite[Rem.~6.17]{AbbadiniGuffanti}}]
    A quantifier completion $(\id_\C,\mathfrak{i})\colon \P\hookrightarrow\P^\EA$ of an elementary Boolean doctrine $\P$ has the following property: for every elementary first-order Boolean doctrine $\R$ and every elementary Boolean doctrine morphism $(M,\m)\colon\P\to\R$, the unique first-order Boolean doctrine morphism $(N,\n)\colon \P^\EA\to\R$ such that $(M,\m)=(N,\n)\circ(\id_\C,\mathfrak{i})$ is elementary.
\end{remark}

\begin{corollary}[{\cite[Cor.~6.18]{AbbadiniGuffanti}}]
    The forgetful functor from the category of elementary first-order Boolean doctrines over a small category to the category of elementary Boolean doctrines
    over a small category has a left adjoint.
\end{corollary}

\begin{example}[{\cite[Example~6.19]{AbbadiniGuffanti}}]\label{ex:sanitycheck=}
    Let $\mathcal{T}$ be a universal theory in a first-order language $(\F,\mathbb{P})$ with equality. 
    Let $\LT_=^\mathcal{T}\colon\Ctx_\F\op\to\BA$ be the syntactic doctrine of $\mathcal{T}$ as in \cref{fbf=} and $\LT^\mathcal{T}_{=,0}$ be the subfunctor of $\LT_=^\mathcal{T}$ consisting of quantifier-free formulas modulo $\T$. Let $i\colon\LT^\mathcal{T}_{=,0}\to\LT_=^\mathcal{T}$ be the componentwise inclusion. Then $(\id_{\Ctx_\F},i)\colon\LT^\mathcal{T}_{=,0}\hookrightarrow\LT_=^\mathcal{T}$ is a quantifier completion of $\LT^\mathcal{T}_{=,0}$.
\end{example}

In the case with equality, the notion of a universal ultrafilter does not need any adjustment.
Indeed, universal ultrafilters of elementary Boolean doctrines still capture the classes of universally valid formulas in some model that preserves equality.
To prove this, we can piggyback on the results where we do not have equality. The key idea for this piggybacking is that, whenever we have a model that does not necessarily respect equality (so that equality is interpreted as an equivalence relation, but not necessarily as the identity), we can obtain a propositional model that preserves the equality by quotienting by the equivalence relation in which the equality is interpreted.

\begin{lemma} \label{l:from-prop-model-to-elementary-prop-model}
    Let $\P \colon \C\op \to \BA$ be an elementary Boolean doctrine. For every propositional model $(N,\n)$ of $\P$ there is an elementary propositional model $(M,\m)$ of $\P$ such that $F^{(M,\m)}=F^{(N,\n)}$ (in the sense of \cref{n:FM-IM-one-model}), i.e.\ such that, for every $X\in\C$,
    \[
        \{\alpha \in \P(X) \mid \text{for all }x \in M(X),\, x \in \m_X(\alpha)\}=\{\alpha \in \P(X) \mid \text{for all }x \in N(X),\, x \in \n_X(\alpha)\}.
    \]
\end{lemma}
\begin{proof}
    For each object $X\in\C$, we define the binary relation ${\sim_X}\coloneqq\n_{X\times X}(\delta_X)$ on $N(X)$ and we prove that it is an equivalence relation.
    \begin{itemize}
        \item 
        Reflexivity: by \cref{def:equality}\eqref{i:def=1}, in $\P(X)$ we have $\top_{\P(X)}\leq\P(\Delta_X)(\delta_X)$, and so in $\mathscr{P}(N(X))$ we have
        \[
        N(X)=\n_X(\top_{\P(X)})\subseteq \n_X(\P(\Delta_X)(\delta_X))=\ple{\id_{N(X)},\id_{N(X)}}^{-1}[\n_{X\times X}(\delta_X)],
        \]
        i.e., for every $a\in N(X)$ we have $a\sim_X a$.
        
        \item 
        Symmetry: by \cref{r:delta-symm}, in $\P(X\times X)$  we have $\delta_X \leq \P(\ple{\pr^{X\times X}_2,\pr^{X\times X}_1})(\delta_X)$, and so in $\mathscr{P}(N(X)\times N(X))$ we have
        \[
            \n_{X\times X}(\delta_X) \subseteq \n_{X\times X}(\P(\ple{\pr^{X\times X}_2,\pr^{X\times X}_1})(\delta_X))=\ple{\pr^{N(X)\times N(X)}_2,\pr^{N(X)\times N(X)}_1}^{-1}[\n_{X\times X}(\delta_X)],
        \]    
        i.e., for every $a,b\in N(X)$ such that $a\sim_X b$ we have $b\sim_X a$.
        
        \item 
        Transitivity: let $a,b,c\in N(X)$ be such that $a\sim_X b \sim_X c$. First of all, we observe that in $\P(X\times X\times X\times X)$
        \begin{align*}
            &\P(\ple{\pr_1,\pr_2})(\delta_X)\land \P(\ple{\pr_1,\pr_3})(\delta_X)\land \P(\ple{\pr_2,\pr_4})(\delta_X)\\
            &\leq \P(\ple{\pr_1,\pr_2})(\delta_X)\land \delta_{X\times X}&&\text{by \cref{def:equality}\eqref{i:def=3}}\\
            &\leq \P(\ple{\pr_3,\pr_4})(\delta_X)&&\text{by \cref{def:equality}\eqref{i:def=2}},
        \end{align*} 
        where $\pr_i\coloneqq\pr^{X\times X\times X\times X}_i$, for $i=1,\dots,4$.
        Applying $\n_{X\times X\times X\times X}$ to the outermost inequality, and using naturality of $\n$, in $\mathscr{P}(N(X)\times N(X)\times N(X)\times N(X))$ we get
        \begin{equation*}
            \ple{\pr'_1,\pr'_2}^{-1}[\n_{X\times X}(\delta_X)]\cap \ple{\pr'_1,\pr'_3}^{-1}[\n_{X\times X}(\delta_X)]\cap \ple{\pr'_2,\pr'_4}^{-1}[\n_{X\times X}(\delta_X)]\subseteq \ple{\pr'_3,\pr'_4}^{-1}[\n_{X\times X}(\delta_X)],
        \end{equation*} 
        where $\pr'_i\coloneqq\pr^{N(X)\times N(X)\times N(X)\times N(X)}_i$, for $i=1,\dots,4$.
        By reflexivity and symmetry, $(b,b,a,c)$ belongs to the intersection on the left, and thus $a\sim_X c$.
    \end{itemize}
    Thus, ${\sim_X} = \n_{X\times X}(\delta_X)$ is an equivalence relation, as claimed.
    
    Let $f\colon X\to Y$ be a morphism in $\C$. We prove that the function $N(f)\colon N(X)\to N(Y)$ is well-defined on the quotients.
    Let $a,b\in N(X)$ be such that $a\sim_X b$; we show that $N(f)(a)\sim_Y N(f)(b)$.
    For every $Z\in\C$ and $\gamma\in\P(Z\times Z)$ we have
    \begin{equation*}
        \P(\ple{\pr^{Z\times Z}_1,\pr^{Z\times Z}_1})(\gamma)\land\delta_Z\leq\gamma
    \end{equation*}
    (see e.g.\ eq.\ (1) in the proof of \cite[Prop.~2.5]{EmPaRo20}).
    Thus, in $\P(X\times X)$ we have $\P(\ple{\pr^{X\times X}_1,\pr^{X\times X}_1})(\P(f\times f)(\delta_Y))\land\delta_X\leq \P(f\times f)(\delta_Y)$.
    Since $(f\times f)\circ\ple{\pr^{X\times X}_1,\pr^{X\times X}_1}=\Delta_Y\circ f\circ\pr^{X\times X}_1$, by \cref{def:equality}\eqref{i:def=1} we get $\delta_X\leq \P(f\times f)(\delta_Y)$.
    It follows that in $\mathscr{P}(N(X)\times N(X))$ we have \[\n_{X\times X}(\delta_X)\subseteq \n_{X\times X}(\P(f\times f)(\delta_Y))=(N(f)\times N(f))^{-1}[\n_{Y\times Y}(\delta_Y)].\]
    Since $a\sim_X b$, we obtain $N(f)(a)\sim_Y N(f)(b)$, as desired.
    
    Hence, we can define the functor
    \begin{equation*}
        M\coloneqq N/{\sim}\colon\C\to\Set.
    \end{equation*}
    Next, we show that $M$ preserves finite products.
    The terminal object is trivially preserved since $M(\tmn)$ is a quotient of the singleton $N(\tmn)$.
    Applying $\n_{X\times Y\times X\times Y}$ to both sides of the equality $\P(\pr^{X\times Y\times X\times Y}_{X\times X})(\delta_X)\land \P(\pr^{X\times Y\times X\times Y}_{Y\times Y})(\delta_Y)= \delta_{X\times Y}$ (see \cref{r:inverse_3_equality}), and using naturality of $\n$, we get
    \[
    (\pr^{N(X)\times N(Y)\times N(X)\times N(Y)}_{N(X)\times N(X)})^{-1}(\n_{X\times X}(\delta_X))\cap(\pr^{N(X)\times N(Y)\times N(X)\times N(Y)}_{N(Y)\times N(Y)})^{-1}(\n_{Y\times Y}(\delta_Y))=\n_{X\times Y \times X \times Y}(\delta_{X\times Y}),
    \]
    i.e.\ for all $a,c\in N(X)$ and $b,d\in N(Y)$ we have $(a,b)\sim_{X\times Y}(c,d)$ if and only if $a\sim_X c$ and $b\sim_Y d$. So,
    \[
    M(X\times Y)=N(X\times Y)/{\sim_{X\times Y}} = (N(X)\times N(Y))/{\sim_{X\times Y}} = N(X)/{\sim_X} \times N(Y)/{\sim_Y} = M(X)\times M(Y).
    \]
    Next we prove that, for all $X\in\C$ and $\alpha\in\P(X)$, the subset $\n_X(\alpha)$ of $N(X)$ is closed under $\sim_X$. Let $a\in\n_X(\alpha)$ and $a\sim_X b$; we show that $b\in\n_X(\alpha)$.
    This is obtained by applying $\n_{X\times X}$ to both sides of the inequality $\P(\pr^{X\times X}_1)(\alpha)\land\delta_X\leq \P(\pr^{X\times X}_2)(\alpha)$ and using naturality of $\n$, from which we get
    \[(\pr^{N(X)\times N(X)}_1)^{-1}[\n_X(\alpha)]\cap\n_{X\times X}(\delta_X)\subseteq (\pr^{N(X)\times N(X)}_2)^{-1}[\n_X(\alpha)].\]
    
    Then, for every $X\in\C$, the function 
    \begin{align*}
        \m_X\colon\P(X)& \longrightarrow\mathscr{P}(M(X))\\
        \alpha & \longmapsto \{[a]_{\sim_X}\mid a\in\n_X(\alpha)\}
    \end{align*}
    is well-defined.
    
    It is easy to see that $(M,\m)$ is a propositional model of $\P$.
    It is also an elementary propositional model, since the fibered equalities are preserved:
    \begin{equation*}
        \m_{X\times X}(\delta_X)=\{[(a,b)]_{\sim_{X\times X}}=([a]_{\sim_X},[b]_{\sim_X})\in M(X)\times M(X)\mid a \sim_X b\}=\Delta_{M(X)}.
    \end{equation*}

    To conclude, we observe that for all $X\in\C$ and $\alpha\in\P(X)$ we have $\n_X(\alpha)= N(X)$ if and only if $\m_X(\alpha)=M(X)$; thus,  $F^{(M,\m)}=F^{(N,\n)}$, as desired.
\end{proof}

\begin{theorem}[Semantic characterizations of universal ultrafilters, filters and ideals, with equality]
    Let $\P \colon \C\op \to \BA$ be an elementary Boolean doctrine, with $\C$ small.
    Let $A = (A_X)_{X \in \C}$ be a family with $A_X \subseteq \P(X)$ for each $X \in \C$. Then
    \begin{enumerate} 
    \item $A$ is a universal ultrafilter for $\P$ if and only if there is an elementary propositional model $(M,\m)$ of $\P$ such that, for every $X \in \C$,
        \[
        A_X=\{\alpha \in \P(X) \mid \text{for all }x \in M(X),\, x \in \m_X(\alpha)\}.
        \] 
    \item 
        $A$ is a universal filter for $\P$ if and only if there is a class $\mathcal{M}$ of elementary propositional models of $\P$ such that, for every $X \in \C$,
        \[
            A_X = \{\alpha \in \P(X) \mid \text{for all }(M,\m) \in \mathcal{M},\text{ for all }x \in M(X),\, x \in \m_X(\alpha)\}.
        \]
    \item $A$ is a universal ideal for $\P$ if and only if there is a class $\mathcal{M}$ of elementary propositional models of $\P$ such that, for every $X \in \C$,
        \[
            A_X = \{\alpha \in \P(X) \mid \text{for all }(M,\m) \in \mathcal{M},  \text{ not all }x \in M(X) \text{ satisfy } x \in \m_X(\alpha)\}.
        \]
    \end{enumerate}
\end{theorem}

\begin{proof}
    By \cref{t:characterization-ultrafilters,t:characterization-filters,t:characterization-ideals,l:from-prop-model-to-elementary-prop-model}.
\end{proof}

What about the free one-step construction in the elementary case? The description of how to freely add one layer of quantifier alternation depth to an elementary Boolean doctrine $\P\colon \C\op \to \BA$ over a small base category $\C$ does not change from the non-elementary case. Indeed, let $(\id_\C,\mathfrak{i})\colon\P\hookrightarrow\P^\EA$ be the quantifier completion of $\P$; by \cref{c:equality-only-p0}, the first-order Boolean doctrine $\P^\EA$ and the Boolean doctrine morphism $(\id_\C,\mathfrak{i})$ are both elementary.

We recall from \cref{n:first-layer} the definition of the subfunctor $\P^\EA_1$ of $\P^\EA$, consisting of the ``formulas in $\P^\EA$ of quantifier alternation depth $\leq 1$'': for every $S\in\C$, the fiber $\P^\EA_1(S)$ is the Boolean subalgebra of $\P^\EA(S)$ generated by the union of the images of $\P(S\times Y)$ under $\fa{Y}{S}\colon \P^\EA(S\times Y)\to\P^\EA(S)$, for $Y$ ranging over objects of $\C$.
The Boolean doctrine $\P^\EA_1$ is elementary as well: it is easy to see that conditions (\ref{i:def=1}--\ref{i:def=3}) in \cref{def:equality} hold in $\P_1^\EA$ if they hold in $\P^\EA$ (with respect to the same fibered equalities).
As a consequence, in light of the isomorphism $\bar \m_1\colon \Free_1^{\P}\to \P^\EA_1$ (\cref{t:section-6}), the Boolean doctrine $\Free_1^{\P}$ is elementary, and the Boolean doctrine morphism $(\id_\C,\forall\tmn) \colon \P \hookrightarrow\Free_1^{\P}$ is elementary.

In conclusion, adding the first layer of quantifier alternation depth to an elementary Boolean doctrine $\P$ is the same as adding the first layer of quantifier alternation depth to $\P$ seen simply as a Boolean doctrine.

\section{Entailment relations} \label{s:entailment}

In this section, we provide some background on entailment relations that we use in \cref{s:construction}.
In \cref{s:construction}, we define certain Boolean algebras via generators and relations.
But it comes in handy that, roughly speaking, no other relation of a similar form can be deduced from the given ones.
For this, the ``Fundamental Theorem of Entailment Relations'', which we now recall, is a useful tool.

\subsection{Entailment relations and bounded distributive lattices}

Entailment relations are due to Dana Scott.

\begin{definition} \label{d:entailment-relation}
    An \emph{entailment relation} on a set $S$ is a binary relation $\vdash$ on the set of finite subsets of $S$ satisfying the following properties.
    \begin{enumerate}
        \item (Reflexivity)
        For all $x \in S$, $\{x\} \vdash \{x\}$.

        \item (Weakening / Monotonicity)
        For all finite subsets $X,X',Y,Y'$ of $S$, if $X \vdash Y$ then $X \cup X' \vdash Y \cup Y'$.

        \item (Cut / Distributivity)
        For all finite subsets $X$ and $Y$ of $S$ and all $s \in S$, if $X \vdash \{s\} \cup Y$ and $X \cup \{s\} \vdash Y$, then $X \vdash Y$.
    \end{enumerate}
\end{definition}

This notion of entailment relation can be seen as an abstract generalization of Gentzen's sequent calculus.

Given a set $S$ with a binary relation $R$ on finite subsets of $S$ one says that a map $f \colon S \to D$ from $S$ to a bounded distributive lattice $D$ \emph{preserves} $R$ if $X \mathrel{R} Y$ implies $\bigwedge_{x \in X} f(x) \leq \bigvee_{y \in Y} f(y)$.
We are interested in the following universal problem: a bounded distributive lattice $D$ together with a map $i \colon S \to D$ preserving $R$ such that for any bounded distributive lattice $L$ and any map $f \colon S \to L$ preserving $R$ there is a unique bounded lattice homomorphism $f' \colon D \to L$ such that $f' \circ i = f$.
One says that $D,\, i \colon S \to D$ is \emph{generated} by $S, R$.
Since the theory of bounded distributive lattices is equational, there is a solution to this universal problem.

\begin{theorem}[Fundamental theorem of entailment relations,  {\cite[Thm.~1]{CederquistCoquand1998}}] \label{t:fundamental-theorem-of-entailment-relations}
    Let $S$ be a set with an entailment relation $\vdash$. If $D,\, i \colon S \to D$ is the bounded distributive lattice generated by $S, \vdash$, then, for all finite subsets $X,Y$ of $S$, we have $X \vdash Y$ if and only if $\bigwedge_{x \in X} i(x) \leq \bigvee_{y \in Y} i(y)$.
\end{theorem}

The interesting implication in \cref{t:fundamental-theorem-of-entailment-relations} is the right-to-left one, which says that in $D$ not too many things collapse.

\begin{remark}\label{r:entailment-DNF}
    The fundamental theorem of entailment relations is obtained by constructing the bounded distributive lattice generated by $S, \vdash$ as the poset reflection of a preordered set whose underlying set is the finite powerset of the finite powerset of $S$ and whose preorder is explicitly given. The idea is that a finite set of finite subsets of $S$ expresses a bounded lattice combination of elements of $S$ in disjunctive normal form.
\end{remark}

\subsection{Entailment relations and Boolean algebras}

We can piggyback on \cref{t:fundamental-theorem-of-entailment-relations} to obtain a similar version for Boolean algebras.

Given a set $S$ with a binary relation $R$ on finite subsets of $S$ one says that a map $f \colon S \to A$ from $S$ to a Boolean algebra $A$ \emph{preserves} $R$ if $X \mathrel{R} Y$ implies $\bigwedge_{x \in X} f(x) \leq \bigvee_{y \in Y} f(y)$.
We are interested in the following universal problem: a Boolean algebra $A$ together with a map $i \colon S \to A$ preserving $R$ such that for any other map $f \colon S \to B$ preserving $R$ there is a unique Boolean homomorphism $f' \colon A \to B$ such that $f' \circ i = f$.
One says that $A,\, i \colon S \to A$ is \emph{generated} by $S, R$.
Since the theory of Boolean algebras is equational, there is a solution to this universal problem.

\begin{theorem}[Fundamental theorem of entailment relations, for Boolean algebras]\label{t:fundamental-theorem-of-entailment-relations-boolean-algebras}
    Let $S$ be a set with an entailment relation $\vdash$. If $A,\, i \colon S \to A$ is the Boolean algebra generated by $S, \vdash$, then, for all finite subsets $X, Y$ of $S$, we have $X \vdash Y$ if and only if $\bigwedge_{x \in X} i(x) \leq \bigvee_{y \in Y} i(y)$.
\end{theorem}

\begin{proof}
    The map $i \colon S \to A$ in the Boolean algebra generated by $S, \vdash$ is, up to isomorphism, the composite of the following two maps:
    \[
        S \xrightarrow{j} D \xrightarrow{f} B
    \]
    where $D, \,j \colon S \to D$ is the bounded distributive lattice generated by $S, \vdash$, and $B,\, D \xrightarrow{f} B$ is the free Boolean algebra over the bounded distributive lattice $D$, i.e., the Boolean envelope of $D$.
    It is known that $f$ is injective (and hence order-preserving and reflecting); this, for example, was derived from \cref{t:fundamental-theorem-of-entailment-relations} in \cite[Cor.~2]{CederquistCoquand1998}.
    Then the desired statement follows from \cref{t:fundamental-theorem-of-entailment-relations}.
\end{proof}

\begin{remark}\label{r:two-copies}
    For a proof that does not pass via the Boolean envelope, one can obtain the Boolean algebra $A,\, i \colon S \to A$ generated by $S, \vdash$ as the bounded distributive lattice generated by two copies $S_1$ and $S_2$ of $S$ (the first one to be thought of as the set of generators of the Boolean algebra, and the second one as the set of negations of generators), together with the entailment $\vdash'$ defined by setting, for finite subsets $X, Z \subseteq S_1$ and $Y, W \subseteq S_2$, $X \cup Y \vdash' Z \cup W$ if and only if $X \cup W \vdash Z \cup Y$. (Compare with \cite[Rem.~3]{CederquistCoquand1998}.)
    Then (by \cref{r:entailment-DNF}), the Boolean algebra generated by $S, \vdash$ can be obtained as the poset reflection of a certain preordered set whose underlying set consists of the finite powerset of the finite powerset of the disjoint union of two copies of $S$ and whose preorder is explicitly given.
\end{remark}

\makeatletter
\begingroup
\let\addcontentsline\@gobblethree
\section*{Acknowledgments}
We would like to thank the logic group at the University of Salerno for supporting a sequence of seminars in logic, concluded with a workshop on doctrines which gave us the opportunity to advance this collaboration.
Moreover, some important breakthroughs in our work occurred during a visit in November 2023 at the University of Luxembourg, for which we are grateful to Bruno Teheux.

We thank J\'er\'emie Marqu\`es for suggesting that we highlight the fact that the main result is a doctrinal version of Herbrand's theorem.

Finally, we are grateful to the referee for their valuable suggestions, which improved the quality and depth of the manuscript.

\section*{Funding}
The first author's research was funded by UK Research and Innovation (UKRI) under the UK government’s Horizon Europe funding guarantee (grant number EP/Y015029/1, Project ``DCPOS''). (The ``Horizon Europe guarantee'' scheme provides funding to researchers and innovators who were unable to receive their Horizon Europe funding---in this case, for a Marie Skłodowska-Curie Actions grant---while the UK was in the process of associating.) 
The first author's research was also funded by the Italian Ministry of University and Research through the PRIN project n.~20173WKCM5 \emph{Theory and applications of resource sensitive logics}.
Moreover, part of the final revision of the manuscript was carried out during the first author's affiliation with UCLouvain, supported by an FSR Incoming Postdoctoral Fellowship.

The second author's research was supported by the Luxembourg National Research Fund (FNR), ``COMOC'' Project (ref.~C21/IS/16101289). Moreover, part of the final revision of the manuscript was carried out during the second author's affiliation with the Laboratoire de Mathématiques (LAMA) of the Université Savoie Mont Blanc, supported by the SHINE program of the French National Research Agency (ANR) under the project ``GULI'' (Grandeurs et Unités pour les Langages Informatiques), grant number ANR-22-EXES-0017.

 \endgroup
 \makeatother

\bibliographystyle{plain}
\bibliography{Biblio}

\end{document}